\documentclass[leqno]{article}
\usepackage{CJK,CJKnumb,CJKulem,times,dsfont,ifthen,mathrsfs,latexsym,amsfonts,color}
\usepackage{amsmath,amsthm,makeidx,fontenc,amssymb,bm,graphicx,psfrag,listings,curves,extarrows,enumitem}
\usepackage{hyperref}
\usepackage{geometry}
\usepackage{indentfirst}
\usepackage{amssymb}
\makeatletter

\newcommand{\Rmnum}[1]{\expandafter\@slowromancap\romannumeral #1@}
\makeatother
\usepackage[showonlyrefs]{mathtools}  
\mathtoolsset{showonlyrefs=true}
\newtheorem{definition}{Definition}[section]
\newtheorem{theorem}{Theorem}[section]
\newtheorem{lemma}{Lemma}[section]

\newtheorem{proposition}{Proposition}[section]
\newtheorem{remark}{Remark}[section]

			\newcommand{\N}{{\mathbb N}}

			\newcommand{\R}{{\mathbb R}}
			\newcommand{\bean}{\begin{eqnarray*}}
				\newcommand{\eean}{\end{eqnarray*}}

			\newcommand{\abs}[1]{\left\lvert#1\right\rvert}
			
\setitemize{itemindent=38pt,leftmargin=0pt,itemsep=-0.4ex,listparindent=26pt,partopsep=0pt,parsep=0.5ex,topsep=-0.25ex}

\numberwithin{equation}{section}
		
\begin{document}
\theoremstyle{plain}			
\title{\bf On the stability of critical points of the Hardy-Littlewood-Sobolev inequality\thanks{Supported by NSFC(No.12171265,11771234,12026227).
E-mail addresses: liuk20@mails.tsinghua.edu.cn (K. Liu), zhangqian9115@mail.tsinghua.edu.cn (Q. Zhang), zou-wm@mail.tsinghua.edu.cn (W. M. Zou)} }		
\author{
{\bf Kuan Liu, Qian Zhang and Wenming Zou }\\
\date{}
\footnotesize
\it  Department of Mathematical Sciences, Tsinghua University, Beijing 100084, China.\\ }							
\maketitle
				
\begin{center}
\begin{minipage}{120mm}
\begin{center}{\bf Abstract}\end{center}		
This paper is concerned with the quantitative stability of critical points of the Hardy-Littlewood-Sobolev inequality. Namely, we give quantitative estimates for the  Choquard equation:
$$-\Delta u=(I_{\mu}\ast|u|^{2_\mu^*}) u^{2_\mu^*-1}\ \ \text{in}\ \ \R^N,$$
where $u>0,\ N\geq 3,\ \mu\in(0,N)$, $I_{\mu}$  is the Riesz potential and $2_\mu^* \coloneqq \frac{2N-\mu}{N-2}$ is the upper Hardy-Littlewood-Sobolev critical exponent. The Struwe's decomposition (see M. Struwe: Math Z.,1984) showed that the equation $\Delta u + u^{\frac{N+2}{N-2 }}=0$ has phenomenon of ``stable up to bubbling'', that is, if $u\geq0$ and $\|\Delta u+u^{\frac{N+2}{N-2}}\|_{(\mathcal{D}^{1,2})^{-1}}$ approaches zero, then $d(u)$ goes to zero, where $d(u)$ denotes the $\mathcal{D}^{1,2}(\R^N)$-distance between $u$ and the set of all sums of Talenti bubbles. Ciraolo, F{}igalli and Maggi (Int. Math. Res. Not.,2017) obtained the f{}irst quantitative version of Struwe's decomposition with single bubble in all dimensions $N\geq 3$, i.e, $\displaystyle d(u)\leq C\|\Delta u+u^{\frac{N+2}{N-2}}\|_{L^{\frac{2N}{N+2}}}.$ For multiple bubbles, F{}igalli and Glaudo (Arch. Rational Mech. Anal., 2020) obtained quantitative estimates depending  on the dimension, namely $$ d(u)\leq C\|\Delta u+u^{\frac{N+2}{N-2}}\|_{(\mathcal{D}^{1,2})^{-1}},  \hbox{ where } 3\leq N\leq 5,$$
which is invalid as $N\geq 6.$
 \vskip0.1in
 \quad In this paper, we prove the quantitative estimate of the Hardy-Littlewood-Sobolev inequality, we get  $$d(u)\leq C\|\Delta u +(I_{\mu}\ast|u|^{2_\mu^*})|u|^{2_\mu^*-2}u\|_{(\mathcal{D}^{1,2})^{-1}},  \hbox{ when }  N=3  \hbox{ and } 5/2< \mu<3.$$

\vskip0.13in
{\bf 2020 Mathematics Subject Classif{}ication: 35A23, 26D10, 35B35, 35J20, 35J60}
\vskip0.1in					
\end{minipage}
\end{center}
\vskip0.18in
\section{Introduction}
The classical   Sobolev inequality (\cite{Aubin,Talenti}) states that for any $N\ge 2$ and $p\in(1,N)$,  there exists a sharp constant $S(N,p)$ such that
\begin{equation}\label{eq:1.1}
    \|\nabla u\|_{L^p}\geq S(N,p) \|u\|_{L^{p^{\ast}}},\quad \forall u\in \dot{W}^{1,p}(\R^N),
\end{equation}
 where $p^*\coloneqq\frac{Np}{N-p}$, and the functional space
$\dot{W}^{1,p}(\R^N)\coloneqq \{u\in L^{p^{\ast}}(\R^N):\nabla u\in L^{p}(\R^N)\}$ equipped with the norm  $\|\nabla u\|_{L^{p}(\R^N)}$. The extremal functions constitute the $N+2$ dimensional manifold:
$$ \mathcal{M}\coloneqq \Big\{ v_{a,b,x_{0}}\coloneqq \frac{a}{(1+b|x-x_{0}|^{\frac{p}{p-1}})^{\frac{N-p}{p}}}: a\in \R\setminus\{0\},\  b>0,\ x_{0}\in \mathbb{R}^N\Big\}.$$
In particular, when $N\geq 3$ and $p=2$, then there exists a sharp constant $S=S(N)$ such that
\begin{equation}\label{eq:1.2}
\|\nabla u\|_{L^2}\geq S\|u\|_{L^{2^{\ast}}}
\end{equation}
for any $u\in\mathcal{D}^{1,2}(\R^N)\coloneqq \dot{W}^{1,2}(\R^N)$, where $2^*\coloneqq\frac{2N}{N-2}$ and the sharp constant $S$ is characterised by
\begin{equation}\label{eq:1.3}
    S=\inf_{u\in\mathcal{D}^{1,2}
(\R^N)\setminus\{0\}}\frac{\|\nabla u\|_{L^{2}}}{\|u\|_{L^{2^{\ast}}}},
\end{equation}
and the extremal functions of \eqref{eq:1.3} are
\begin{equation}\label{eq:1.4}
  U[z,\lambda](x)\coloneqq(N(N-2 ))^{\frac{N-2 }4}
  \frac{\lambda^{\frac{N-2 }2}}{(1+\lambda^2\abs{x-z}^2)^{\frac{N-2 }2}} ,\quad  \lambda>0,\  z\in \R^N,
  \end{equation}
which is called Aubin-Talenti bubbles or Talenti bubbles in literature.  It is well-known that the Euler-Lagrange equation
of \eqref{eq:1.3}, i.e., critical points of the Sobolev inequality \eqref{eq:1.2}, up to scaling, is given by
\begin{equation}\label{eq:1.5}
  -\Delta u=|u|^{2^*-2}u\ \ \text{in}\ \ \R^N.
\end{equation}
According to Gidas, Ni and Nirenberg \cite{sun18}, Caf{}farelli, Gidas and Spruck \cite{sun6}, Chen, Li \cite{chenduke}, it is
known that all the positive solutions are Talenti bubbles $U[z,\lambda]$, where $z\in\mathbb{R}^{N}$, $\lambda>0$.
\vskip0.08in
Since 1980s, there are growing interests in the stability of the  geometric and functional inequality, such as \eqref{eq:1.2}. There are mainly two directions about this topic, one is the geometric and functional inequality setting and the other one  is the critical point  setting. In the geometric and functional inequality setting, one is concerned with whether the def{}icit functional, which is def{}ined as the discrepancy of two sides in the given inequality, can control the appropriate distance between the extremal functions. For example, def{}ine the $p$-Sobolev  def{}icit of the inequality \eqref{eq:1.1}:
$$\delta (u)\coloneqq \|\nabla u\|_{L^p}- S(N,p) \|u\|_{L^{p^{\ast}}},
$$
Lions \cite{lion-compactness} proved that if $\delta(u)\to 0$, then $u$ is close to the manifold $\mathcal{M}$ in the sense of $\mathcal{D}^{1,2}$-distance. Br\'ezis and Lieb  in \cite{Brezis-Lieb} f{}irstly raised    such     questions  about the  quantitative stability for the classical Sobolev inequality \eqref{eq:1.1}, i.e., whether the functional $\delta(u)$ can be bounded from below by the $\mathcal{D}^{1,2}$-distance between  $u$  and  $\mathcal{M}$. Later, in the case $p=2$, Bianchi and Egnell \cite{bianchi1991} gave an optimal
quantitative estimate near the minimizers, i.e.,
$$ \inf_{z\in\R^N,\lambda>0,\alpha\in\R}\|\nabla(u-\alpha U[z,\lambda])\|_{L^2}^2\leq C(N)
(\|\nabla u\|_{L^2}^2-S^2\|u\|^2_{L^{2^*}}),$$
where $S$ is the constant in \eqref{eq:1.2}. This result is sharp in two aspects: the f{}irst one is that the $\mathcal{D}^{1,2}(\R^N)$-distance in the left hand side is the strongest, and the second one is that the quadratic exponent in the left hand side  is sharp, which cannot be replaced by a smaller one. After that, whether such result is true for $p\neq 2$ is an open question for a long time. After almost twenty years later, Cianchi, Fusco, Maggi and Pratelli \cite{cianchi-fusco-maggi-pratelli} extend this result to general $p\in (1,N)$ by using  symmertrization argument and optimal transport theory,  they proved that the $p$-Sobolev def{}icit can control some weaker distance, which uses $L^{p^{\ast}}$-norm of functions instead of $L^{p}$-norm of gradients, i.e.,
 $$
\inf_{ a\in \R\setminus\{0\}, b>0, x_{0}\in\mathbb{R}^N } \Big(\frac{\|u-v_{a,b,x_{0}}\|_{L^{p^{\ast}}}}{\|u\|_{L^{p^{ \ast}}}}\Big)^{\zeta}\leq C \ \frac{\!\!\!\delta(u)}{\ \ \|u\|_{L^{p^{\ast}}}},
 $$
where the exponent $\zeta=C_{CFMP}\coloneqq p^{\ast}(3+4p-\frac{3p+1}{N})$, which is most likely not sharp. However the strongest distance is
\begin{equation}\label{eq:1.6}
\inf_{ a\in \R\setminus\{0\}, b>0, x_{0}\in\mathbb{R}^N } \Big(\frac{\|\nabla u-\nabla v_{a,b,x_{0}}\|_{L^{p}}}{\|\nabla u\|_{L^{p}}}\Big)^{\zeta},
\end{equation}
thus the remained open problem is whether the quantitative stability result valid for the distance \eqref{eq:1.6}. Along this direction, Figalli and Neumayer \cite{figalli-Neumayer} f{}irstly  answered  this question in the case $p\geq 2$, they showed that there exists $C=C(N,p)$ such that
\begin{equation}\label{eq:1.7}
   \inf_{ a\in \R\setminus\{0\}, b>0, x_{0}\in\mathbb{R}^N } \Big(\ \frac{\|\nabla u-\nabla v_{a,b,x_{0}}\|_{L^{p}}}{\|\nabla u\|_{L^{p}}}\Big)^{\zeta}\leq C \ \frac{\!\!\!\delta(u)}{\ \ \|u\|_{L^{p^{\ast}}}},
\end{equation}
where $\zeta = p\,C_{CFMP}$. Later, by combining Clarkson's inequalities and the result of \cite{cianchi-fusco-maggi-pratelli}, Neumayer \cite{neumayer3} extended the above quantitative stability estimate \eqref{eq:1.7} to full range of $p\in(1,N)$, where $\zeta = \frac{p}{p-1} C_{CFMP}$ if $p\in(1,2)$, and $\zeta = p\,C_{CFMP}$ if $p\in[2,N)$. However since her results are based on the result of \cite{cianchi-fusco-maggi-pratelli}, so the exponent is not sharp.
Recently, Figalli and Zhang \cite{figalli-Zhang1} gave a f{}inal answer about this sharp stability exponent, they proved \eqref{eq:1.7} in the full range of $p\in(1,N)$ with sharp exponent $\zeta=\max\{2,p\}$. The result about the case $p=1$ was shown by Figalli, Maggi and Pratelli in \cite{figalli-maggi-Pratelli-2013}.
\vskip0.18in
Many results have been obtained in the geometric and functional inequality setting, for example, isoperimetric inequalities \cite{fusco-maggi-pratelli,figalli-Maggi-Pratelli-iso,maggi,neumayer2,cicalese-Leonardi,neumayer2,figalli-fusco,howard}, Brunn-Minkowski inequality \cite{figalli-maggi-Pratelli-BM,figalli-Jerison}, Sobolev-type inequality \cite{bartsch,chen-frank,dolbeault-Toscani,seuffert,Frank}, Hardy-Littlewood-Sobolev inequality \cite{figalli-Carlen,carlen1}, Gagliardo-Nirenberg-Sobolev inequalities \cite{figalli-Carlen,dolbeault-Toscani,ruffini,bonforte1,bonforte2}, Wulf{}f inequality  \cite{neumayer1,figalli-Zhang2}. As for the stability with  Yamabe metrics, we refer the readers to \cite{Engelstein} and references therein.
\vskip0.18in
Another remarkable and more challenging topic  is the critical point setting, namely for the corresponding Euler-Lagrange equation of a given inequality, if $u$ is an
approximate (in the sense of Palais-Smale sequence) solution of the equation, then whether $u$ is close (in some sense) to a manifold spanned by special solutions (e.g. Talenti bubbles). For the corresponding equation \eqref{eq:1.5} of the classical Sobolev inequality \eqref{eq:1.2}, a pioneering work by Struwe \cite{struwe1984}, which usually called as global compactness result, showed that  this is always the case in the qualitative sense, at least for non-negative functions. He showed that if a non-negative function sequence $\{u_k\}_{k\in\N}\subseteq \mathcal{D}^{1,2}(\mathbb{R}^N)$ satisf{}ies Dirichlet energy bound condition (which controls the number of multi-bubbles) and the Fr\'echlet derivatives tend to zero in dual space, i.e., there exists $\nu\in\N^{+}\coloneqq\{1,2,\cdots\}$, such that
$$\Big(\nu-\frac12\Big)S^N\le \int_{\R^N}\abs{\nabla u_k}^2\,dx \le \Big(\nu+\frac12\Big)S^N,$$
and $\big\|\Delta u_{k} + u_k^{2^{\ast}-1}\big\|_{(\mathcal{D}^{1,2})^{-1}} \to 0$ as $k\to\infty,$ then $\{u_{k}\}_{k\in\N}$ is  convergent to  a sum of Talenti bubbles in the energy norm, i.e., there exist a sequence $(z^{(k)}_1,\dots,z^{(k)}_\nu)_{k\in\N}$ of $\nu$-tuples of points in $\R^N$ and a sequence $(\lambda^{(k)}_1,\dots,\lambda^{(k)}_\nu)_{k\in\N}$ of $\nu$-tuples of positive real numbers
  such that
$$ \Big\|\nabla\big(u_k-\sum_{i=1}^\nu U[z_i^{(k)},\lambda_i^{(k)}]\big)
\Big\|_{L^2}\to 0\text{\quad as \quad} k\to\infty.$$ Equivalently, he gave a decomposition of  the Palais-Smale sequence. Usually, this result is  also called as the Struwe's decomposition in literature. Furthermore, Bahri-Coron \cite{bahri} established a more precise  asymptotic behavior of the above parameters $\{z_{i}^{(k)},\lambda_{i}^{(k)}\}$, they  showed
$$   \frac{\lambda^{(k)}_i}{\lambda^{(k)}_j}+\frac{\lambda^{(k)}_j}{\lambda^{(k)}_i}+\lambda^{(k)}_i\lambda^{(k)}_j |z^{(k)}_{i}-z^{(k)}_{j}|^{2}\to\infty
\text{\quad as \quad} k\to\infty \text{\quad for all \quad} i\neq j.$$

F{}igalli and his collaborators f{}irstly gave  a quantitative version of Struwe's decomposition \cite{struwe1984}.
Along this direction, Ciraolo, F{}igalli and Maggi
\cite{9}
obtained the f{}irst quantitative estimate $d(u)\lesssim\big\|\Delta u+u^{2^{\ast}-2}u\big\|_{L^{\frac{2N}{N+2}} }$
for all $N\geq3$ when $\nu=1$, i.e., there is only one bubble, here $d(u)$ denotes the $\mathcal{D}^{1,2}(\mathbb R^N)$-distance of $u$ from this bubble. Later F{}igalli and Glaudo \cite{figalli} established the linear estimate
\begin{equation}\label{eq:1.8}
  \Big\|\nabla u-\sum_{i=1}^{\nu}\nabla U_i\Big\|_{L^2}
\le C(N,\nu)\big\|\Delta u+\abs{u}^{2^*-2}u\big\|_{(\mathcal{D}^{1,2})^{-1}}
\end{equation}
for $3\leq N\leq 5$ and f{}inite number ($\nu\geq2$) of bubbles. Moreover, they constructed counter-examples for the case
$N\geq6$, $\nu\geq 2$, showing that the linear estimate \eqref{eq:1.8} does not hold.
 However,  recently, Deng, Sun and Wei in \cite{dsw} using f{}inite dimensional reduction method showed the right hand side of \eqref{eq:1.8} can be a nonlinear term, and deduced a sharp estimate in the case $N\geq6$.
 More precisely, they proved that
 $$
\big\|\nabla u-\sum_{i=1}^{\nu}\nabla U_i\big\|_{L^2}
\lesssim
\begin{cases}
\ \Upsilon(u) \big|\textup{log }\Upsilon(u)\big|^{\frac{1}{2}},\ \ &N=6,\\
\ \Upsilon(u)^{\frac{N+2}{2(N-2)}},\ \ &N\geq 7,
\end{cases}
 $$
where $\Upsilon(u)\coloneqq \big\|\Delta u+\abs{u}^{2^*-2}u\big\|_{(\mathcal{D}^{1,2})^{-1}}$.
Moreover, the exponent of the right hand side is sharp in the sense that the above estimate cannot hold for a bigger exponent. Recently, Wei and Wu \cite{Wu-Wei} established the stability for a special case of the  Caffarelli-Kohn-Nirenberg inequality in the critical point setting.

\vskip0.18in

Given these results, it is natural to ask what will happen   to the well known Hardy-Littlewood-Sobolev (HLS for short) inequality. Recall that the classical  Hardy-Littlewood-Sobolev inequality (\cite{hls9,hls10,hls15,021,hls18}) states that for $N\geq 1$, $\mu\in(0,N)$, there exists a sharp constant $C(N,\mu,s,t)>0$ such that
  \begin{equation}\label{eq:1.9}
   \Big|\int_{\mathbb{R}^N} (I_{\mu}\ast f) \,g\,\,dx\Big|\le C(N,\mu,s,t)\,
  ||f||_{L^s}||g||_{L^t},  \quad (f,g)\in L^s(\mathbb{R}^N)\times L^t(\mathbb{R}^N),
  \end{equation}
where $s,t\in(1,+\infty)$ satisf{}ies $\frac 1s +\frac 1t+ \frac{\mu}{N}=2$, and $I_{\mu}:\mathbb R^{N}\backslash\{0\}\rightarrow\mathbb R$ is the Riesz potential def{}ined by $$I_{\mu}(x)\coloneqq\frac{\mathcal{K}_{\mu}}{|x|^{\mu}}
 \ \ \hbox{with}\ \
 \ \mathcal{K}_{\mu}=\frac{\Gamma(\frac{\mu}{2})}{2^{N-\mu}\pi^{\frac{N}{2}}\Gamma(\frac{N-\mu}{2})},$$
 here $\Gamma$ denotes the classical Gamma function and the notation $\ast$  represents the convolution operator in the Euclidean space
  $\mathbb R^{N}$, i.e., $(I_{\mu}\ast f)(x)\coloneqq \int_{\mathbb R^{N}} I_{\mu}(x-y)f(y) \,dy$.
 \vskip0.08in
Lieb \cite{hls15} proved the existence of the extremal functions to inequality \eqref{eq:1.9} with sharp constant and computed the sharp constant in the case of $s=t$. Later, by using the rearrangement technique, Frank and Lieb \cite{dou11,dou12} reproved the sharp constant of the inequality \eqref{eq:1.9} for the case $s=t=\frac{2N}{2N-\mu}$, they obtained that
$$
 C(N,\mu,s,t)=C_{N,\mu}=\pi ^{\frac \mu 2}
 \frac{\Gamma (\frac{N- \mu}{2})}{\Gamma (N-\frac \mu 2)}
\left(\frac{\Gamma (\frac N2)}{\Gamma (N)}\right)^{-1+\frac \mu N}.
 $$
In particular, combining the inequality \eqref{eq:1.9} and Sobolev inequality \eqref{eq:1.2}, for any $N\ge 3$ and $u\in \mathcal{D}^{1,2}(\R^N)$, it holds
  \begin{equation}\label{eq:1.10}
  S_{eq:1.11}\bigg(\int_{\mathbb{R}^N}(I_{\mu}\ast
  |u|^{2_{\mu}^{\ast}})\,|u|^{2_{\mu}^{\ast}}\,dx\bigg)
  ^{\frac{1}{2_{\mu}^{\ast}}}\leq\|\nabla u\|_{L^2}^2,
  \end{equation}
where  $2_\mu^*\coloneqq\frac{2N-\mu}{N-2}$
 is the upper Hardy-Littlewood-Sobolev critical exponent, and $S_{eq:1.11}=S_{eq:1.11}(N,\mu)$ denotes the sharp constant def{}ined by
\begin{equation}\label{eq:1.11}
S_{eq:1.11}\coloneqq\inf_{u\in\mathcal{D}^{1,2}
(\R^N)\setminus\{0\}}\frac{\int_{\mathbb R^N}|\nabla u|^2\,dx}{\left(\int_{\mathbb{R}^N}(I_{\mu}\ast
  |u|^{2_{\mu}^{\ast}})\,|u|^{2_{\mu}^{\ast}}\,dx\right)
  ^{\frac{1}{2_{\mu}^{\ast}}}}.
  \end{equation}
In this paper, we also called \eqref{eq:1.10} as the  Hardy-Littlewood-Sobolev (HLS for short) inequality.
\vskip0.18in
It is well-known that the corresponding Euler-Lagrange equation of \eqref{eq:1.11}, i.e., critical points of the Hardy-Littlewood-Sobolev inequality \eqref{eq:1.10},  is given by
\begin{equation}\label{eq:1.12}
-\Delta u=  (I_{\mu}\ast|u|^{2_\mu^*} )|u|^{2_\mu^*-2}u\ \ \text{in}\ \ \R^N,
\end{equation}
where $N\geq3$ and $0<\mu<N$.
The associated energy functional of equation \eqref{eq:1.12} is
$$
    J(u)\coloneqq \frac12\int_{\R^N}\abs{\nabla u}^2\,dx-\frac{1}{2\cdot 2_{\mu}^{\ast}}\int_{\R^N}(I_{\mu}\ast|u|^{2_{\mu}^{\ast}})
|u|^{2_{\mu}^{\ast}}\,dx.
 $$
In recent years, there has been extensive studies about the equation:
\begin{equation}\label{eq:1.13}
-\Delta u+u=\left(I_{\mu}\ast
|u|^{q}\right)|u|^{q-2}u \ \
\hbox{in}\ \
\mathbb R^N,
\end{equation}
which is usually called as the nonlinear Choquard or Choquard-Pekar equation. Equation \eqref{eq:1.13} has several physical motivations, if $N=3,\ q=2$ and $\mu=1$, the problem
\begin{equation}\label{eq1.2}
-\Delta u+u=\left(I_{1}\ast|u|^{2}\right)u \ \
\hbox{in}\ \
\mathbb R^3
\end{equation}
already appeared as early as in 1954, describing the quantum mechanics of a polaron at rest,see
\cite{pekar}. In 1976, by using such an equation, Choquard described  an electron trapped in its own hole and  enriched the Hartree-Fock theory,  see \cite{lieb}. For more interesting existence and qualitative properties of solutions to \eqref{eq:1.13}, we refer the readers to
\cite{li,moroz,morozs,ms,mos} and references therein. Furthermore, by using the moving plane methods developed in Chen, Li and Ou \cite{gao8,gao9}, Guo, Hu, Peng and Shuai \cite{peng},  Du and Yang \cite{gao17} proved independently that when $0<\mu<N$ if $N=3 $ or $4$ and $0<\mu\leq4$ if $N\geq5$, then any positive solution of equation \eqref{eq:1.12} must assume the form of
\begin{equation}\label{eq:1.14}
\tilde{U}[z,\lambda](x)=c(\mu)U[z,\lambda](x),
\end{equation}
where $z\in\mathbb{R}^{N}$, $\lambda>0$ and
$$c(\mu)=\frac{1}{
\bigl(C_{N,\mu}\mathcal{K}_\mu S^{\frac{N-\mu}{2}}\bigr)^{\frac{N-2}{4+2(N-\mu)}}}.$$
It is worth mentioning  that the only dif{}ference  between  \eqref{eq:1.14} and \eqref{eq:1.4} is the constant $c(\mu)$. Therefore, we also call  $\tilde{U}$ as the Talenti bubbles in the current paper.
\vskip0.08in
To give a more quantitative description of the above parameters $\{z_{i}^{(k)},\lambda_{i}^{(k)}\}$, we will deal with  the so-called weakly-interacting family of Talenti bubbles which were def{}ined in \cite{figalli}. Roughly speaking, a family
$\{U[z_i,\lambda_i]\}_{i=1}^{\nu}$ is called weakly-interacting
if either the centers $z_i$ of the Talenti bubbles are very far from
the others, or their scaling factors $\lambda_i$ have very dif{}ferent magnitude. More precisely,

\begin{definition}\cite[Interaction of Talenti bubbles]{figalli}\label{def-delta-intracting} {\it Let $U_1\coloneqq U[z_1, \lambda_1],\dots,U_\nu\coloneqq U[z_\nu,\lambda_\nu]$  be a family of Talenti bubbles from \eqref{eq:1.4}.
The family $\{U_{i} \}_{i=1}^{\nu}$ is called $\delta$-interacting for some $\delta>0$ if{}f
\begin{equation}\label{eq:1.15}
\max_{i\neq j}\,\min\left(\frac{\lambda_i}{\lambda_j},\frac{\lambda_j}{\lambda_i},
\frac{1}{\lambda_i\lambda_j\abs{z_i-z_j}^2}\right)\le\delta.
\end{equation}
 For $\alpha_1,\dots,\alpha_\nu>0$, the family $\{(\alpha_{i},U_{i}) \}_{i=1}^{\nu}$  is called $\delta$-interacting if \eqref{eq:1.15}
holds and moreover
\begin{equation}\label{eq:1.16}
 \max_{1\le i\le \nu} \abs{\alpha_i-1} \le \delta.
\end{equation}
Similarly, let $U_1\coloneqq\tilde{U}[z_1, \lambda_1],\dots,U_\nu\coloneqq\tilde{U}[z_\nu,\lambda_\nu]$ be a family of Talenti bubbles from \eqref{eq:1.14}.
The family $\{U_{i} \}_{i=1}^{\nu}$ is called $\delta$-interacting for some $\delta>0$ if \eqref{eq:1.15} is satisf{}ied. For $\alpha_1,\dots,\alpha_\nu>0$, the family $\{(\alpha_{i},U_{i}) \}_{i=1}^{\nu}$  is called $\delta$-interacting if \eqref{eq:1.15} and \eqref{eq:1.16} hold.}
\end{definition}

\vskip0.08in
 It is worth noting that  the Struwe's type global compactness result has been established by
  Guo, Hu, Peng and Shuai in \cite{peng}, they proved that when $0<\mu<N$ if $N=3 $ or $4$ and $0<\mu\leq4$ if $N\geq5$, for a non-negative function sequence $\{u_k\}_{k\in\N}$,  if $J(u_{k}) \to c $ and $$\|J^{\prime}(u_{k})\|_{(\mathcal{D}^{1,2})^{-1}}=\|\Delta u_{k} +(I_\mu\ast |u_{k}|^{2_\mu^*})|u_{k}|^{2_{\mu}^{\ast}-2}u_{k} \|_{(\mathcal{D}^{1,2})^{-1}}\to 0$$ as $k\to+\infty$ (equivalently $\{u_k\}_{k\in\N}$ is a Palais-Smale sequence of $J$ at level $c$), then there exist  an integer $\nu\geq 1$, a sequence $(z^{(k)}_1,\dots,z^{(k)}_\nu)_{k\in\N}$ of $\nu$-tuples of points in $\R^N$ and a sequence $(\lambda^{(k)}_1,\dots,\lambda^{(k)}_\nu)_{k\in\N}$ of $\nu$-tuples of positive real numbers such that, up to a subsequences, there holds
$$
      \Big\|\,\nabla\big( u_{k}-\sum_{i=1}^{\nu}
      \tilde{U}[z_{i}^{(k)}, \lambda_{i}^{(k)}]\big)\,\Big\|_{L^{2}}\to0\text{\quad as \quad} k\to +\infty,
 $$
where $\tilde{U}[z_{i}^{(k)}, \lambda_{i}^{(k)}]$ has the form  of \eqref{eq:1.14}.
 \vskip0.108in

 A natural problem is that whether the analogue estimate as  \eqref{eq:1.8} is valid in the setting of the Hardy-Littlewood-Sobolev inequality. More precisely, our problem is that if $N\ge 3$, $\mu\in(0,N)$ and $\nu\ge 1$ is a positive integer, does there exist a constant $C=C(N,\mu,\nu)>0$ such that the linear bound
 $$
    \inf_{\substack{ (z'_i)_{1\le i\le\nu}\subseteq\R^N\\
(\lambda_i')_{1\le i\le\nu}\subseteq (0,\infty)}}
      \Big\|\nabla \big(u - \sum_{i=1}^{\nu}\tilde{U}[z_i',\lambda_i']\big)\Big\|_{L^2}
      \le C \big\|\Delta u +(I_\mu\ast |u|^{2_\mu^*})|u|^{2_{\mu}^{\ast}-2}u\big\|_{(\mathcal{D}^{1,2})^{-1}}
 $$
  holds for $\delta$-interacting family of Talenti bubbles from \eqref{eq:1.14}? We have the following result.
  \begin{theorem}\label{thm:main_close}
   Let  $N=3$, $5/2<\mu<3$ and $\nu\in\N^{+}$, there exist a small constant $\delta=\delta(N,\mu,\nu)>0$ and
    a large constant $C=C(N,\mu,\nu)>0$ such that the following statement holds.
    Let $u\in \mathcal{D}^{1,2}(\R^N)$ be a function such that
    \begin{equation}\label{eq:1.17}
     \Big\|u-\sum_{i=1}^{\nu}\hat{U}_i\Big\|_{\mathcal{D}^{1,2}(\R^N)} \le \delta,
     \end{equation} where $\{\hat U_{i} \}_{i=1}^{\nu}$ is a $\delta$-interacting family of Talenti bubbles from \eqref{eq:1.14}. Then there exist $\nu$ Talenti bubbles $\tilde{U}_1,\tilde{U}_2,\dots, \tilde{U}_\nu$ from \eqref{eq:1.14}, such that
     \begin{equation}\label{eq:1.18}
         \Big\|\nabla u-\sum_{i=1}^{\nu}\nabla \tilde{U}_i\Big\|_{L^2}
  \le C\big\|\Delta u+(I_{\mu}\ast|u|^{2_\mu^*})|u|^{2_\mu^*-2}u
  \big\|_{(\mathcal{D}^{1,2})^{-1}}.
     \end{equation}
 Furthermore, for any $i\not=j$, the interaction between the bubbles can be
 estimated as
 \begin{equation}\label{eq:1.19}
 \int_{\R^N}
 \tilde{U}_i^{p}\tilde{U}_j\,dx
 \le C\big\|\Delta u + (I_\mu\ast |u|^{2_\mu^*})|u|^{2_\mu^*-2}u\big\|_{(\mathcal{D}^{1,2})^{-1}}.
 \end{equation}
 \end{theorem}
\vskip0.08in
\begin{remark}\label{rmk1}
Theorem \ref{thm:main_close} gives a positive answer to our problem when $N=3$ and $\mu\in(5/2,3)$, and note that the convergence property (see \cite{landkof,gao17}) of the Riesz potential showed that $I_{\mu}\to \delta_{x}$ as $\mu\to N$, therefore the equation \eqref{eq:1.12} goes back to the equation \eqref{eq:1.5} as $\mu\to N$. Hence, our results extend (in the limit sense) the quantitative stability results of the classical Sobolev inequality in Figalli and Glaudo \cite{figalli} to the Hardy-Littlewood-Sobolev inequality in dimension $N=3$.
\end{remark}

\vskip0.08in
\begin{remark}\label{rmk2}
We cannot prove that Theorem \ref{thm:main_close} holds in other cases, such as $N=3$, $\mu\in(0,5/2]$ or $N\geq 4$, $\mu\in(0,N)$. It is worth noting that the classif{}ication result about the positive solutions of \eqref{eq:1.12} and the Struwe's type result have not yet been established in the case $N\geq 5$ and $\mu>4$.  In our proof, the setting of $\mathcal{D}^{1,2}$-distance, the using of H\"older inequality and the  estimation  of $\mathcal{H}_{1}\sim\mathcal{H}_{16}$, $\mathcal{H}_{19}$, $\mathcal{H}_{17}$ and $\mathcal{H}_{18}$ indicate the reasonable assumptions why the condition $N=3, \mu\in(5/2,3)$ are needed. However,  although the techniques developed in this paper do not provide any positive result in those cases,  we believe that the similar conclusions  as  Theorem \ref{thm:main_close} will hold. But we do not address this question here.
\end{remark}
\vskip0.18in

The organization of the paper is as follows.
After the notation and preliminaries in Section \ref{sec-notation}. In Section \ref{sec-intraction}, we prove the interaction
integral estimate, which play  key roles in the Section \ref{sec-quantitative}.
F{}inally, Section  \ref{sec-quantitative} is devoted
to the proof of Theorem \ref{thm:main_close}.

\newpage

\section{Notation and preliminaries}\label{sec-notation}
From now on, we use the following notations and conventions.
\begin{itemize}
    \item  We denote by $N$ the dimension of Euclidean space and $\nu$ the number of bubbles.
    \item For $R>0$, the notation $B(z,R)$ denotes the ball in $\mathbb{R}^{N}$ with center $z$ and radius $R$.
    \item For a nonempty subset $M\subseteq \mathbb{R}^{N}$, the characteristic function $\mathds{1}_{M}$ def{}ined by
$$
    \mathds{1}_{M}(x)\coloneqq\begin{cases}
        1 ,&\ \ x\in M ,\\
        0 ,&\ \ x\notin M.
    \end{cases}
 $$
    \item   For a real-valued function $f$, def{}ine $f^{\pm}\coloneqq\max\{\pm f,0\}$,  then $0\leq f^{\pm}\leq |f|$ and $f=f^{+}-f^{-}$.
    \item  For functions $f$ and $g$, the set $\{f\geq g\}\coloneqq\{x\in \mathbb{R}^{N}: f(x)\geq g(x)\}$.
    \item  Let $\|u\|_{L^r}$ denote  the usual norm of $L^r(\mathbb R^N)\ (r\geq1)$ and the working space is $$\mathcal{D}^{1,2}(\R^N)\coloneqq \{u\in L^{2^{\ast}}(\R^N):\nabla u\in L^{2}(\R^N)\}$$ equipped with norm $\|\nabla u\|_{L^2}$.
    \item   We denote $2^*\coloneqq\frac{2N}{N-2}$ and $2_\mu^*\coloneqq\frac{2N-\mu}{N-2}$, where $\mu\in(0,N)$ is a parameter.
    \item   From now on, f{}ix notation $p\coloneqq2^*-1=\frac{N+2}{N-2 }$ and $\tilde{p}\coloneqq2_\mu^*-1=\frac{N+2-\mu}{N-2 }$.
    \item   Given $q \in(1,\infty)$, we denote by $q^\prime\coloneqq\frac{q}{q-1}$ the H\"older conjugate of $q$.
    \item   We write that $a \lesssim b$ (or $a\gtrsim b$) if $a \leq cb$ (or $ca \geq b$), where $c$ is a positive constant depends only on the parameters $N$, $\mu$ and  $\nu$. Also, we say that $a\approx b$ if $a\lesssim b$ and $a \gtrsim b$. We write $a \lesssim_{s_{1},\cdots,s_{t}} b$ (or $a \gtrsim_{s_{1},\cdots,s_{t}} b$ or $a \approx_{s_{1},\cdots,s_{t}} b$) if in addition the constant also depends on the  parameters $s_{1},\cdots,s_{t}$.
    \item  We denote with $o(1)$ any expression that goes to zero when the parameter $\delta$ goes to zero; let $o(E)$ denote the quantity  that, when divided
  by the quantity $E$, is $o(1)$ as $\delta \to 0$. For simplicity we will not write out explicitly $\delta \to 0$  and  use  $o(1)$ or $o(E)$ directly.
    \item    The notation $\{f\}_{g}$ means that the term ``$f$'' appears only as the condition ``$g$'' is satisf{}ied.
\end{itemize}

\vskip0.08in
Setting $\tilde{U}=\tilde{U}[z,\lambda]$, by the classif{}ication results in \cite{peng,gao17}, the Talenti bubbles satisfy
\begin{equation}\label{eq:2.1}
-\Delta\tilde{U}
=(I_{\mu}\ast\tilde{U}^{\tilde{p}+1})\tilde{U}^{\tilde p}.
\end{equation}
Moreover, according to \eqref{eq:1.14} and the fact that $U$ solve \eqref{eq:1.5}, we know
$-\Delta\tilde{U}=c(\mu)(-\Delta U)
 =c(\mu) U^{p}.$
Combining the  above two equations, we have the signif{}icant pointwise relation:
\begin{equation}\label{eq:2.2}
    (I_{\mu}\ast\tilde{U}^{\tilde{p}+1})\tilde{U}^{\tilde{p}}=c(\mu) U^{p}= c(\mu)^{1-p}\tilde{U}^{p}  \textit{\quad in \ \  $\R^{N}$},
\end{equation}
namely
\begin{equation}\label{eq:2.3}
    I_{\mu}\ast\tilde{U}^{\tilde{p}+1}=c(\mu)^{1-p}\tilde{U}^{p-\tilde{p}}\approx \tilde{U}^{p-\tilde{p}}    \textit{\quad in \ \  $\R^{N}$},
\end{equation}
this relation will simplify our calculations and deduce some new estimations.  By taking the derivative of equation \eqref{eq:2.1}, we know that  $ \partial_\lambda\tilde{U}$ and $\nabla_z\tilde{U}$ satisfy
\begin{equation}\label{eq:2.4}
\aligned -\Delta(\partial_\lambda\tilde{U})
 & = (\tilde{p}+1)(I_{\mu}\ast\tilde{U}^{\tilde{p}}\partial_\lambda\tilde{U})\tilde{U}^{\tilde{p}}+\tilde{p}(I_{\mu}\ast\tilde{U}^{\tilde{p}+1})\tilde{U}^{\tilde{p}-1}\partial_\lambda\tilde{U},\endaligned
    \end{equation}
and
 \begin{equation}\label{eq:2.5}
  \aligned  -\Delta (\nabla_z \tilde{U})
 & =(\tilde{p}+1)(I_{\mu}\ast\tilde{U}^{\tilde{p}}\nabla_z\tilde{U})\tilde{U}^{\tilde{p}}+\tilde{p}(I_{\mu}\ast\tilde{U}^{\tilde{p}+1})\tilde{U}^{\tilde{p}-1}\nabla_z\tilde{U} .\endaligned
    \end{equation}
  F{}inally, it follows from \eqref{eq:1.14} that  the $\lambda$-derivative of the Talenti bubble is
$$
\aligned\partial_\lambda\tilde{U}(x)
&=c(\mu)
\frac{N-2 }{2\lambda}
\bigg(\dfrac{(N(N-2))^{\frac{N-2 }{4}}\lambda^{\frac{N-2 }{2}} }{(1+\lambda^2|x-z|^2)^{\frac{N-2}{2}}}\bigg)
\frac{1-\lambda^2|x-z|^2}{1+\lambda^2|x-z|^2}\\
&=c(\mu)\frac{N-2 }{2\lambda}
U(x)\left(\frac{1-\lambda^2|x-z|^2}{1+\lambda^2|x-z|^2} \right)=\frac{N-2 }{2\lambda} \tilde{U}(x)
  \left(\frac{1-\lambda^2\abs{x-z}^2}{1+\lambda^2\abs{x-z}^2}\right), \endaligned
 $$
then
\begin{equation}\label{eq:2.6}
    \Big|\partial_{\lambda}\big|_{\lambda=1} \tilde{U}\Big| \lesssim \tilde{U}.
\end{equation}
The gradient of $\partial_{\lambda} \tilde{U}$ is
 $$
    \aligned
    \nabla_{x}(\partial_{\lambda} \tilde{U})(x)&= \frac{N-2}{2\lambda} \nabla \tilde{U}(x)\left( \frac{1-\lambda^2|x-z|^2}{1+\lambda^2|x-z|^2} \right)+\frac{N-2}{2\lambda}  \tilde{U}(x)\nabla\left( \frac{1-\lambda^2|x-z|^2}{1+\lambda^2|x-z|^2} \right)   \\
    &=\frac{N-2}{2\lambda} \nabla \tilde{U}(x)\left( \frac{1-\lambda^2|x-z|^2}{1+\lambda^2|x-z|^2} \right)+\frac{N-2}{2\lambda}  \tilde{U}(x)\frac{-4\lambda^2 (x-z) }{(1+\lambda^2|x-z|^2)^2},
    \endaligned
 $$
then for $\lambda=1$, there holds
\begin{equation}\label{eq:2.7}
    |\nabla_{x}(\partial_{\lambda} \tilde{U})| \lesssim  |\nabla \tilde{U}|+ \tilde{U}^{1+\frac{2}{N-2}}\approx |\nabla \tilde{U}|+\tilde{U}^{\frac{p+1}{2}}.
\end{equation}
Given $\lambda>0$ and $z\in\R^N$, let $T_{z,\lambda}:C^\infty_c(\R^N)\to C^\infty_c(\R^N)$ be the operator
def{}ined as
\begin{equation}\label{eq:2.8}
    T_{z,\lambda}(\varphi)(x)\coloneqq\lambda^{\frac{N-2 }2}\varphi(\lambda(x-z)).
\end{equation}
The operator $T_{z,\lambda}$ satisf{}ies a multitude of properties:
\begin{itemize}
\item The operator $T_{z,\lambda}$ is invertible, and $T_{z,\lambda}^{-1}=T_{-\lambda z,\lambda^{-1}}$.
\item For any function $\varphi\in C^\infty_c(\R^N)$ it holds that
\begin{equation}\label{eq:2.9}
    \int_{\R^N} \abs{ T_{z,\lambda}(\varphi)}^{p+1}\,dx
= \int_{\R^N} \abs{\varphi}^{p+1}\,dx,\ \
    \int_{\R^N} \abs{\nabla T_{z,\lambda}(\varphi)}^2\,dx= \int_{\R^N} \abs{\nabla\varphi}^2\,dx.
\end{equation}
\item Given $k\in\N^{+}$, for any choice of positive exponents $(e_i)_{1\leq i \leq k}$ with
$e_1+\cdots+e_k=\tilde{p}+1$, and for any choice of non-negative functions $\varphi_1,\dots,\varphi_k\in C^\infty_c(\R^N)$, it holds that
\begin{equation}\label{eq:2.10}
    \aligned\ &\ \int_{\R^N}
 \big(I_\mu\ast\big(T_{z,\lambda}(\varphi_1)^{e_1}\cdots T_{z,\lambda}
 (\varphi_k)^{e_k}\big)\big)\,
 T_{z,\lambda}(\varphi_1)^{e_1}\cdots
 T_{z,\lambda}(\varphi_k)^{e_k}\,dx\\
 \ = &\   \lambda^{\frac{N-2}{2}(\tilde{p}+1)2}\int_{\R^N}\int_{\R^N}\frac{\varphi_1(\lambda(y-z))^{e_1}\cdots \varphi_k(\lambda(y-z))^{e_k} }{|x-y|^{\mu}} \varphi_1(\lambda(x-z))^{e_1}\cdots\varphi_k(\lambda(x-z))^{e_k}
 \,dy\,dx\\
 \ =&\ \lambda^{(N-2)(\tilde{p}+1)+\mu-2N}\int_{\R^N}\int_{\R^N} \frac{ \varphi_1(y)^{e_1}
\cdots\varphi_k(y)^{e_k}}{ |x-y|^{\mu}}\varphi_1(x)^{e_1}
\cdots\varphi_k(x)^{e_k}\,dy\,dx \\
 \ =&\ \int_{\R^N}\big(I_\mu\ast\big(\varphi_1^{e_1}
 \cdots\varphi_k^{e_k}\big)\big)\,\varphi_1^{e_1}
 \cdots\varphi_k^{e_k}\,dx.\endaligned
\end{equation}
Particularly, we observe that
 $$
\int_{\R^N}\big(I_\mu\ast
T_{z,\lambda}(\varphi)^{\tilde{p}+1}\big)
T_{z,\lambda}(\varphi)^{\tilde{p}+1}\,dx
=\int_{\R^N}(I_\mu\ast \varphi^{\tilde{p}+1})
\varphi^{\tilde{p}+1}\,dx
$$
for any $\varphi\in C^\infty_c(\R^N)$.
\end{itemize}
 \vskip0.06in
The transformations $T_{z,\lambda}$ play an important role in the study of the Hardy-Littlewood-Sobolev inequality, since the inequalities \eqref{eq:1.2}, \eqref{eq:1.10} and the equation \eqref{eq:1.12} are invariant under this transformation.
 \vskip0.08in
To prove our main results, we need the following two Lemmas, the proof can be found in the appendix B of F{}igalli and Glaudo \cite{figalli}.

\begin{lemma}\cite[Proposition B.2]{figalli}\label{prop:interaction_approx}
Given $N\ge 3$, let $\tilde{U}=\tilde{U}[z_1, \lambda_1]$ and $\tilde{V}=\tilde{U}[z_2, \lambda_2]$ be two bubbles from \eqref{eq:1.14} such that
$\lambda_1\ge\lambda_2$.
Let us def{}ine the quantity $Q=Q(z_1, \lambda_1, z_2,\lambda_2)$ as
\begin{equation}\label{eq:2.11}
    Q \coloneqq\min\bigg(\frac{\lambda_1}{\lambda_2},\frac{\lambda_2}{\lambda_1}, \frac1{\lambda_1\lambda_2 \abs{z_1-z_2}^2}\bigg)=\bigg(\frac{\lambda_2}{\lambda_1}, \frac1{\lambda_1\lambda_2 \abs{z_1-z_2}^2}\bigg).
\end{equation}
Then, for any f{}ixed $ \varepsilon>0$ and any non-negative exponents such
 that $\alpha+\beta=p+1$,
it holds
$$\int_{\R^N} \tilde{U}^\alpha \tilde{V}^\beta\,dx\approx_{N,\mu,\varepsilon}
\begin{cases}
 Q^{\frac{(N-2 )\min(\alpha, \beta)}2}&\text{if}\ \abs{\alpha-\beta}\ge\varepsilon,\\
 Q^{\frac N2}\log(\frac1Q)&\text{if }\ \  \alpha=\beta.
\end{cases}$$
In particular ,
\begin{equation}\label{eq:2.12}
    \int_{\R^N} \tilde{U}^{p} \tilde{V} \,dx\approx Q^{\frac{N-2}{2}}.
\end{equation}
\end{lemma}

\begin{lemma}\cite[Corollary B.4]{figalli}\label{cor:interaction_integral_localized}
      Given $N\ge 3$ and two bubbles $\tilde{U}_1=\tilde{U}[z_1, \lambda_1]$ and
      $\tilde{U}_2=\tilde{U}[z_2, \lambda_2]$ with
      $\lambda_1\ge\lambda_2$, it holds
      $$
        \int_{\R^N} \tilde{U}_1^{p} \tilde{U}_2\,dx  \approx \int_{B(z_1, \lambda_1^{-1})}
        \tilde{U}_1^{p}\tilde{U}_2\,dx .$$
\end{lemma}

To prove our main theorem in this current paper, we  have to prepare  two inequalities, which might be elementary. But we cannot f{}ind the references in the literature. For readers convenience, we provide  a  complete proof based on the compactness argument.
\begin{proposition}\label{propos-1}
  Assume  $r\geq 1$, $l\geq 0$. Then there exists a constant $C_{r}>0$ such that
 \begin{equation}\label{eq:2.13}
 \aligned
   &\big|(a+b)|a+b|^{r-1}e^{l(|a|+|b|)} - a|a|^{r-1}e^{l|a|}-(r|a|^{r-1}+la|a|^{r-1})e^{l|a|}b\,\big|\\
   \leq\ &  C_r\big(\{|a|^{r-2}|b|^2 e^{l|a|}\}_{r>2}+l|a|^{r}|b|e^{l|a|} +|b|^re^{l(|a|+|b|)}\big)
   \endaligned
  \end{equation}
  holds for any $a,b\in\R$.
\end{proposition}
 \begin{proof}
 Notice that the inequality is always true with $C_{r}=1$ as $a=0$ or $b=0$. Suppose the inequality \eqref{eq:2.13} fails, then there exist $\{a_{j}\}_{j\geq 1},\{b_{j}\}_{j\geq 1}\in \mathbb{R}\setminus\{0\} $
and  $C_{j}\to +\infty$ as $j\to \infty$, such that for any $j\geq 1$, there holds
\begin{equation}\label{eq:2.14}
\aligned
&\big|(a_{j}+b_{j})|a_{j}+b_{j}|^{r-1}e^{l(|a_{j}|+|b_{j}|)}-a_{j}|a_{j}|^{r-1}e^{l|a_{j}|}- (r|a_{j}|^{r-1}+la_{j}|a_{j}|^{r-1})e^{l|a_{j}|}b_{j}\big|\\
>\ \ &C_j\big(\{|a_{j}|^{r-2}|b_{j}|^2 e^{l|a_{j}|}\}_{r>2} +l|a_{j}|^{r}|b_{j}|e^{l|a_{j}|} +|b_{j}|^re^{l (|a_{j}|+|b_{j}|)}\big),
\endaligned
\end{equation}	
then we divide by $|a_{j}|^{r}e^{l|a_{j}|}$ in both side of \eqref{eq:2.14} to get
\begin{equation}\label{eq:2.15}
\aligned
&\Big| \,\Big(1+\frac{b_{j}}{a_{j}}\Big)\Big|1+\frac{b_{j}}{a_{j}}\Big|^{r-1}e^{l |b_{j}|}-1-r\,\frac{b_{j}}{a_{j}}-lb_{j}\, \Big|> C_{j}\Big(  \Big\{ \Big|\frac{b_{j}}{a_{j}} \Big|^{2}  \Big\}_{r>2} + l|b_{j}|+ \Big|\frac{b_{j}}{a_{j}} \Big|^{r} e^{l|b_{j}|}   \Big).
\endaligned
\end{equation}
We claim that $\varliminf\limits_{j\to \infty}\big|\frac{b_{j}}{a_{j}}\big|=0$. In fact, if $\varliminf\limits_{j\to \infty}\big|\frac{b_{j}}{a_{j}}\big|>0$, then up to a subsequence, there exists a constant $m>0$, such that $\big|\frac{b_{j}}{a_{j}}\big|\geq m$. Then by \eqref{eq:2.15}, we have
$$                                                     \aligned
&\Big[\,(1+\varlimsup\limits_{j\to \infty}\big|\frac{b_{j}}{a_{j}}\big|)^{r}e^{l\cdot\varlimsup\limits_{j\to \infty}|b_{j}|}+1+r\cdot\varlimsup\limits_{j\to \infty}\big|\frac{b_{j}}{a_{j}}\big| +l\cdot \varlimsup\limits_{j\to \infty}|b_{j}|  \,\Big]\geq  \lim\limits_{j\to \infty} C_{j}(\{m^{2}\}_{r>2}+m^{r}),
\endaligned
 $$
which means  that either $\varlimsup\limits_{j\to \infty}\big|\frac{b_{j}}{a_{j}}\big|=+\infty$ or $\varlimsup\limits_{j\to \infty}|b_{j}|=+\infty$, then we divide by $|b_{j}|^{r}e^{l(|a_{j}|+|b_{j}|) }$ in both side of \eqref{eq:2.14} we obtain
\begin{equation}\label{eq:2.16}
\aligned
&\Big| \,\Big(\frac{a_{j}}{b_{j}}+1\Big)\Big|\frac{a_{j}}{b_{j}}+1\Big|^{r-1}-\frac{a_{j}}{b_{j}}\Big|\frac{a_{j}}{b_{j}}\Big|^{r-1}e^{-l|b_{j}|}-r \Big|\frac{a_{j}}{b_{j}}\Big|^{r-1}e^{-l|b_{j}|} -l \frac{a_{j}}{b_{j}}\Big|\frac{a_{j}}{b_{j}}\Big|^{r-1}b_{j}e^{-l|b_{j}|} \, \Big|\\
>\ \ &C_{j}\Big(  \Big\{\Big|\frac{a_{j}}{b_{j}} \Big|^{r-2}\Big\}_{r>2}\cdot e^{-l|b_{j}|} +l\Big|\frac{a_{j}}{b_{j}} \Big|^{r} \cdot|b_{j}| \cdot e^{-l|b_{j}|} +1  \Big)>C_{j}.
\endaligned
\end{equation}
If $\varlimsup\limits_{j\to \infty}\big|\frac{b_{j}}{a_{j}}\big|=+\infty$, then up to a subsequence, there holds $\lim\limits_{j\to \infty}\big|\frac{a_{j}}{b_{j}}\big|=0$,
let $j\to \infty$ in \eqref{eq:2.16}, we arrive at $1\geq +\infty$, which gives a contradiction. If $\varlimsup\limits_{j\to \infty}|b_{j}|=+\infty$, then up to a subsequence, there holds $\lim\limits_{j\to \infty}|b_{j}|=+\infty$. Recall that we have $\big|\frac{a_{j}}{b_{j}}\big|\leq \frac{1}{m}$, by letting  $j\to \infty$ in \eqref{eq:2.16}, we arrive at $(\frac{1}{m}+1)^{r}\geq +\infty$, which also gives contradiction. Therefore $\varliminf\limits_{j\to \infty}\big|\frac{b_{j}}{a_{j}}\big|=0$. Assume $\lim\limits_{j\to \infty}\big|\frac{b_{j}}{a_{j}}\big|=0$, up to a subsequence.Then, taking the Taylor's expansion of the left hand side  of \eqref{eq:2.15} with respect to $\frac{b_{j}}{a_{j}}$, we get that
$$
\aligned
 &\Big| \,\Big(1+\frac{b_{j}}{a_{j}}\Big)\Big|1+\frac{b_{j}}{a_{j}}\Big|^{r-1}e^{l |b_{j}|}-1-r\,\frac{b_{j}}{a_{j}}-lb_{j}\, \Big|\\
 =\ &\bigg|\Big(1+\frac{b_{j}}{a_{j}}\Big)\Big|1+\frac{b_{j}}{a_{j}}\Big|^{r-1}\Big(e^{l|b_{j}|}-1-l b_{j}\Big)+\Big[\Big(1+\frac{b_{j}}{a_{j}}\Big)\Big|1+\frac{b_{j}}{a_{j}}\Big|^{r-1}-1   \Big]l b_{j} \\
    & + \Big(1+\frac{b_{j}}{a_{j}}\Big)\Big|1+\frac{b_{j}}{a_{j}}\Big|^{r-1}-1 -r\frac{b_{j}}{a_{j}} \bigg|\\
 \leq\ & 2^{r+1}l\cdot |b_{j}| + O(\Big|\frac{b_{j}}{a_{j}}\Big|)  \cdot l\cdot |b_{j}| +\frac{r(r-1)}{2}\Big|\frac{b_{j}}{a_{j}}\Big|^{2}+o\,(\Big|\frac{b_{j}}{a_{j}}\Big|^{2}) \\
 \leq\ & (2^{r+1}+1)l\cdot |b_{j}|+\frac{r(r-1)}{2}\Big|\frac{b_{j}}{a_{j}}\Big|^{2}+o\,(\Big|\frac{b_{j}}{a_{j}}\Big|^{2})  \text{   \quad as \quad    } j\gg 1 ,
\endaligned
 $$
hence the inequality  $\eqref{eq:2.15}$ reduces  to
\begin{equation}\label{eq:2.17}
\aligned
 (2^{r+1}+1)l\cdot |b_{j}|+\frac{r(r-1)}{2}\Big|\frac{b_{j}}{a_{j}}\Big|^{2}+o\,(\Big|\frac{b_{j}}{a_{j}}\Big|^{2}) &>   C_{j}\Big(  \Big\{ \Big|\frac{b_{j}}{a_{j}} \Big|^{2}  \Big\}_{r>2}+ l|b_{j}| + \Big|\frac{b_{j}}{a_{j}} \Big|^{r} \Big) \\
 &>C_{j}\Big(   l|b_{j}| + \Big|\frac{b_{j}}{a_{j}} \Big|^{2} \Big)  \text{   \quad as \quad    } j\gg 1 .
\endaligned
\end{equation}
On the other hand, we take $j\gg1$, such that $o\,(\big|\frac{b_{j}}{a_{j}}\big|^{2})<\frac{r(r-1)}{2}\big|\frac{b_{j}}{a_{j}}\big|^{2}$ and $C_{j}>(2^{r+1}+1)+r(r-1)$, then
$$
(2^{r+1}+1)l\cdot |b_{j}|+\frac{r(r-1)}{2}\big|\frac{b_{j}}{a_{j}}\big|^{2}+o\,(\big|\frac{b_{j}}{a_{j}}\big|^{2})<(2^{r+1}+1)l\cdot |b_{j}|+ r(r-1)\big|\frac{b_{j}}{a_{j}}\big|^{2} < C_{j}\Big(  l|b_{j}|+\big|\frac{b_{j}}{a_{j}}\big|^{2} \Big),
 $$
which contradicts to \eqref{eq:2.17}. Therefore there exists a constant $C_{r}>0$ such that \eqref{eq:2.13} holds for all $a,b\in\mathbb{R}$.
 \end{proof}

\begin{proposition}
    For the real number $r\geq 1$ and the integer $\nu\geq 1 $, there exists a constant $C_{r,\nu}>0$, such that
   \begin{equation}\label{eq:2.18}
     \Big|(\sum_{i=1}^{\nu}a_i)\big|\sum_{i=1}^{\nu}a_i\big|^{r-1}
  -\sum_{i=1}^\nu a_i|a_i|^{r-1}\Big|\leq C_{r,\nu} \sum_{1\le i\not=j \le\nu} |a_i|^{r-1}|a_j|
  \end{equation}
  holds for any $a_1,\cdots,a_\nu\in\R$.
  \end{proposition}
\begin{proof} If  $r=1$ or $\nu=1$, the inequality \eqref{eq:2.18} always holds, thus we assume $r>1$ and $\nu\geq 2$. Next we  apply  the  induction method   on $\nu$ to prove the inequality \eqref{eq:2.18}.
\vskip0.08in
\noindent {\bf Step 1:}
Consider the case $\nu=2 $, it is suf{}f{}ice to prove that, there exists a constant $C_{r}>0$, such that
\begin{equation}\label{eq:2.19}
\big|(a+b)|a+b|^{r-1}-a|a|^{r-1}-b|b|^{r-1}\big|\leq C_{r}( |a|^{r-1}|b|+|a||b|^{r-1})
\end{equation}
holds for any $a,b\in \mathbb{R}$. Obviously, the inequality \eqref{eq:2.19} is true when  $a=0$ or $b=0$. Suppose the inequality \eqref{eq:2.19} fails, then there exist $\{a_{j}\}_{j\geq 1},\{b_{j}\}_{j\geq 1}\in \mathbb{R}\setminus\{0\} $
and $C_{j}\to +\infty$ as $j\to +\infty$, such that for any $j\geq 1$, there holds
\begin{equation}\label{eq:2.20}
\big|(a_{j}+b_{j})|a_{j}+b_{j}|^{r-1}-a_{j}
|a_{j}|^{r-1}-b_{j}|b_{j}|^{r-1}\big|>C_{j}
\Big( |a_{j}|^{r-1}|b_{j}|+|a_{j}||b_{j}|^{r-1}\Big),
\end{equation}	
then we divide by $|a_{j}|^{r}$ in both sides and  get that
\begin{equation}\label{eq:2.21}
\Big| \,(1+\frac{b_{j}}{a_{j}})\big|1+\frac{b_{j}}{a_{j}}\big|^{r-1}-1-\frac{b_{j}}{a_{j}}|\frac{b_{j}}{a_{j}}|^{r-1}\, \Big|> C_{j}\Big(  \big|\frac{b_{j}}{a_{j}} \big| +  \big|\frac{b_{j}}{a_{j}} \big|^{r-1}   \Big).
\end{equation}
We claim that $\varliminf\limits_{j\to \infty}\big|\frac{b_{j}}{a_{j}}\big|=0$. In fact, if $\varliminf\limits_{j\to \infty}\big|\frac{b_{j}}{a_{j}}\big|>0$, then up to a subsequence, there exists $m>0$, such that $\big|\frac{b_{j}}{a_{j}}\big|\geq m$  as  $j\gg 1$. Then by $\eqref{eq:2.21}$, we have
$$
\Big[(1+\varlimsup\limits_{j\to \infty}\big|\frac{b_{j}}{a_{j}}\big|\big)^{r}+1+\big(\varlimsup\limits_{j\to \infty}\big|\frac{b_{j}}{a_{j}}\big|\big)^{r}  \Big]>\lim\limits_{j\to \infty}C_{j}(m+m^{r-1}) ,
 $$
which means $\varlimsup\limits_{j\to \infty}\big|\frac{b_{j}}{a_{j}}\big|=+\infty$, so we can assume $\lim\limits_{j\to \infty}\big|\frac{a_{j}}{b_{j}}\big|=0$ up to a subsequence.  Dividing  by $|b_{j}|^{r}$ in both sides of $\eqref{eq:2.20}$, we observe that
 $$
\Big| \,(\frac{a_{j}}{b_{j}}+1)\big|\frac{a_{j}}{b_{j}}+1\big|^{r-1}-\frac{a_{j}}{b_{j}}\big|\frac{a_{j}}{b_{j}}\big|^{r-1}-1\, \Big|> C_{j}\Big(  \big|\frac{a_{j}}{b_{j}} \big|^{r-1} + \big|\frac{a_{j}}{b_{j}} \big|  \Big),
 $$
then the Taylor's expansion of the left hand side with respect to $\frac{a_{j}}{b_{j}}$ gives
\begin{equation}\label{eq:2.22}
\Big| \,r\frac{a_{j}}{b_{j}}-\frac{a_{j}}{b_{j}}\big|\frac{a_{j}}{b_{j}}\big|^{r-1}+o\Big(\big| \frac{a_{j}}{b_{j}}\big|\Big)\, \Big|> C_{j}\Big(   \big|\frac{a_{j}}{b_{j}} \big|^{r-1} + \big|\frac{a_{j}}{b_{j}} \big|   \Big)  \text{   \quad as \quad    } j\gg 1.
\end{equation}
Therefore, as  $j\gg 1$ (using $r>1$), we get that
 $$
2r \big| \frac{a_{j}}{b_{j}}\big| >C_{j}\big| \frac{a_{j}}{b_{j}}\big|,
 $$
which gives a contradiction. Therefore $\varliminf\limits_{j\to \infty}\big|\frac{b_{j}}{a_{j}}\big|=0$. Assume $\lim\limits_{j\to \infty}\big|\frac{b_{j}}{a_{j}}\big|=0$ up to a subseqence, then  the Taylor's expansion of the left hand side of $\eqref{eq:2.21}$ with respect to $\frac{b_{j}}{a_{j}}$ implies  similar form as  \eqref{eq:2.22}, i.e.,
$$ 
\Big| \,r\frac{b_{j}}{a_{j}}-\frac{b_{j}}{a_{j}}
\big|\frac{b_{j}}{a_{j}}\big|^{r-1}+o
\Big(\big| \frac{b_{j}}{a_{j}}\big|\Big)\, 
\Big|> C_{j}\Big(   \big|\frac{b_{j}}{a_{j}} \big|^{r-1} 
+ \big|\frac{b_{j}}{a_{j}} \big|   \Big)  
\text{   \quad as \quad    } j\gg 1 ,
 $$
which  also gives a contradiction. Therefore, there exists a constant $C_{r}>0$ such that $\eqref{eq:2.19}$ holds true for any $a,b\in\mathbb{R}$.
\vskip0.108in
\noindent {\bf Step 2:} By induction, we assume that the inequality \eqref{eq:2.18} is true  with $\nu$-Sum and try to prove  \eqref{eq:2.18}   with $(\nu+1)$-Sum. First, since $r>2$, so $f(x)\coloneqq |x|^{r-1}$ is a convex function, thus by the Jensen's inequality
\begin{align}
f(\frac{1}{\nu}\sum^{\nu}_{i=1}a_{i})\leq \frac{1}{\nu}\sum^{\nu}_{i=1}f(a_{i})\implies \big|\sum^{\nu}_{i=1}a_{i}\big|^{r-1}\leq \nu^{r-2}\sum^{\nu}_{i=1}|a_{i}|^{r-1}.
\end{align}
Then for the $(\nu+1)$-Sum, we let $a=\sum_{i=1}^{\nu}a_{i}$, $b=a_{\nu+1}$. By using the estimate \eqref{eq:2.19} and the induction hypothesis which claims the inequality \eqref{eq:2.18} holds with $\nu$-Sum, we obtain that
\begin{align}
\phantom{\leq\,\,}&\bigg| \Big(\sum_{i=1}^{\nu+1}a_{i}  \Big)\Big|  \sum_{i=1}^{\nu+1}a_{i}   \Big|^{r-1}  -\sum_{i=1}^{\nu+1}a_{i}\big|a_{i}\big|^{r-1}  \bigg| \\
\leq\,\,&  \bigg| \Big(\sum_{i=1}^{\nu}a_{i}+a_{\nu+1}  \Big)\Big|  \sum_{i=1}^{\nu}a_{i} +a_{\nu+1}   \Big|^{r-1}  -\Big(\sum_{i=1}^{\nu}a_{i}  \Big)\Big|  \sum_{i=1}^{\nu}a_{i}   \Big|^{r-1}-a_{\nu+1}\big|a_{\nu+1}\big|^{r-1}  \bigg|\\
&+ \bigg| \Big(\sum_{i=1}^{\nu}a_{i}  \Big)\Big|  \sum_{i=1}^{\nu}a_{i}   \Big|^{r-1}  -\sum_{i=1}^{\nu}a_{i}\big|a_{i}\big|^{r-1}  \bigg| \\
=\,\,&\Big| (a+b)|a+b|^{r-1}  -a|a|^{r-1}-b|b|^{r-1}  \Big|+\bigg| \Big(\sum_{i=1}^{\nu}a_{i}  \Big)\Big|  \sum_{i=1}^{\nu}a_{i}   \Big|^{r-1}  -\sum_{i=1}^{\nu}a_{i}\big|a_{i}\big|^{r-1}  \bigg|\\
\leq\,\,& C_{r}(|a|^{r-1} |b|+|a||b|^{r-1})+C_{r,\nu}\sum_{1\leq i\neq j \leq \nu} \big|a_{i}\big|^{r-1} \big|a_{j}\big| \\
\leq \,\,& C_{r}\bigg(\sum^{\nu}_{i=1}|a_{i}|^{r-1}\bigg)|a_{\nu+1}|	+C_{r}\bigg(\sum^{\nu}_{i=1}|a_{i}|\bigg)|a_{\nu+1}|^{r-1}+ C_{r,\nu}\sum_{1\leq i\neq j \leq \nu} \big|a_{i}\big|^{r-1} \big|a_{j}\big|\\
\leq\,\,& C_{r,\nu+1} \sum_{1\leq i\neq j \leq \nu+1} \big|a_{i}\big|^{r-1} \big|a_{j}\big|,
\end{align}
where $C_{r,\nu+1}\coloneqq\max\,(C_{r},C_{r,\nu})$.
\end{proof}

\section{Interaction integral estimate}\label{sec-intraction}

  In this section, we prove the interaction integral inequality \eqref{eq:3.12}, which will play key roles in the Section \ref{sec-quantitative}. As a f{}irst step, we give the localization argument in the next auxiliary Lemma, which is crucial in the interaction integral estimate.

  \begin{lemma}\label{lem:localization}
  For any $N\ge 3$, $\nu\in\N^{+}$ and $ \varepsilon>0$, there exists $\delta=\delta(N,\nu,
  \varepsilon)>0$ such that if
  $U_1=U[z_1,\lambda_1],\dots,U_\nu=U[z_\nu,\lambda_\nu]$ is a
  $\delta$-interacting family of $\nu$ Talenti
  bubbles from \eqref{eq:1.4}, then for any $1\le i\le \nu$ there exists a Lipschitz bump function
  $\Phi_i:\R^N\to[0,1] $ such that the following statements hold:
  \begin{enumerate}[label=\textit{(\arabic*)}]
  \item \label{it:localization1}The mass of $U_{i}^{p+1}$ and $|\nabla U_{i}|^{2} $ in the region $\{\Phi_i<1\}$ is small, that is
  $$\int_{\{\Phi_i<1\}} U_{i}^{p+1}\,dx\leq  \varepsilon\  \text{\ \ \ and\ \ \ } \int_{\{\Phi_i<1\}} |\nabla U_{i}|^{2}\,dx\leq  \varepsilon  ;  $$
  \item\label{it:localization2}In the region $\{\Phi_i>0\}$ it holds $ \varepsilon U_i >U _j$
  for any $j\not=i$, which implies that  dif{}ferent $\Phi_{i}$ has disjoint supports;
  \item \label{it:localization3}The $L^N$-norm of the gradient is small, that is
  $$\|\nabla\Phi_i\|_{L^N}\le\varepsilon;$$
  \item \label{it:localization4}
  For any $j\not=i$ such that $\lambda_j\le \lambda_i$, it holds
  $$\frac{\sup_{\{\Phi_i>0\}}U_j}{\inf_{\{\Phi_i>0\}}U_j}\le 1+\varepsilon.$$
    \end{enumerate}
     \end{lemma}
     \begin{proof}
     Here we  only give the  proof   for $i=\nu$ and assume $U_{\nu}=U[0,1]$. Thus,  by the def{}inition of $\delta$-interacting, we have that
   $$
         \max_{j\neq \nu}\ \min\Big(\lambda_{j},\frac{1}{\lambda_{j}},\frac{1}{\lambda_{j}|z_{j}|^{2}}\Big)\leq \delta.
$$
     Inspired by \cite{figalli}, we take
     \begin{equation}\label{eq:3.1}
        \delta=\epsilon ^{16-\frac{2}{N-2 }}, \quad R=\epsilon ^{-2}.
     \end{equation}
Going back to the Def{}inition \ref{def-delta-intracting} of $\delta$-interacting, we observe that if $0<\delta^{\prime}<\delta$, then any $\delta^{\prime}$-interacting family is also a $\delta$-interacting family.
     Based  on the choice \eqref{eq:3.1} and note that the number of bubbles is f{}inite, by shrinking $\epsilon$, we can make further simplif{}ication that
$$
         \delta=\min\Big(\lambda_{j},\frac{1}{\lambda_{j}},\frac{1}{\lambda_{j}|z_{j}|^{2}}\Big)
 $$
     for any $j\neq \nu$.
     \vskip0.06in
     \noindent{\bf Claim 1.}
      For any $1\le j < \nu$ such that $\lambda_j < 1$ or $\abs{z_j}>2R$, we have  $\epsilon U>
       U_j$ in $B(0,R)$.

  \noindent {\bf Proof of the claim 1.} Note that the minimum of $U$ is on the boundary $\{|x|=R\}$, and the maximum of $U_{j}$ is at the center $z_{j}$, so we only need to verify
     \begin{equation}\label{eq:3.2}
         \bigg(\frac{\lambda_{j}}{1+\lambda_{j}^{2}|x-z_{j}|^{2}}  \bigg)^{\frac{N-2}{2}}<\epsilon
         \bigg(\frac{1}{1+R^{2}} \bigg)^{\frac{N-2}{2}} \text{\quad for \quad}  |x|\leq R,
     \end{equation}
     or even more strongly
     \begin{equation}\label{eq:3.3}
         \lambda_{j}^{\frac{N-2}{2}}<\epsilon \Big( \frac{1}{1+R^{2}} \Big)^{\frac{N-2}{2}}.
     \end{equation}
   We divide the proof into the following two cases.
     \vskip0.08in
    {\bf Case 1:} $\lambda_{j}<1$. So $\delta =
     \min\big(\lambda_{j},\frac{1}{\lambda_{j}},\frac{1}{\lambda_{j}|z_{j}|^{2}}\big)=\min\big(\lambda_{j},\frac{1}{\lambda_{j}|z_{j}|^{2}}\big)$.
    If $\frac{1}{\lambda_{j}|z_{j}|^{2}}\geq\lambda_{j}=\delta$, then
    \begin{align}
    \eqref{eq:3.3}\iff
    \epsilon ^{16-\frac{2}{N-2 }}=\delta=\lambda_{j}<\epsilon ^{\frac{2}{N-2 }}\cdot \frac{1}{1+R^{2}}
    \approx\epsilon ^{4+\frac{2}{N-2 }},
    \end{align}
    where the last approximation ``$\approx$'' is from the fact that $$\epsilon \to0\implies R\to\infty\implies
    1+R^{2}\approx R^{2}=\epsilon ^{-4}.$$
     Thus it suf{}f{}ices to compare  the exponent $16-\frac{2}{N-2}$ and $4+\frac{2}{N-2}+4$. Note that  $16-\frac{2}{N-2 }\geq 16-2=14>6\geq 4+\frac{2}{N-2}$ implies that
$$
     \epsilon^{16-\frac{2}{N-2}}\ll
    \epsilon^{4+\frac{2}{N-2 }}\text{\quad as \quad}\epsilon \ll 1\implies\eqref{eq:3.3}  \implies U_{j}\ll\epsilon  U\text{\quad as \quad}\epsilon \ll 1.
 $$
    If $\lambda_{j}>\frac{1}{\lambda_{j}|z_{j}|^{2}}=\delta$, then $|z_{j}|>\frac{1}{\lambda_{j}}>1$.

When  $|z_{j}|\geq2R\geq2|x|$, then $|z_{j}-x|\geq|z_{j}|-|x|>\frac{1}{2}|z_{j}|$,
  $$
    \aligned
    \bigg(\frac{\lambda_{j}}{1+\lambda_{j}^{2}|x-z_{j}|^{2}}
    \bigg)^{\frac{N-2 }{2}}&<\bigg(\frac{1}{\lambda_{j}^{-1}+\frac{1}{4}
    \lambda_{j}|z_{j}|^{2}}\bigg)^{\frac{N-2 }{2}}\leq\bigg(\frac{1}{0+\frac{1}{4}\lambda_{j}|z_{j}|^{2}}
    \bigg)^{\frac{N-2 }{2}}=(4\delta)^{\frac{N-2 }{2}}.
    \endaligned
 $$
    It follows that
  $$
    \aligned
    \eqref{eq:3.2} (4\delta)^{\frac{N-2 }{2}}<\epsilon
    \bigg(\frac{1}{1+R^{2}}\bigg)^{\frac{N-2 }{2}}&\iff \delta^{-1}>4\epsilon ^{-\frac{2}{N-2 }}(1+R^{2})\\
    &\iff \epsilon ^{\frac{2}{N-2 }}\epsilon^{-(16-\frac{2}{N-2 })}>4(1+\epsilon ^{-4})\\
    &\iff\epsilon^{-(16-\frac{4}{N-2 })}>4(1+\epsilon ^{-4}).
    \endaligned
   $$
    Notice $(16-\frac{4}{N-2 })\geq (16-4) >4$, thus we have \eqref{eq:3.2} and therefore $U_{j}<\epsilon U$ for $\epsilon \ll1$.

When  $|z_{j}|<2R$, then  $2R\lambda_{j}|z_{j}|>\lambda_{j}|z_{j}|^{2}
    =\frac{1}{\delta}$, thus $|\lambda_{j}z_{j}-\lambda_{j}x|\geq \lambda_{j}|z_{j}|
    -\lambda_{j}|x|\geq\frac{1}{2\delta R}-R$,
$$
    \aligned \bigg(\frac{\lambda_{j}}{1+\lambda_{j}^{2}|x-z_{j}|^{2}}\bigg)^{\frac{N-2 }{2}}
    &<\bigg(\frac{1}{1+|\lambda_{j}z_{j}-\lambda_{j}x|^{2}}\bigg)^{\frac{N-2 }{2}} <\bigg(\frac{1}{0+|\frac{1}{2\delta R}-R|^{2}}\bigg)^{\frac{N-2 }{2}},\endaligned
 $$
   then \begin{equation}\label{figa218note2}
  \aligned \eqref{eq:3.2}\ \impliedby\ &\ \bigg(\frac{1}{0+|\frac{1}{2\delta R}-R|^{2}}\bigg)^{\frac{N-2 }{2}}
   <\epsilon \bigg(\frac{1}{1+R^{2}}\bigg)^{\frac{N-2}{2}}\\
  \ \iff &\ \Big|\frac{1}{2\delta R}-R\Big|^{2}
   >\epsilon ^{-\frac{2}{N-2 }}(1+R^{2})\\
  \ \iff &\ \Big(\frac{1}{2}\cdot\epsilon ^{-(16-\frac{2}{N-2}-2)}-\epsilon ^{-2} \Big)^{2}>(\epsilon ^{-\frac{2}{N-2 }}+\epsilon ^{-\frac{2}{N-2 }-4}).\endaligned
    \end{equation}
    Thus it is suf{}f{}ice to compare the exponent $2(16-\frac{2}{N-2}-2)$ and $\frac{2}{N-2 }+4$.
    Since
    $2(16-\frac{2}{N-2}-2)\geq 2(16-2-2)>6\geq\frac{2}{N-2}+4$, thus we have \eqref{eq:3.2} and therefore $U_{j}<\epsilon U$ for $\epsilon \ll1$.
     \vskip0.08in
   {\bf Case 2:} $|z_{j}|>2R$. If $\lambda_{j}<1$, then  it returns  to  the Case 1. Therefore, we  assume $\lambda_{j}\geq 1$. It implies that $\delta =
     \min\big(\lambda_{j},\frac{1}{\lambda_{j}},\frac{1}{\lambda_{j}|z_{j}|^{2}}\big)=\min\big(\frac{1}{\lambda_{j}},\frac{1}{\lambda_{j}|z_{j}|^{2}}\big)$, then we have that
  $$|z_{j}|>2R\geq 2 \implies \frac{1}{\lambda_{j}}>\frac{1}{4\lambda_{j}}>
    \frac{1}{\lambda_{j}|z_{j}|^{2}}\implies  \frac{1}{\lambda_{j}|z_{j}|^{2}}=\delta , $$
   $$ |z_{j}|>2R\geq 2|x|\implies |z_{j}-x|\geq |z_{j}|-|x|>\frac{1}{2}|z_{j}|. $$
 then
 $$
    \aligned \bigg(\frac{\lambda_{j}}{1+\lambda_{j}^{2}|x-z_{j}|^{2}}  \bigg)^{\frac{N-2}{2}}= \bigg(\frac{1}{\lambda_{j}^{-1}+\lambda_{j}|x-z_{j}|^{2}}
    \bigg)^{\frac{N-2 }{2}}
    <\bigg(\frac{1}{0+\frac{1}{4}\lambda_{j}|z_{j}|^{2}}  \bigg)^{\frac{N-2 }{2}}=(4\delta)^{\frac{N-2 }{2}}.
    \endaligned
 $$
On the other hand,
    \begin{equation}\label{figa218note3}
    \aligned \eqref{eq:3.2} (4\delta)^{\frac{N-2 }{2}}<
    \epsilon \bigg(\frac{1}{1+R^{2}}\bigg)^{\frac{N-2 }{2}}
     \iff\epsilon^{-(16-\frac{4}{N-2 })}>4(1+\epsilon ^{-4}).
    \endaligned
    \end{equation}
    Notice $16-\frac{4}{N-2 }\geq 16-4=12>4$, thus we have \eqref{eq:3.2} and therefore $U_{j}<\epsilon U$ for $\epsilon\ll1$.
    \noindent{\bf Claim 2.}
       $\dfrac{\sup_{B(0,R)}U_j}{\inf_{B(0,R)}U_j}\leq 1+\epsilon$ \ \ \textup{for any}  $1\leq j<\nu $ \textup{such that} $\lambda_{j}\leq1$.

   \noindent {\bf Proof of the claim 2.}
   Since
       \begin{align}
        U_{j}=(N(N-2 ))^{\frac{N-2 }{2}} \bigg(\frac{\lambda_{j}}{1+\lambda_{j}^{2}|x-z_{j}|^{2}}
        \bigg)^{\frac{N-2 }{2}},	
       \end{align}
       $|z_{j}-x|\leq |z_{j}|+R$ and   $|z_{j}-x|\geq \big||z_{j}|-R\big|$, we have
       \begin{align}
           \dfrac{\sup_{B(0,R)}U_j}{\inf_{B(0,R)}U_j}\leq  \bigg( \frac{1+\lambda_{j}^{2}
           (|z_{j}|+R)^2}{1+\lambda_{j}^{2} (|z_{j}|-R)^2} \bigg)^{\frac{N-2 }{2}}=\bigg(1
           +\frac{4\lambda_{j}^{2}|z_{j}|R}{1+\lambda_{j}^{2} (|z_{j}|-R)^2}
           \bigg)^{\frac{N-2 }{2}}.
       \end{align}
      We claim that
      $(1+x)^{\frac{N-2 }{2}}\leq 1+C(N)x $ as   $x\in[0,1].$
     Indeed, let $h(x)=cx+1-(1+x)^{\frac{N-2 }{2}}$, then
     $$h^{\prime}(x)=c-\frac{N-2 }{2}(1+x)^{\frac{N-4}{2}}.$$
Assume $x\in[0,1]$. If $N=3$  then $ 2^{-\frac{1}{2}}\leq
        (1+x)^{\frac{N-4}{2}}\leq 1, $ and if $
          N\geq4$, we see that $1\leq (1+x)^{\frac{N-4}{2}} \leq 2^{\frac{N-4}{2}}$.
Therefore, we just need to  choose
  \begin{align}
  c=C(N)\coloneqq\begin{cases}
   \frac{N-2 }{2} \cdot 1,\quad
   &N=3,\\
   \frac{N-2}{2}\cdot 2^{\frac{N-4}{2}},\quad &N\geq 4,
  \end{cases}
  \end{align}
it follows that $(1+x)^{\frac{N-2 }{2}}\leq 1+C(N)x$ as  $x\in[0,1].$
       Next we show
     $$\frac{4\lambda_{j}^{2}|z_{j}|R}{1+\lambda_{j}^{2} (|z_{j}|-R)^2} \to
     0 \text{\quad as\quad}\delta\to 0.$$
  Indeed, since $\lambda_{j}\leq 1$, so
  $\delta=\min\big(\lambda_{j},\frac{1}{\lambda_{j}},
  \frac{1}{\lambda_{j}|z_{j}|^{2}}\big)=\min\big(\lambda_{j},
   \frac{1}{\lambda_{j}|z_{j}|^{2}}\big)$. If
     $ \frac{1}{\lambda_{j}|z_{j}|^{2}}\geq
    \lambda_{j}=\delta$, then $\lambda_{j}|z_{j}|\leq 1$, thus
       \begin{align}
           \frac{4\lambda_{j}^{2}|z_{j}|R}{1+\lambda_{j}^{2} (|z_{j}|-R)^2}
            \leq \frac{4\lambda_{j}R}{1+0}
=4\delta R \approx  \epsilon ^{16-\frac{2}{N-2 }}\cdot \epsilon ^{-2}  =\epsilon ^{14-\frac{2}{N-2 }}\to 0 \text{\quad as \quad} \epsilon \to 0,
       \end{align}
       since $14-\frac{2}{N-2 }\geq 12>0$. If $\lambda_{j}>\frac{1}{\lambda_{j}|z_{j}|^{2}}=\delta$, then $1\leq \lambda_{j}|z_{j}|\leq |z_{j}|$,
       \begin{equation}\label{eq:3.4}
           \frac{4\lambda_{j}^{2}|z_{j}|R}{1+\lambda_{j}^{2} (|z_{j}|-R)^2} \leq
            \frac{4\lambda_{j}^{2}|z_{j}|R}{0+\lambda_{j}^{2} (|z_{j}|-R)^2}
            = \frac{4}{\frac{(|z_{j}|-R)^2}{|z_{j}|R}}
            \eqqcolon\dfrac{4}{L}, 
       \end{equation}
       where
       \begin{align}
       L&=  \bigg( \frac{|z_{j}-R|}{|z_{j}|^{\frac{1}{2}} R^{\frac{1}{2}}}\bigg)^{2}
       =\bigg[\bigg(\frac{|z_{j}|}{R} \bigg)^{\frac{1}{2}}-\bigg( \frac{R}{|z_{j}|}
        \bigg)^{\frac{1}{2}}   \bigg]^{2}.
       \end{align}
Recall that $\delta=\frac{1}{\lambda_{j}|z_{j}|^{2}}\geq\frac{1}{|z_{j}|^2}$, thus
    $$\frac{|z_{j}|}{R}\geq\frac{\delta^{-\frac{1}{2}}}{R}=\epsilon^{-8+\frac{1}{N-2}}\cdot
    \epsilon^{2}=\epsilon^{-6+\frac{1}{N-2}}
    \to\infty \text{\quad as \quad} \epsilon \to 0,$$
which implies $L\to \infty$ as $\epsilon
       \to 0$. Therefore,  by \eqref{eq:3.4}
       \begin{align}
            \frac{4\lambda_{j}^{2}|z_{j}|R}{1+\lambda_{j}^{2} (|z_{j}|-R)^2} \to 0
            \text{\quad as \quad} \epsilon\to 0.
       \end{align}
       Thus, if $\epsilon$ is suf{}f{}iciently small,
       $$ \frac{\sup_{B(0,R)}U_j}{\inf_{B(0,R)}U_j}\leq 1+C(N)\frac{4\lambda_{j}^{2}|z_{j}|R}{1+\lambda_{j}^{2} (|z_{j}|-R)^2}\approx 1+\epsilon^{6-\frac{1}{N-2}} \ll 1 + \epsilon, $$
       for any $1\le j < \nu$ such that $\lambda_j\le 1$.
    \vskip0.08in
       Let $G\subset \{1,\ldots,\nu-1\}$ be the set of indices $j$ such that $\lambda_j>1$ and $\abs{z_j}<2R$.
       For any $j\in G$,  by the intermediate value theorem, there exists  $R_j\in(0,\infty)$  such that
       $$\epsilon \left(\frac1{1+R^2}\right)^{\frac{N-2 }2}=\left(\frac{\lambda_j}{1+\lambda_j^2\abs{R_j}^2}
       \right)^{\frac{N-2 }2}.$$

       \noindent{\bf Claim 3.}
       If $\epsilon\ll 1$, then $R_{j}\lesssim \epsilon^{2}. $

    \noindent {\bf Proof of the claim 3.} Since \begin{align}
    \epsilon \bigg( \frac{1}{1+R^2} \bigg)^{\frac{N-2 }{2}}=\bigg( \frac{\lambda_{j}}{1+\lambda_{j}^2|R_{j}|^{2}} \bigg)^{\frac{N-2 }{2}}
    \implies\quad& 1+\lambda_{j}^2|R_{j}|^{2}=\epsilon ^{-\frac{2}{N-2 }}(1+R^{2})\lambda_{j}\\
    \implies\quad& \lambda_{j}^2|R_{j}|^{2}=\epsilon ^{-\frac{2}{N-2 }}(1+R^{2})\lambda_{j}-1\\
    \implies\quad& |R_{j}|^{2}=\frac{\epsilon ^{-\frac{2}{N-2 }}(1+R^{2})}{\lambda_{j}}-\frac{1}{\lambda_{j}^{2}}\lesssim \frac{\epsilon ^{-\frac{2}{N-2 }}R^{2}}{\lambda_{j}}.
    \end{align}
    By the setting $\lambda_j> 1$ and $|z_j|<2R$, so  $\delta =\min\big(\lambda_{j},\frac{1}{\lambda_{j}},\frac{1}{\lambda_{j}|z_{j}|^{2}}\big)=\min\big( \frac{1}{\lambda_{j}},\frac{1}{\lambda_{j}|z_{j}|^{2}}\big)$.
     \vskip0.08in
  {\bf Case 1:}  $|z_{j}|\geq1 \implies \frac{1}{\lambda_{j}}\geq\frac{1}{\lambda_{j}|z_{j}|^{2}}\implies\frac{1}{\lambda_{j}|z_{j}|^{2}}=\delta$, i.e., $\frac{1}{\lambda_{j}}= |z_{j}|^{2}\delta$, therefore
        \begin{align}
        |R_{j}|^{2}\lesssim \epsilon ^{-\frac{2}{N-2 }} R^{2} |z_{j}|^{2}\delta \lesssim  \epsilon ^{-\frac{2}{N-2 }} R^{2} R^{2}\delta \approx \epsilon ^{-\frac{2}{N-2 }-4-4+(16-\frac{2}{N-2 })}
        \approx  \epsilon ^{8-\frac{4}{N-2 }}\lesssim \epsilon ^{4}.
        \end{align}
     \vskip0.08in
  {\bf Case 2:}    $|z_{j}|<1 \implies  \frac{1}{\lambda_{j}|z_{j}|^{2}} >
    \frac{1}{\lambda_{j}} \implies \frac{1}{\lambda_{j}} =\delta $, it follows that
        \begin{align}
        |R_{j}|^{2}\lesssim \epsilon ^{-\frac{2}{N-2 }}R^{2}\delta \approx \epsilon ^{-\frac{2}{N-2 }
        -4+(16-\frac{2}{N-2 })}\approx \epsilon ^{12-\frac{4}{N-2 }} \lesssim \epsilon ^{8}.
        \end{align}
        It is worth mentioning that all the above ``$\lesssim$'' hold as $\epsilon\ll1$. In a word, $|R_{j}|^{2}\lesssim \epsilon ^{4}$ as $\epsilon\ll 1$.
  Thus, as a consequence of the def{}inition of $R_j$, for any $j\in G$ it holds that, in $B(0,R)\setminus B(z_{j},R_j)$,
    $$            \frac{\epsilon U}{U_j}=\epsilon\Big(\frac{1}{1+|x|^{2}}\Big)^{\frac{N-2 }2} \Big/\Big(\frac{\lambda_j}{1+\lambda_j^2\abs{x-z_{j}}^2}\Big)^{\frac{N-2}{2}}   \geq \epsilon \Big(\frac1{1+R^2}\Big)^{\frac{N-2 }2}\Big/\Big(\frac{\lambda_j}{1+\lambda_j^2\abs{R_j}^2}
       \Big)^{\frac{N-2 }2}=1,
 $$
       i.e., $\epsilon U\ge U_{j}$ in $B(0,R)\setminus B(z_{j},R_j)$. Lastly we estimate the order of $\int_{B(0,\epsilon  R)^{c}}| U|^{p+1}\,dx$ and $\int_{B(0,\epsilon  R)^{c}}| \nabla U|^{2}\,dx$ with respect to $\epsilon $, and show that the statement \ref{it:localization1} is hold. Firstly,
       \begin{align}\label{decay-of-L(p+1)norm}
       \int_{B(0,\epsilon  R)^{c}}U^{p+1} \,dx
       \approx \int_{|x|>\epsilon ^{-1}}
        \bigg(\frac{1}{1+|x|^{2}}\bigg)^{\frac{N-2}{2}\cdot \frac{2N}{N-2 }}
        \,dx\approx \int_{\epsilon ^{-1}}^{\infty} \frac{1}{r^{2N}} r^{N-1}
        \,dr\approx \epsilon ^{N}.
       \end{align}
      Similarly
    \begin{align}
     \int_{B(0,\epsilon  R)^{c}}|\nabla U|^{s} \,dx \approx \int_{|x|>\epsilon^{-1}}
    \bigg(\frac{|x|}{(1+|x|^{2})^{\frac{N-2 }{2}+1}}\bigg)^{s}\,dx
    \approx\int_{\epsilon^{-1}}^{\infty}\frac{1}{r^{(N-1)s}}r^{N-1}\,dr\approx \epsilon ^{(N-1)s-N},
    \end{align}
    and
    \begin{align}
    \epsilon ^{(N-1)s-N}\leq\epsilon\iff(N-1)s-N\geq 1\iff s\geq \frac{N+1}{N-1}.
    \end{align}
  Particularly,
    \begin{align}
     \int_{B(0,\epsilon  R)^{c}}|\nabla U|^{2} \,dx \lesssim \epsilon,
    \end{align}
    then combining with the proof of Lemma 3.9 in \cite{figalli}, we have the statement \ref{it:localization1}. For the construction of the function $\Phi_{i}$ and other statements, the readers can refer to Lemma 3.9 in \cite{figalli}.
    \end{proof}
     \vskip0.05in
\begin{remark}\label{power of epilson}
As we have seen, in the above proof we have chosen  $\delta=\varepsilon^{16-\frac{2}{N-2}}$. However,   if we replace it by   $\delta=\varepsilon^{\eta-\frac{2}{N-2}}$, where $\eta\geq16$ is any f{}ixed number, then the above arguments will also hold. Thus in the later proof we can choose the appropriate $\eta$ according to our needs,  this observation will help us to prove the Lemma \ref{prop:interaction_and_coef}.
     \end{remark}

\vskip0.18in
    To bound the double integral terms, such as the term $\int_{\R^{N}} (I_\mu\ast h_{1}^{\alpha_{1}}|\rho|^{\beta_{1}} ) h_{2}^{\alpha_{2}}|\rho|^{\beta_{2}} \,dx $, with some powers of $\|\nabla \rho\|_{L^{2}}$, where $\alpha_{i}+\beta_{i}=\tilde{p}+1$, $i=1,2$, we need the following  Lemma.
    \begin{lemma}\label{double-integral-estimate}
    For $i=1,2,$ let  $\alpha_{i},\beta_{i}> 0$ satisfy $\alpha_{i}+\beta_{i}=\tilde{p}+1$, then for any $h_{i}\in \mathcal{D}^{1,2}$ and $\rho\in\mathcal{D}^{1,2}$, there holds
   $$
        \int_{\R^{N}} (I_\mu\ast h_{1}^{\alpha_{1}}|\rho|^{\beta_{1}} ) h_{2}^{\alpha_{2}}|\rho|^{\beta_{2}} \,dx \lesssim_{h_{1},h_{2}} \|\nabla \rho\|_{L^{2}}^{\beta_{1}+\beta_{2}}
 $$
    \hbox {if} $\tilde{p}+1>\max(\alpha_{1},\beta_{1})$ and $\tilde{p}+1>\max(\alpha_{2},\beta_{2})$.
    \begin{proof}
      Let $s,t\in(1,+\infty)$ be  two constants  whose values are to be determined, such that $1/s+1/t=2(\tilde{p}+1)/(p+1)$.  Then by using the HLS inequality \eqref{eq:1.9}, we know that
      \begin{equation}\label{eq:3.5}
          \int_{\R^{N}} (I_\mu\ast h_{1}^{\alpha_{1}}|\rho|^{\beta_{1}} ) h_{2}^{\alpha_{2}}|\rho|^{\beta_{2}} \,dx \lesssim \|h_{1}^{\alpha_{1}}|\rho|^{\beta_{1}}\|_{L^{s}}\| h_{2}^{\alpha_{2}}|\rho|^{\beta_{2}} \|_{L^{t}}.
      \end{equation}
    Next we apply the H\"older's inequality to  control the above two terms.  Since $h_{i},\rho\in \mathcal{D}^{1,2}$,  we can only control the terms related to  $h_{i}$ or $\rho$ by their  $L^{p+1}$ norms. More precisely, we have that
     \begin{equation}\label{eq:3.6}
     \aligned
     \|h_{1}^{\alpha_{1}}|\rho|^{\beta_{1}}\|_{L^{s}}&=\Big(\int_{\R^{N}} |h_{1}|^{\alpha_{1}s}|\rho|^{\beta_{1}s} \,dx\Big)^{\frac{1}{s}} \\
     &\leq  \Big(\int_{\R^{N}} |h_{1}|^{\alpha_{1}s\cdot\frac{p+1}{\alpha_{1}s}}\,dx\Big)^{\frac{\alpha_{1}s}{p+1}\cdot\frac{1}{s}}  \Big(\int_{\R^{N}} |\rho|^{\beta_{1}s\cdot\frac{p+1}{\beta_{1}s}} \,dx\Big)^{\frac{\beta_{1}s}{p+1}\cdot\frac{1}{s}}\\
     &=\|h_{1}\|_{L^{p+1}}^{\alpha_{1}} \|\rho\|_{L^{p+1}}^{\beta_{1}} \lesssim  \|\nabla h_{1}\|_{L^{2}}^{\alpha_{1}}\|\nabla \rho\|_{L^{2}}^{\beta_{1}}\lesssim_{h_{1}} \|\nabla \rho\|_{L^{2}}^{\beta_{1}}.
     \endaligned
     \end{equation}
 Here only the f{}irst inequality  requires the  H\"older inequality and hence needs the following conditions:
  $$
         \frac{p+1}{\alpha_{1}s}>1,\  \frac{p+1}{\beta_{1}s}>1,\ \frac{\alpha_{1}s}{p+1}+\frac{\beta_{1}s}{p+1}=1.
     $$
 Obviously, the last one is equivalent to $s=(p+1)/(\tilde{p}+1)=(2N-\mu)/2N>1$, thus the f{}irst two conditions are equivalent to $\tilde{p}+1>\max(\alpha_{1},\beta_{1})$. Then the relation $1/s+1/t=2(\tilde{p}+1)/(p+1)$ implies that $t=(p+1)/(\tilde{p}+1)=s$. Applying the H\"older inequality with  the exponent $(p+1)/\alpha_{2}t=(\tilde{p}+1)/\alpha_{2}$ and the Sobolev inequality \eqref{eq:1.2}, we get,   if $\tilde{p}+1>\max(\alpha_{2},\beta_{2})$, that
     \begin{equation}\label{eq:3.7}
     \aligned
     \|h_{2}^{\alpha_{2}}|\rho|^{\beta_{2}}\|_{L^{t}}&=\Big(\int_{\R^{N}} |h_{2}|^{\alpha_{2}t}|\rho|^{\beta_{2}t} \,dx\Big)^{\frac{1}{t}} \\
     &\leq  \Big(\int_{\R^{N}} |h_{2}|^{\alpha_{2}t\cdot\frac{p+1}{\alpha_{2}t}}\,dx\Big)^{\frac{\alpha_{2}t}{p+1}\cdot\frac{1}{t}}  \Big(\int_{\R^{N}} |\rho|^{\beta_{2}t\cdot\frac{p+1}{\beta_{2}t}} \,dx\Big)^{\frac{\beta_{2}t}{p+1}\cdot\frac{1}{t}}\\
     &=\|h_{2}\|_{L^{p+1}}^{\alpha_{2}} \|\rho\|_{L^{p+1}}^{\beta_{2}} \lesssim  \|\nabla h_{2}\|_{L^{2}}^{\alpha_{1}}\|\nabla \rho\|_{L^{2}}^{\beta_{2}}\lesssim_{h_{2}} \|\nabla \rho\|_{L^{2}}^{\beta_{2}}.
     \endaligned
     \end{equation}
Further, by combining \eqref{eq:3.6} and \eqref{eq:3.7}, we know that
   $$
        \int_{\R^{N}} (I_\mu\ast h_{1}^{\alpha_{1}}|\rho|^{\beta_{1}} ) h_{2}^{\alpha_{2}}|\rho|^{\beta_{2}} \,dx \lesssim_{h_{1},h_{2}} \|\nabla \rho\|_{L^{2}}^{\beta_{1}+\beta_{2}},
    $$
whenever  $\tilde{p}+1>\max(\alpha_{1},\beta_{1})$ and $\tilde{p}+1>\max(\alpha_{2},\beta_{2})$.
    \end{proof}
    \end{lemma}
\begin{remark}\label{rmk1-double-integral-estimate}
The above Lemma also holds for $\alpha_{i},\beta_{i}\geq 0$ satisfying $\alpha_{i}+\beta_{i}=\tilde{p}+1$. In fact, if  $\alpha_{1}=0$ or $\beta_{1}=0$, then we may take $s=(p+1)/(\tilde{p}+1)>1$, it follows that  the f{}irst term $\|h_{1}^{\alpha_{1}}|\rho|^{\beta_{1}}\|_{L^{s}}$ in the right hand side of \eqref{eq:3.5}  is equal to $\|h_{1}\|_{L^{p+1}}^{\alpha_{1}}\|\rho\|_{L^{p+1}}^{\beta_{1}}$, which is bounded by $\|\nabla\rho\|_{L^{2}}^{\beta_{1}}$ due to the  Sobolev inequality. It is worth mentioning that the domination of this term in such case only needs the precondition  $\tilde{p}+1>0\iff \mu<2N$, which is always true since our assumption  is $\mu\in(0,N)$. Similarly, the second term can be controlled   by $\|\nabla\rho\|_{L^{2}}^{\beta_{2}}$ if $\alpha_{2}=0$ or $\beta_{2}=0$, with the same precondition $\tilde{p}+1>0$.
 \end{remark}

\vskip0.12in

\begin{remark}\label{rmk2-double-integral-estimate}
In particular, if we  let $\sigma\coloneqq\sum_{i=1}^{\nu}\alpha_{i}\tilde{U}[z_{i},\lambda_{i}]$,   then there holds
$$
\int_{\R^N}(I_\mu\ast\sigma^{\tilde{p}+1-m}|\rho|^{m})\sigma^{\tilde{p}+1-n}|\rho|^{n}  \,dx \lesssim \|\nabla\rho\|_{L^2}^{m+n},
 $$
for  $ m,n=0,1,2,3, ...$ satisfying
 \begin{equation}\label{eq:3.8}
     \tilde{p}+1>m \ \ \text{and} \ \ \tilde{p}+1>n.
 \end{equation}
On the other hand, if $m=\tilde{p}+1$ or $n=\tilde{p}+1$, then   the corresponding  condition reduces to $\tilde{p}+1>0$. Here we list an useful table, which translates  the conditions on $\tilde{p}$ to the equivalent conditions on $\mu$.\\
\begin{table}[h]
  \centering
  \begin{tabular}{c@{$\ \ \iff  \ \ $}l@{ }l p{10em}}
  \hline
  $\tilde{p}+1>0$   & $\mu<2N$  \\
  \hline
  $\tilde{p}+1>1$  &       $\mu<N+2$   \\
  \hline
  $\tilde{p}+1>2$    &     $\mu<4$ \\
  \hline
  $\tilde{p}+1>3$    &     $\mu<6-N$ \\
  \hline
\end{tabular}
    \caption{\small Equivalent condition between $\tilde{p}$ with $\mu$}
    \label{tab:Equivalent with-p wan-mu}
\end{table}
    \end{remark}

\begin{remark}\label{rmk3-double-integral-estimate}
     In the proof, the reason of unique choice $s=t$ in the using of HLS inequality, is that the inside exponent of convolution and the outside exponent of convolution are all equals to $\tilde{p}+1$. In the rest of this paper, we will directly use $s=t=(p+1)/(\tilde{p}+1)$ when encounter similar terms. If one of these exponent are not $\tilde{p}+1$, will cause dif{}f{}iculty, which need further steps to overcome.
\end{remark}

    \vskip0.08in
    We now turn to prove  the interaction integral estimate.  Denote $\tilde{U}_i\coloneqq\tilde{U}[z_i,\lambda_i]$, $1\le i\le\nu$.
     \begin{lemma}\label{prop:interaction_and_coef}
      Let $3\leq N\leq 5$, $(N+2)/2<\mu<\min(N,4)$ and $\nu\in\N^{+}$. For any $\hat\varepsilon>0$ there exists $\delta=\delta(N,\mu,\nu,\hat\varepsilon)>0$
        such that the following statement holds.
        Let $u=\sigma+\rho\coloneqq\sum\limits_{i=1}^\nu\alpha_i\tilde{U}_i+\rho$, where the family $\{(\alpha_i,\tilde{U}_i)\}_{1\le i\le\nu}$
        is $\delta$-interacting, and $\rho$ has the bound $\|\nabla\rho\|_{L^2}\le 1$. Assume that $\rho$ is orthogonal to  $\tilde{U}_{i}$ , $\partial_\lambda\tilde{U}_{i}$ and $\partial_{z_j}\tilde{U}_{i}$ in $\mathcal{D}^{1,2}(\R^{N})$, quivalently, $\rho$ satisf{}ies the orthogonality conditions:
 \begin{equation}\label{eq:3.9}
 \int_{\R^N}(I_\mu\ast\tilde{U_{i}}^{\tilde{p}+1} )\tilde{U_{i}}^{\tilde{p}}\rho\,dx=0,
 \end{equation}
 \begin{equation} \label{eq:3.10}
 \int_{\R^N}\big[(\tilde{p}+1)(I_{\mu}\ast\tilde{U}_{i}^{\tilde{p}}\partial_\lambda\tilde{U}_{i})\tilde{U}_{i}^{\tilde{p}}+\tilde{p}(I_{\mu}\ast\tilde{U}_{i}^{\tilde{p}+1})\tilde{U}_{i}^{\tilde{p}-1}\partial_\lambda\tilde{U}_{i}\big]\, \rho\,dx= 0,
 \end{equation}
 $$
 \int_{\R^N}\big[  (\tilde{p}+1)(I_{\mu}\ast\tilde{U}_{i}^{\tilde{p}}\partial_{z_j}\tilde{U}_{i})\tilde{U}_{i}^{\tilde{p}}+\tilde{p}(I_{\mu}\ast\tilde{U}_{i}^{\tilde{p}+1})\tilde{U}_{i}^{\tilde{p}-1}\partial_{z_j}\tilde{U}_{i} \big]\, \rho\,dx= 0 \quad\text {for any}\ \ 1\le j\le N.
 $$
 Then, for any $1\le i\le \nu$, it holds
        \begin{equation}\label{eq:3.11}
          |\alpha_i-1|\lesssim \hat\varepsilon\|\nabla\rho\|_{L^2}  +\|\nabla \rho\|_{L^{2}}^{\min(\tilde{p}, 2) }
          +\big\|\Delta u +(I_\mu\ast|u|^{\tilde{p}+1})|u|^{\tilde{p}-1}u
          \big\|_{(\mathcal{D}^{1,2})^{-1}}.
        \end{equation}
Further,  for any pair of indices $i\not=j$, it holds
        \begin{equation}\label{eq:3.12}
        \aligned   \int_{\R^N} \tilde{U}_i^{p}\tilde{U}_j
        \,dx&\lesssim \hat\varepsilon\|\nabla\rho\|_{L^2}  +\|\nabla \rho\|_{L^{2}}^{\min(\tilde{p}, 2) }
          +\big\|\Delta u +(I_\mu\ast|u|^{\tilde{p}+1})|u|^{\tilde{p}-1}u
          \big\|_{(\mathcal{D}^{1,2})^{-1}}.  \endaligned
        \end{equation}
      \end{lemma}

      \begin{proof}
       Let $\Theta(u)\coloneqq \big\|\Delta u +(I_\mu\ast|u|^{\tilde{p}+1})|u|^{\tilde{p}-1}u
       \big\|_{(\mathcal{D}^{1,2})^{-1}}$.  We denote by $o(1)$ the quantity converging to zero as  the parameter $\delta$ goes to zero. By the proof of Lemma \ref{lem:localization} and   Remark \ref{power of epilson}, $\varepsilon=o(1)$. Similarly the notation $o(E)$  means $o(E)/E=o(1) \;( as\, \delta \to 0)$. The notation $\{f\}_{g}$ means  that the term $f$ appears  when  the condition ``$g$'' is satisf{}ied.
      \vskip0.18in
     Without loss of generality,  we assume $\{\lambda_i\}$ is  decreasing
       (i.e., $\tilde{U}_1$ is the most concentrated bubble). We assume $\delta<\frac{1}{2}$.  By def{}inition of $\delta$-interacting, $\max\limits_{1\leq i\leq \nu}|\alpha_i-1|\leq \frac{1}{2}$, namely $\alpha_i\in[\frac{1}{2}, \frac{3}{2}]$ for $1\leq i\leq \nu$. Observe that the left hand side of  $\eqref{eq:3.11}$ and $\eqref{eq:3.12}$ is bounded from above,we can assume $\Theta(u)\leq 1$.
\vskip0.18in
We will prove the estimates \eqref{eq:3.11}-\eqref{eq:3.12} by induction on the index $i=1,\ldots,\nu$ (starting from the most concentrated bubble). The proof consists of the following two steps.
In the f{}irst step, we show the result holds for $i=1$. Then we f{}ix an index $i$ ($1< i\le\nu$) and suppose the result holds for all index strictly less than $i$, we proceed to show it is true  for index $i$. However, the proofs of the two steps are the same  (see  \eqref{eq:3.29} below  for the splitting of $\tilde{V}$), therefore we deal them simultaneous.

\vskip0.18in
Fix $i\in\{1,\cdots,\nu\}$.  Since all the  quantities involved, such as  $|\alpha_{i}-1|,\int_{\R^N}\tilde{U}_i^{p}\tilde{U}_j\,dx,\|\nabla\rho\|_{L^2},
        \big\|\Delta u +(I_\mu\ast|u|^{\tilde{p}+1})|u|^{\tilde{p}-1}u
       \big\|_{\left(\mathcal{D}^{1,2}\right)^{-1}}$, are invariant
       under the action of the symmetries described in \eqref{eq:2.8}, so we can assume $\tilde{U}_i=\tilde{U}=\tilde{U}[0,1]$, namely $z_{i}=0$ and $\lambda_{i}=1$.  For notational simplicity we denote $$\alpha=\alpha_i\, ,\ \tilde{V}=\sum\limits_{j\not=i}\alpha_j \tilde{U}_j.$$
       Let $\Phi=\Phi_i$ be the bump function built in Lemma \ref{lem:localization} for a f{}ixed $\varepsilon>0$. Let $f\coloneqq -\Delta u-(I_{\mu}\ast |u|^{\tilde{p}+1})|u|^{\tilde{p}-1}u$, then since $u=\sigma+\rho$, thus $$-\Delta
       (\sigma+\rho)=(I_{\mu}\ast |\sigma+\rho|^{\tilde{p}+1})|\sigma+\rho|^{\tilde{p}-1}(\sigma+\rho)+f.$$ By the def{}inition of $ \sigma$ and the equation \eqref{eq:2.1} of $\tilde{U}_{i}$, we obtain that
      $$
           -\Delta \rho=  (I_{\mu}\ast | \sigma+\rho|^{\tilde{p}+1})|\sigma+\rho|^{\tilde{p}-1}(\sigma+\rho)-\sum_{j=1}^{\nu} \alpha_{j} (I_{\mu}\ast \tilde{U}_{j}^{\tilde{p}+1})\tilde{U}_{j}^{\tilde{p}} +f.
       $$
       Next we consider the linearized equation at $\alpha \tilde{U}$,
     $$
           \aligned
           &-\Delta \rho-(\tilde{p}+1)(I_{\mu}\ast(\alpha \tilde{U})^{\tilde{p}}\rho)(\alpha \tilde{U})^{\tilde{p}} - \tilde{p}\,(I_{\mu}\ast(\alpha \tilde{U})^{\tilde{p}+1})(\alpha \tilde{U})^{\tilde{p}-1}\rho \\
           &=(I_{\mu}\ast | \sigma+\rho|^{\tilde{p}+1})|\sigma+\rho|^{\tilde{p}-1}(\sigma+\rho)-\sum_{j=1}^{\nu} \alpha_{j} (I_{\mu}\ast \tilde{U}_{j}^{\tilde{p}+1})\tilde{U}_{j}^{\tilde{p}}\\
            &\quad - (\tilde{p}+1)(I_{\mu}\ast(\alpha \tilde{U})^{\tilde{p}}\rho)(\alpha \tilde{U})^{\tilde{p}} - \tilde{p}\,(I_{\mu}\ast(\alpha \tilde{U})^{\tilde{p}+1})(\alpha \tilde{U})^{\tilde{p}-1}\rho+f \\
           &\eqqcolon  \mathcal{A}_{1}+\mathcal{A}_{2}+\mathcal{A}_{3}+\mathcal{A}_{4}+\mathcal{A}_{5}.
           \endaligned
     $$
Then
\begin{equation}\label{eq:3.13}
-\Delta \rho + \mathcal{A}_{3} + \mathcal{A}_{4} = \mathcal{A}_{1}+\mathcal{A}_{2}+\mathcal{A}_{3}+\mathcal{A}_{4}+\mathcal{A}_{5}.
\end{equation}
Let $\xi$ be $\tilde{U}$ or $ \partial_{\lambda}\tilde{U}$, here $\partial_{\lambda}=\partial_{\lambda}\big|_{\lambda=1}$.  Then by \eqref{eq:2.6}, there holds $|\xi|\lesssim \tilde{U} $. We will test the above equation with $\xi \Phi$, and then estimate each term. Our purpose is that (see the right hand sides of \eqref{eq:3.11} and \eqref{eq:3.12}), if some terms like $C\|\nabla \rho \|_{L^{2}}^{q}$ appear  in the right hand side of the estimate, then the exponent $q$ and the quantity $C$ must satisfy either $q=1$ and $C=o(1)$ or $q>1$ and $C=constant.$ Before going  further,   we need to  decompose each term. Firstly we observe that
     $$
           \aligned
           \mathcal{A}_{1} &  =\big(I_{\mu}\ast | \sigma+\rho|^{\tilde{p}+1}\big)|\sigma+\rho|^{\tilde{p}-1}(\sigma+\rho)  \\
           & =\big(I_{\mu}\ast | \sigma+\rho|^{\tilde{p}+1}\big) \big(|\sigma+\rho|^{\tilde{p}-1}(\sigma+\rho)-\sigma^{\tilde{p}}-\tilde{p} \sigma^{\tilde{p}-1}\rho \big)\\
           &\quad  + \Big(I_{\mu}\ast \big(| \sigma+\rho|^{\tilde{p}+1}-\sigma^{\tilde{p}+1}-(\tilde{p}+1)\sigma^{\tilde{p}}\rho\big) \Big) \big(  \sigma^{\tilde{p}}+\tilde{p} \sigma^{\tilde{p}-1}\rho  \big)\\
           & \quad + \Big(I_{\mu}\ast \big(\sigma^{\tilde{p}+1}+(\tilde{p}+1)\sigma^{\tilde{p}}\rho\big) \Big) \big(  \sigma^{\tilde{p}}+\tilde{p} \sigma^{\tilde{p}-1}\rho  \big)\\
           & \eqqcolon  \mathcal{B}_{1}+\mathcal{B}_{2}+\mathcal{B}_{3},
           \endaligned
       $$
       then using  the Proposition \ref{propos-1} with $r=\tilde{p}$   (here we need that $\tilde{p}\geq 1\iff\mu\leq 4$ ), $l=0$, $a=\sigma, b=\rho$, and the Lemma \ref{double-integral-estimate},  we know
      $$
       \aligned
           \Big|\int_{\mathbb{R}^{N}}\mathcal{B}_{1}\cdot \xi \Phi \,dx \Big|&\lesssim \int_{\mathbb{R}^{N}}  \big(I_{\mu}\ast | \sigma+\rho|^{\tilde{p}+1}\big) \big( |\rho|^{\tilde{p}} + \{\sigma^{\tilde{p}-2}|\rho|^{2}  \}_{\tilde{p}>2} \big)  |\xi|\Phi      \,dx\\
           &\lesssim \int_{\mathbb{R}^{N}}  \big(I_{\mu}\ast | \sigma+\rho|^{\tilde{p}+1}\big) |\rho|^{\tilde{p}}\tilde{U}    \,dx + \int_{\mathbb{R}^{N}}  \big(I_{\mu}\ast | \sigma+\rho|^{\tilde{p}+1}\big) \{\sigma^{\tilde{p}-1}|\rho|^{2}  \}_{\tilde{p}>2}   \,dx\\
           &\lesssim \|| \sigma+\rho|^{\tilde{p}+1}\|_{L^{\frac{p+1}{\tilde{p}+1}}} \||\rho|^{\tilde{p}}\tilde{U} \|_{L^{\frac{p+1}{\tilde{p}+1}}}   +\big\{\|| \sigma+\rho|^{\tilde{p}+1} \|_{L^{\frac{p+1}{\tilde{p}+1}}}\|\sigma^{\tilde{p}-1}|\rho|^{2} \|_{L^{\frac{p+1}{\tilde{p}+1}}} \big\}_{\tilde{p}>2}\\
           &\lesssim \|\sigma+\rho\|_{L^{p+1}}\|\nabla \rho\|_{L^{2}}^{\tilde{p}}\|\tilde{U}\|_{L^{p+1}}  + \big\{\| \sigma+\rho \|_{L^{p+1}}  \| \sigma\|_{L^{p+1}}^{\tilde{p}-1}\|\nabla \rho\|^{2}_{L^{2}}  \big\}_{\tilde{p}>2},
           \endaligned
       $$
       where the second term in the last inequality, we have used the H\"older's inequality which
       requires conditions \eqref{eq:3.8} and   $\tilde{p}+1>2\iff \mu <4 $ in the Table \ref{tab:Equivalent with-p wan-mu}). However the second term appears only when $\tilde{p}>2(\iff \mu<6-N)$ holds, which is a stronger condition since
       $6-N<4$ as $N\geq3$. Then by using the Sobolev inequality, the def{}inition of $\sigma$ and the condition $\alpha_{i}\lesssim 1$, $\|\nabla \rho\|_{L^{2}}\leq 1$, we know  that $ \|\sigma+\rho\|_{L^{p+1}}\lesssim \|\nabla(\sigma+\rho)\|_{L^{2}}\lesssim 1$, it follows that
       \begin{equation}\label{eq:3.14}
       \aligned
          \Big|\int_{\mathbb{R}^{N}}\mathcal{B}_{1}\cdot \xi \Phi \,dx \Big|\lesssim
          \|\nabla \rho\|_{L^{2}}^{\tilde{p}} + \big\{\|\nabla \rho\|^{2}_{L^{2}}\big\}_{\tilde{p}>2} \lesssim \|\nabla \rho\|_{L^{2}}^{\min(\tilde{p}, 2) },
          \endaligned
       \end{equation}
here we need $\tilde{p}>1 (\iff \mu <4)$. Using  the Proposition \ref{propos-1} with $r=\tilde{p}+1$ (it requires that $\tilde{p}+1\geq 1 (\iff\mu\leq N+2$) ), $l=0$, $a=\sigma, b=\rho$, the fact $|\xi|\lesssim \tilde{U} \lesssim \sigma$ and the Lemma \ref{double-integral-estimate}, we have the following estimates:
\begin{equation}\label{eq:3.15}
       \aligned
           \Big|\int_{\mathbb{R}^{N}}\mathcal{B}_{2}\cdot \xi \Phi \,dx \Big|&\lesssim \int_{\mathbb{R}^{N}}  \big(I_{\mu}\ast (|\rho|^{\tilde{p}+1} +  \{\sigma^{\tilde{p}-1}|\rho|^{2}  \}_{\tilde{p}+1>2}) \big) \big(\sigma^{\tilde{p}}+\tilde{p} \sigma^{\tilde{p}-1}|\rho|     \big)  |\xi|\Phi      \,dx\\
           &\lesssim \int_{\mathbb{R}^{N}}  \big(I_{\mu}\ast |\rho|^{\tilde{p}+1} \big) \sigma^{\tilde{p}+1}   \,dx + \int_{\mathbb{R}^{N}}  \big(I_{\mu}\ast |\rho|^{\tilde{p}+1} \big)   \sigma^{\tilde{p}}|\rho|    \,dx\\
           &\ \ \ \ +\int_{\mathbb{R}^{N}}  \big(I_{\mu}\ast \{\sigma^{\tilde{p}-1}|\rho|^{2}  \}_{\tilde{p}+1>2} \big) \sigma^{\tilde{p}+1}   \,dx  + \int_{\mathbb{R}^{N}}  \big(I_{\mu}\ast  \{\sigma^{\tilde{p}-1}|\rho|^{2}  \}_{\tilde{p}+1>2} \big)  \sigma^{\tilde{p}}|\rho|   \,dx\\
           &\lesssim \|\nabla \rho\|_{L^{2}}^{\tilde{p}+1} +\|\nabla \rho\|_{L^{2}}^{\tilde{p}+2}+\big\{\|\nabla \rho\|^{2}_{L^{2}}\big\}_{ \tilde{p}+1>2 }+ \big\{\|\nabla \rho\|^{3}_{L^{2}}\big\}_{\tilde{p}+1>2} \\
           &\lesssim \|\nabla \rho\|_{L^{2}}^{\min(\tilde{p}+1,2)},
           \endaligned
          \end{equation}
here we need $\tilde{p}+1>1$ i.e., $ \mu <N+2$, which   always holds since $\mu\in(0,N)$. Observe that
      $$
           \aligned
           \mathcal{B}_{3}&=\Big(I_{\mu}\ast \big(\sigma^{\tilde{p}+1}+(\tilde{p}+1)\sigma^{\tilde{p}}\rho\big) \Big) \big(  \sigma^{\tilde{p}}+\tilde{p} \sigma^{\tilde{p}-1}\rho  \big) \\
           &=\tilde{p}(\tilde{p}+1)\big(I_{\mu}\ast \sigma^{\tilde{p}}\rho\big) \sigma^{\tilde{p}-1}\rho + (\tilde{p}+1)\big(I_{\mu}\ast \sigma^{\tilde{p}}\rho \big)\sigma^{\tilde{p}}
           + \tilde{p} \big(I_{\mu}\ast \sigma^{\tilde{p}+1}\big) \sigma^{\tilde{p}-1}\rho
           + \big(I_{\mu}\ast \sigma^{\tilde{p}+1} \big)\sigma^{\tilde{p}}\\
           &\eqqcolon \mathcal{C}_{1}+\mathcal{C}_{2}+\mathcal{C}_{3}+\mathcal{C}_{4},
           \endaligned
       $$
       thus
       \begin{equation}\label{eq:3.16}
           \Big|\int_{\mathbb{R}^{N}}\mathcal{C}_{1}\cdot \xi \Phi \,dx\Big|\lesssim  \int_{\mathbb{R}^{N}} \big(I_{\mu}\ast \sigma^{\tilde{p}}|\rho|\big) \sigma^{\tilde{p}-1}|\rho| \sigma \,dx  \lesssim \|\nabla \rho\|_{L^{2}}^{2}.
       \end{equation}
       The second term $\mathcal{C}_{2}$ and the third term $\mathcal{C}_{3}$ will be combined with $\mathcal{A}_{3},\mathcal{A}_{4}$ respectively, they will be evaluated  later. Next we  estimate the fourth term $\mathcal{C}_{4}$,
$$
           \aligned
           \mathcal{C}_{4}=\big(I_{\mu}\ast \sigma^{\tilde{p}+1} \big)\sigma^{\tilde{p}}&= \big(I_{\mu}\ast \sigma^{\tilde{p}+1}\big) \big(  \sigma^{\tilde{p}}-(\alpha \tilde{U})^{\tilde{p}}-\tilde{p}(\alpha \tilde{U})^{\tilde{p}-1} \tilde{V}\big)\\
           & \quad + \Big( I_{\mu}\ast \big( \sigma^{\tilde{p}+1} -(\alpha\tilde{U})^{\tilde{p}+1}-(\tilde{p}+1)(\alpha\tilde{U})^{\tilde{p}} \tilde{V}\big)   \Big) (\alpha\tilde{U})^{\tilde{p}}\\
           & \quad + (\tilde{p}+1) \big(I_{\mu}\ast (\alpha\tilde{U})^{\tilde{p}} \tilde{V}\big) (\alpha\tilde{U})^{\tilde{p}}  \\
           & \quad + \tilde{p}\ ( I_{\mu}\ast \big( \sigma^{\tilde{p}+1}- (\alpha\tilde{U})^{\tilde{p}+1}  \big)  (\alpha\tilde{U})^{\tilde{p}-1}\tilde{V}    \\
           & \quad +  \big( I_{\mu}\ast  (\alpha\tilde{U})^{\tilde{p}+1} \big) (\alpha\tilde{U})^{\tilde{p}} +\tilde{p}  \big( I_{\mu}\ast  (\alpha\tilde{U})^{\tilde{p}+1} \big) (\alpha\tilde{U})^{\tilde{p}-1}  \tilde{V}   \\
           &\eqqcolon \mathcal{D}_{1}+\mathcal{D}_{2}+\mathcal{D}_{3}+\mathcal{D}_{4}+\mathcal{D}_{5}.
           \endaligned
       $$
Now we proceed to estimate the  above four terms $\mathcal{D}_{1}$-$\mathcal{D}_{4}$ in the region $\{\Phi>0\}$.
       \vskip0.15in
      \noindent$\bullet$ For $\mathcal{D}_{1}$, since in the region $\{\Phi>0\}$, $\tilde{U}_{j}< \varepsilon\tilde{U}$ for any $j\neq i$, thus $\tilde{V}\lesssim \varepsilon\tilde{U}$. If $\tilde{p}\geq 1 (\iff \mu\leq 4)$, it follows from \eqref{eq:2.13} that
      $$
        \aligned
             \frac{|  \sigma^{\tilde{p}}-(\alpha \tilde{U})^{\tilde{p}}-\tilde{p}(\alpha \tilde{U})^{\tilde{p}-1} \tilde{V}|}{\tilde{U}^{\tilde{p}-1}\tilde{V}}  &=\frac{| (\alpha \tilde{U}+\tilde{V})^{\tilde{p}}-(\alpha \tilde{U})^{\tilde{p}}-\tilde{p}(\alpha \tilde{U})^{\tilde{p}-1} \tilde{V}|}{\tilde{U}^{\tilde{p}-1}\tilde{V}}\lesssim \frac{| \tilde{V} |^{\tilde{p}}+\{|\tilde{U}|^{\tilde{p}-2}\tilde{V}^{2}  \}_{\tilde{p}>2}}{\tilde{U}^{\tilde{p}-1}\tilde{V}}\\
            \ &\lesssim \Big|\frac{\tilde{V}}{\tilde{U}}\Big|^{\tilde{p}-1}+
            \Big\{ \Big|\frac{\tilde{V}}{\tilde{U}}\Big|
            \Big\}_{\tilde{p}>2} \lesssim \ \varepsilon ^{\min(\tilde{p}-1,1)},
         \endaligned
       $$
       where the last inequality need $\tilde{p}-1>0 (\iff\mu<4)$. Then by the relation \eqref{eq:2.2}, we have
     $$
      \aligned
          \big|\mathcal{D}_{1}\big|  &= \big|\big(I_{\mu}\ast \sigma^{\tilde{p}+1}\big) \big(  \sigma^{\tilde{p}}-(\alpha \tilde{U})^{\tilde{p}}-\tilde{p}(\alpha \tilde{U})^{\tilde{p}-1} \tilde{V}\big) \big|   \lesssim \varepsilon ^{\min(\tilde{p}-1 , 1)} \big(I_{\mu}\ast \sigma^{\tilde{p}+1}\big)\, \tilde{U}^{\tilde{p}-1}\tilde{V}  \\
          &\lesssim \varepsilon ^{\min(\tilde{p}-1,1)} \bigg[(I_{\mu}\ast \tilde{U}^{\tilde{p}+1})\,\tilde{U}^{\tilde{p}-1}   \tilde{V}  +\sum_{j\neq i} (I_{\mu}\ast \tilde{U}_{j}^{\tilde{p}+1})\, \tilde{U}^{\tilde{p}-1}  \tilde{V}  \bigg]\\
          & \approx \varepsilon ^{\min(\tilde{p}-1,1)} \bigg[\tilde{U}^{p-1}  \tilde{V}  +\sum_{j\neq i} \tilde{U}_{j}^{p-\tilde{p}} \tilde{U}^{\tilde{p}-1}  \tilde{V} \bigg] \\
          &\lesssim  \varepsilon ^{\min(\tilde{p}-1,1)} \bigg[\tilde{U}^{p-1}  \tilde{V}  +\sum_{j\neq i}( \varepsilon \tilde{U})^{p-\tilde{p}} \tilde{U}^{\tilde{p}-1}  \tilde{V} \bigg]\\
          &\lesssim\varepsilon ^{\min(\tilde{p}-1,1)} (1+\varepsilon^{p-\tilde{p}})\ \tilde{U}^{p-1}  \tilde{V} ,
      \endaligned
       $$
       i.e., $\mathcal{D}_{1}=o(1)\tilde{U}^{p-1}\tilde{V}$ in the region $\{\Phi>0\}$ if $\mu<4$.
       \vskip0.16in
      \noindent$\bullet$ For $\mathcal{D}_{2}$, since in the region $\{\Phi>0\}$, $\tilde{U}_{j}< \varepsilon\tilde{U}$ for any $j\neq i$, thus $\{\Phi>0\}\subseteq \{\tilde{V}< \varepsilon(\nu-1)\tilde{U} \}\subseteq \{\tilde{V}< \varepsilon\nu\tilde{U} \} $ and  the set $\{\tilde{V}< \varepsilon\nu\tilde{U} \}$ is nonempty. Therefore,
     $$
      \aligned
          \big|\mathcal{D}_{2} \big|&= \Big|\Big( I_{\mu}\ast \big( \sigma^{\tilde{p}+1} -(\alpha\tilde{U})^{\tilde{p}+1}-(\tilde{p}+1)(\alpha\tilde{U})^{\tilde{p}} \tilde{V}\big)   \Big) (\alpha\tilde{U})^{\tilde{p}} \Big|   \\
          &\lesssim \Big( I_{\mu}\ast \big| \sigma^{\tilde{p}+1} -(\alpha\tilde{U})^{\tilde{p}+1}-(\tilde{p}+1)(\alpha\tilde{U})^{\tilde{p}} \tilde{V}\big|   \Big) (\alpha\tilde{U})^{\tilde{p}}\\
          &\lesssim  \Big( I_{\mu}\ast \big( \tilde{V}^{\tilde{p}+1} +\{ (\alpha\tilde{U})^{\tilde{p}-1} \tilde{V}^{2}\}_{\tilde{p}+1>2} \big)   \Big) (\alpha\tilde{U})^{\tilde{p}}    \\
          &= \Big( I_{\mu}\ast \mathds{1}_{\{\tilde{V}< \varepsilon\nu\tilde{U} \}}\big( \tilde{V}^{\tilde{p}+1} +\{ (\alpha\tilde{U})^{\tilde{p}-1} \tilde{V}^{2}\}_{\tilde{p}+1>2} \big)   \Big) (\alpha\tilde{U})^{\tilde{p}}   \\
          &\ + \Big( I_{\mu}\ast \mathds{1}_{\{\tilde{V}\geq \varepsilon\nu\tilde{U} \}}\big( \tilde{V}^{\tilde{p}+1} +\{ (\alpha\tilde{U})^{\tilde{p}-1} \tilde{V}^{2}\}_{\tilde{p}+1>2} \big)   \Big) (\alpha\tilde{U})^{\tilde{p}},
      \endaligned
      $$
       where the f{}irst term
      $$
      \aligned\ & \Big( I_{\mu}\ast \mathds{1}_{\{\tilde{V}< \varepsilon\nu\tilde{U} \}}\big( \tilde{V}^{\tilde{p}+1} +\{ (\alpha\tilde{U})^{\tilde{p}-1} \tilde{V}^{2}\}_{\tilde{p}+1>2} \big)   \Big) (\alpha\tilde{U})^{\tilde{p}} \\
         \lesssim \ &   \Big( I_{\mu}\ast  \varepsilon^{\min(\tilde{p}+1, 2)}\tilde{U}^{\tilde{p}+1} \big)   \Big) (\alpha\tilde{U})^{\tilde{p}}
       \approx    \varepsilon^{\min(\tilde{p}+1, 2)} \tilde{U}^{p-\tilde{p}} \tilde{U}^{\tilde{p}}    \approx \varepsilon^{\min(\tilde{p}+1, 2)}  \tilde{U}^{p},
      \endaligned
 $$
     and   the second term
  $$
      \aligned
        & \Big( I_{\mu}\ast \mathds{1}_{\{\tilde{V}\geq \varepsilon\nu\tilde{U} \}}\big( \tilde{V}^{\tilde{p}+1} +\{ (\alpha\tilde{U})^{\tilde{p}-1} \tilde{V}^{2}\}_{\tilde{p}+1>2} \big)   \Big) (\alpha\tilde{U})^{\tilde{p}}    \\
        \lesssim \ & \Big( I_{\mu}\ast \big( \tilde{V}^{\tilde{p}+1} +\Big\{ \Big(\frac{1}{\varepsilon}\Big)^{\tilde{p}-1} \tilde{V}^{\tilde{p}-1}\tilde{V}^{2}\Big\}_{\tilde{p}+1>2} \big)   \Big) (\alpha\tilde{U})^{\tilde{p}}    \\
        \lesssim \ & \Big(1+\Big\{\Big(\frac{1}{\varepsilon}\Big)^{\tilde{p}}\Big\}_{\tilde{p}+1>2}\Big)\big( I_{\mu}\ast  \tilde{V}^{\tilde{p}+1}  \big) (\alpha\tilde{U})^{\tilde{p}}    \\
        \lesssim \ & \Big(1+\Big\{\Big(\frac{1}{\varepsilon}\Big)^{\tilde{p}}\Big\}_{\tilde{p}+1>2}\Big) \sum_{j\neq i}\big( I_{\mu}\ast  \tilde{U}_{j}^{\tilde{p}+1}  \big) \tilde{U}^{\tilde{p}}    \\
        \approx \ & \Big(1+\Big\{\Big(\frac{1}{\varepsilon}\Big)^{\tilde{p}}\Big\}_{\tilde{p}+1>2}\Big) \sum_{j\neq i}  \tilde{U}_{j}^{p-\tilde{p}}   \tilde{U}^{\tilde{p}}    \\
        \lesssim \ & \Big(1+\Big\{\Big(\frac{1}{\varepsilon}\Big)^{\tilde{p}}\Big\}_{\tilde{p}+1>2}\Big) \sum_{j\neq i}  \varepsilon^{p-\tilde{p}}\tilde{U}^{p-\tilde{p}}   \tilde{U}^{\tilde{p}}     \lesssim \big(\varepsilon^{p-\tilde{p}}+\big\{\varepsilon^{p-2\tilde{p}}\big\}_{\tilde{p}+1>2}\big)\tilde{U}^{p}.
      \endaligned
      $$
Combining the above  two terms, we know  that $\mathcal{D}_{2}=o(1)\tilde{U}^{p}$ in the region $\{\Phi>0\}$, if
      $$
          \begin{cases}
              p-\tilde{p}>0 &\text{as \quad} \tilde{p}+1\leq2\\
              p-2\tilde{p}>0   &\text{as \quad}  \tilde{p}+1>2
          \end{cases}\iff \begin{cases}
\mu>0 &\text{as \quad}\mu\geq4\\
\mu>\frac{N+2}{2} &\text{as \quad}\mu<4
\end{cases}.
 $$
In particular, if $\frac{N+2}{2}<\mu<4$, then $\mathcal{D}_{2}=o(1)\tilde{U}^{p}$ in the region $\{\Phi>0\}$.


\vskip0.16in
\noindent$\bullet$ For $\mathcal{D}_{3}$, in the region $\{\Phi>0\}$, we have the following estimates:
 $$
 \aligned
 \mathcal{D}_{3}&=(\tilde{p}+1) \big(I_{\mu}\ast (\alpha\tilde{U})^{\tilde{p}} \tilde{V}\big) (\alpha\tilde{U})^{\tilde{p}}\\
  &=  (\tilde{p}+1) \big(I_{\mu}\ast \mathds{1}_{\{\tilde{V}< \varepsilon\nu\tilde{U} \}}(\alpha\tilde{U})^{\tilde{p}} \tilde{V}\big) (\alpha\tilde{U})^{\tilde{p}}+(\tilde{p}+1) \big(I_{\mu}\ast \mathds{1}_{\{\tilde{V}\geq \varepsilon\nu\tilde{U} \}}(\alpha\tilde{U})^{\tilde{p}} \tilde{V}\big) (\alpha\tilde{U})^{\tilde{p}} \\
   &\lesssim \varepsilon \big( I_{\mu}\ast\tilde{U}^{\tilde{p}+1} \big)\tilde{U}^{\tilde{p}}+\Big(\frac{1}{\varepsilon}\Big)^{\tilde{p}}\big(I_{\mu}\ast \tilde{V}^{\tilde{p}+1}\big)\tilde{U}^{\tilde{p}}
       \lesssim \varepsilon \big( I_{\mu}\ast\tilde{U}^{\tilde{p}+1} \big)\tilde{U}^{\tilde{p}}+\Big(\frac{1}{\varepsilon}\Big)^{\tilde{p}}\sum_{j\neq i}\big(I_{\mu}\ast \tilde{U}_{j}^{\tilde{p}+1}\big)\tilde{U}^{\tilde{p}} \\
       &\approx \varepsilon\tilde{U}^{p-\tilde{p}}  \tilde{U}^{\tilde{p}} + \Big(\frac{1}{\varepsilon}\Big)^{\tilde{p}}\sum_{j\neq i}  \tilde{U}_{j}^{p-\tilde{p}}  \tilde{U}^{\tilde{p}} \lesssim \varepsilon\tilde{U}^{p}+\Big(\frac{1}{\varepsilon}\Big)^{\tilde{p}}\sum_{j\neq i}  \varepsilon^{p-\tilde{p}}\tilde{U}^{p-\tilde{p}}  \tilde{U}^{\tilde{p}} \lesssim\big( \varepsilon+ \varepsilon^{p-2\tilde{p}} \big) \tilde{U}^{p},
      \endaligned
 $$
       thus  $\mathcal{D}_{3}=o(1)\tilde{U}^{p}$ in the region $\{\Phi>0\}$ if $p-2\tilde{p}>0\iff \mu>\frac{N+2}{2}$.
       \vskip0.06in
      \noindent
       \vskip0.16in
      \noindent$\bullet$ For $\mathcal{D}_{4}$, since $\alpha$ has a positive upper bound in the region $\{\tilde{V}\lesssim \varepsilon \tilde{U}\}$ we obviously have
     $$
           \frac{|\sigma^{\tilde{p}+1}- (\alpha\tilde{U})^{\tilde{p}+1} |}{\tilde{U}^{\tilde{p}+1}}=\Big|\big(\alpha+\frac{\tilde{V}}{\tilde{U}}\big)^{\tilde{p}+1}-\alpha^{\tilde{p}+1}\Big|\lesssim(\alpha+t\,\frac{\tilde{V}}{\tilde{U}}\big)^{\tilde{p}}\cdot\frac{\tilde{V}}{\tilde{U}}\lesssim \varepsilon,
$$    
       where $t\in[0,1]$. Since $\alpha$ also has a positive lower bound,  for any $q\in \mathbb{R}$, there holds $ |\sigma^{q}- (\alpha\tilde{U})^{q} |\lesssim_{q} \varepsilon \tilde{U}^{q}$ in the region $\{\tilde{V}\lesssim \varepsilon \tilde{U}\}$.
       On the other hand, since $\tilde{p}+1>0$, hence  in the region $\{\tilde{V}\gtrsim \varepsilon \tilde{U}\}$, we see that
  $$\aligned
          \frac{|\sigma^{\tilde{p}+1}-
          (\alpha \tilde{U})^{\tilde{p}+1} 
          |}{ \tilde{V}^{\tilde{p}+1}}
          &=\frac{(\alpha \tilde{U}+\tilde{V}
          )^{\tilde{p}+1}-(\alpha \tilde{U}
          )^{\tilde{p}+1}}{\tilde{V}^{\tilde{p}+1}}\\
          &=\Big(1+\frac{\alpha \tilde{U}}{\tilde{V}}
          \Big)^{\tilde{p}+1}-
          \Big(\frac{\alpha \tilde{U}}{\tilde{V}}
          \Big)^{\tilde{p}+1}\\
         & \lesssim \Big(1+\frac{\alpha}{
            \varepsilon\nu}\Big)^{\tilde{p}+1}\\
            &\lesssim \Big(\frac{1}{\varepsilon}\Big)^{\tilde{p}+1}.
     \endaligned $$
     Similarly, for any $q\geq0$, there holds $ |\sigma^{q}- (\alpha\tilde{U})^{q} |\lesssim_{q} (\frac{1}{\varepsilon})^{q} \tilde{V}^{q}$ in the region $\{\tilde{V}\gtrsim \varepsilon \tilde{U}\}$.
       Therefore,  similar to  the estimate of $\mathcal{D}_{3}$, in the region $\{\Phi>0\}$, we have that
$$
      \aligned
       \mathcal{D}_{4}&=\tilde{p}\ ( I_{\mu}\ast
       \big( \sigma^{\tilde{p}+1}- (\alpha\tilde{U})^{\tilde{p}+1}
        \big))
         (\alpha\tilde{U})^{\tilde{p}-1}\tilde{V}\\
       &=\tilde{p}\
       ( I_{\mu}\ast \mathds{1}_{\{\tilde{V}
       < \varepsilon\nu\tilde{U} \}})
       \big( \sigma^{\tilde{p}+1}- (\alpha\tilde{U})^{\tilde{p}+1}  \big)  (\alpha\tilde{U})^{\tilde{p}-1}\tilde{V}\\
       &\ +\tilde{p}\ ( I_{\mu}\ast 
       \mathds{1}_{\{\tilde{V}\geq
        \varepsilon\nu\tilde{U} \}})
        \big( \sigma^{\tilde{p}+1}- 
        (\alpha\tilde{U})^{\tilde{p}+1}  
        \big)  (\alpha\tilde{U}
        )^{\tilde{p}-1}\tilde{V}\\
       &\lesssim \varepsilon\big(I_{\mu}\ast  
       \tilde{U}^{\tilde{p}+1}\big)\tilde{U}^{\tilde{p}-1}\tilde{V}+ \Big(\frac{1}{\varepsilon}\Big)^{\tilde{p}+1}\big(I_{\mu}\ast  \tilde{V}^{\tilde{p}+1}\big)\tilde{U}^{\tilde{p}-1}\tilde{V}\\
       &\lesssim \varepsilon^{2}\tilde{U}^{p}+ 
       \varepsilon^{-\tilde{p}-1}
       \varepsilon^{p-\tilde{p}}\varepsilon 
       \tilde{U}^{p} \\
       &\approx (\varepsilon^{2}
       +\varepsilon^{p-2\tilde{p}})\tilde{U}^{p},
      \endaligned
  $$
       therefore, $\mathcal{D}_{4}=o(1)\tilde{U}^{p}$ in the region $\{\Phi>0\}$ whenever  $p-2\tilde{p}>0 (i.e.,\;\mu>\frac{N+2}{2})$.
       \vskip0.16in
We have shown that the four terms $\mathcal{D}_{1}$-$\mathcal{D}_{4}$ in $\mathcal{C}_{4}$ are higher order terms  which can be neglected  as long as $\frac{N+2}{2}<\mu<4$. Next we will see that  the main term is $\mathcal{D}_{5}$. Let
    $$
        \mathcal{A}_{2}^{(1)}\coloneqq -
        \alpha (I_{\mu}\ast \tilde{U}^{\tilde{p}+1})
        \tilde{U}^{\tilde{p}},\ \ \  \ \mathcal{A}_{2}^{(2)}\coloneqq - \sum_{j\neq i}\alpha_{j}(I_{\mu}\ast \tilde{U}_{j}^{\tilde{p}+1})\tilde{U}_{j}^{\tilde{p}},
     $$
       then $\mathcal{A}_{2}=\mathcal{A}_{2}^{(1)}+\mathcal{A}_{2}^{(2)}$, and
       by the relation \eqref{eq:2.2},
   $$
           \aligned
           \mathcal{D}_{5}+\mathcal{A}_{2}^{(1)}&=\big( I_{\mu}\ast  (\alpha\tilde{U})^{\tilde{p}+1} \big) (\alpha\tilde{U})^{\tilde{p}} +\tilde{p}  \big( I_{\mu}\ast  (\alpha\tilde{U})^{\tilde{p}+1} \big) (\alpha\tilde{U})^{\tilde{p}-1}  \tilde{V}-\alpha (I_{\mu}\ast \tilde{U}^{\tilde{p}+1})\tilde{U}^{\tilde{p}}\\
           &\approx \alpha^{2\tilde{p}+1}  \tilde{U}^{p-\tilde{p}}\tilde{U}^{\tilde{p}} + \tilde{p}\alpha^{2\tilde{p}}\tilde{U}^{p-\tilde{p}}\tilde{U}^{\tilde{p}-1}\tilde{V}-\alpha  \tilde{U}^{p-\tilde{p}}\tilde{U}^{\tilde{p}}\\
           &\approx (\alpha^{2\tilde{p}+1}-\alpha )  \tilde{U}^{p} + \tilde{p}\alpha^{2\tilde{p}}\tilde{U}^{p-1}\tilde{V}.
           \endaligned
   $$
      \vskip0.06in
      \noindent$\bullet$ For $\mathcal{A}_{2}^{(2)}$, in the region $\{\Phi>0\}$ we have that
 $$
           \aligned
           |\mathcal{A}_{2}^{(2)}|&= \sum_{j\neq i}\alpha_{j}(I_{\mu}\ast \tilde{U}_{j}^{\tilde{p}+1})\tilde{U}_{j}^{\tilde{p}}\approx  \tilde{U}_{j}^{p-\tilde{p}}\tilde{U}_{j}^{\tilde{p}}\lesssim (\varepsilon \tilde{U})^{p-\tilde{p}}(\varepsilon \tilde{U})^{\tilde{p}}\approx \varepsilon^{p} \tilde{U}^{p},
           \endaligned
        $$
      i.e., $\mathcal{A}_{2}^{(2)}=o(1)\tilde{U}^{p}$ in the region $\{\Phi>0\}$. Writing equation \eqref{eq:3.13} as the above decomposition, we have  the following identity
  $$
      \aligned
          -(\mathcal{D}_{5}+\mathcal{A}_{2}^{(1)})-(\mathcal{D}_{1}+\mathcal{D}_{2}+\mathcal{D}_{3}+\mathcal{D}_{4})-\mathcal{A}_{2}^{(2)} = \Delta \rho &-( \mathcal{A}_{3}+\mathcal{A}_{4})+(\mathcal{B}_{1}+\mathcal{B}_{2}+\mathcal{C}_{1})\\
          &+(\mathcal{C}_{2}+\mathcal{A}_{3})+(\mathcal{C}_{3}+\mathcal{A}_{4}) + \mathcal{A}_{5},
      \endaligned
  $$
  Multiply both sides with $\xi \Phi$, where $\xi$ is $\tilde{U}$ or $ \partial_{\lambda}\tilde{U}$  and $\partial_{\lambda}=\partial_{\lambda}\big|_{\lambda=1}$, then integral  by parts we have that
      \begin{equation}\label{eq:3.17}
      \aligned
      &\Big|\int_{\mathbb{R}^{N}}  \Big[(\alpha-\alpha^{2\tilde{p}+1}+o(1) )  \tilde{U}^{p} - (\tilde{p}\alpha^{2\tilde{p}}+o(1))\tilde{U}^{p-1}\tilde{V} \Big]  \xi \Phi\,dx \Big| \\
       & \lesssim  \Big| \int_{\mathbb{R}^{N}} \nabla \rho \cdot \nabla (\xi\Phi)  \,dx \Big|+\Big|
      \int_{\mathbb{R}^{N}}  (\mathcal{A}_{3}+\mathcal{A}_{4}) \cdot \xi\Phi \,dx\Big|  \\
       &\quad +   \Big| \int_{\mathbb{R}^{N}} \mathcal{B}_{1} \cdot \xi\Phi  \,dx \Big| +  \Big| \int_{\mathbb{R}^{N}} \mathcal{B}_{2}  \cdot \xi\Phi  \,dx \Big|  + \Big| \int_{\mathbb{R}^{N}} \mathcal{C}_{1} \cdot \xi\Phi  \,dx \Big| \\
       &\quad +     \Big| \int_{\mathbb{R}^{N}} (\mathcal{C}_{2}+ \mathcal{A}_{3}) \cdot\xi\Phi  \,dx \Big| +  \Big| \int_{\mathbb{R}^{N}} (\mathcal{C}_{3}+ \mathcal{A}_{4}) \cdot\xi\Phi  \,dx \Big| +  \Big| \int_{\mathbb{R}^{N}}  \mathcal{A}_{5}\cdot \xi\Phi  \,dx \Big| \\
       &\eqqcolon  \mathcal{E}_{1} + \cdot\cdot\cdot    + \mathcal{E}_{8}. \\
      \endaligned
      \end{equation}
      Next we estimate each term $\mathcal{E}_{i}$, $i=1,\cdots,8$.
      \vskip0.16in
      \noindent$\bullet$ For $\mathcal{E}_{1}$, combining the orthogonality conditions  \eqref{eq:3.9}, \eqref{eq:3.10} with the equations \eqref{eq:2.1}, \eqref{eq:2.4}, we know that
      $\rho$ is orthogonal with $\tilde{U}$ and $\partial_{\lambda}\tilde{U}$ in $\mathcal{D}^{1,2}(\R^{N})$, it follows that  $\int_{\mathbb{R}^{N}}\nabla \rho \cdot \nabla \xi \,dx=0$. Further, by the H\"older's inequality $(\frac{1}{p+1}+\frac{1}{N}=\frac{1}{2})$, we get that
    $$
         \aligned
             \mathcal{E}_{1}&=\Big| 
             \int_{\mathbb{R}^{N}} \nabla 
             \rho \cdot \nabla (\xi\Phi)  \,dx \Big|\\
             &=\Big| \int_{\mathbb{R}^{N}} \nabla
              \rho \cdot \nabla \big(\xi(\Phi-1)\big)  
              \,dx\Big|\\
              &\leq \|\nabla \rho\|_{L^{2}} 
               \|\nabla \big(\xi(\Phi-1)\big) 
               \|_{L^{2}}\\
             &=\|\nabla \rho\|_{L^{2}} \| 
              (\Phi-1) \nabla \xi + \xi \nabla 
              \Phi \|_{L^{2}}\\
              &\leq \|\nabla \rho\|_{L^{2}}
              \big( \| \nabla \xi\|_{L^{2}\{\Phi<1\}}
               + \|\xi\|_{L^{p+1}}\| \nabla 
               \Phi \|_{L^{N}}\big).
         \endaligned
   $$
      Thanks to the Lemmas \ref{lem:localization}-\ref{it:localization3},  we know that $\| \nabla \Phi \|_{L^{N}}=o(1)$. Combining  the estimate \eqref{eq:2.7} and  Lemma \ref{lem:localization}-\ref{it:localization1}, we know that
      $$\aligned\| \nabla \xi\|_{L^{2}\{\Phi<1\}}
      &\lesssim \|\nabla \tilde{U}\|_{L^{2}\{\Phi<1\}}
      +\|\tilde{U}\|_{L^{p+1}\{\Phi<1\}} \\
      &\lesssim\|\nabla U\|_{L^{2}\{\Phi<1\}}
      +\|U\|_{L^{p+1}\{\Phi<1\}}=o(1),\endaligned$$ 
      therefore,
      \begin{equation}\label{eq:3.18}
          \mathcal{E}_{1}\lesssim o(1)\|\nabla \rho\|_{L^{2}}.
      \end{equation}
      \vskip0.06in
      \noindent$\bullet$ For $\mathcal{E}_{2}$,  by using Fubini's theorem, we have that
    $$
         \aligned
             \mathcal{E}_{2}&=\Big|\int_{\mathbb{R}^{N}}  (\mathcal{A}_{3}+\mathcal{A}_{4}) \cdot \xi\Phi \,dx\Big| \\
             &=\Big|\int_{\mathbb{R}^{N}} \Big[(\tilde{p}+1)(I_{\mu}\ast(\alpha \tilde{U})^{\tilde{p}}\rho)(\alpha \tilde{U})^{\tilde{p}} + \tilde{p}\,(I_{\mu}\ast(\alpha \tilde{U})^{\tilde{p}+1})(\alpha \tilde{U})^{\tilde{p}-1}\rho \Big]\cdot \xi\Phi \,dx\Big|\\
             & =\Big|\int_{\mathbb{R}^{N}} \Big[(\tilde{p}+1)(I_{\mu}\ast(\alpha \tilde{U})^{\tilde{p}}\rho)(\alpha \tilde{U})^{\tilde{p}} + \tilde{p}\,(I_{\mu}\ast(\alpha \tilde{U})^{\tilde{p}+1})(\alpha \tilde{U})^{\tilde{p}-1}\rho \Big] \xi \,dx\\
             &\ \ +\int_{\mathbb{R}^{N}} \Big[(\tilde{p}+1)(I_{\mu}\ast(\alpha \tilde{U})^{\tilde{p}}\rho)(\alpha \tilde{U})^{\tilde{p}} + \tilde{p}\,(I_{\mu}\ast(\alpha \tilde{U})^{\tilde{p}+1})(\alpha \tilde{U})^{\tilde{p}-1}\rho \Big] \xi(\Phi-1)\,dx\Big|\\
             & =\Big|\int_{\mathbb{R}^{N}} \alpha^{2\tilde{p}}\Big[(\tilde{p}+1)(I_{\mu}\ast\tilde{U}^{\tilde{p}}\rho) \tilde{U}^{\tilde{p}} + \tilde{p}\,(I_{\mu}\ast \tilde{U}^{\tilde{p}+1}) \tilde{U}^{\tilde{p}-1}\rho \Big] \xi \,dx\\
             &\ \ +\int_{\mathbb{R}^{N}} \Big[(\tilde{p}+1)(I_{\mu}\ast(\alpha \tilde{U})^{\tilde{p}}\rho)(\alpha \tilde{U})^{\tilde{p}} + \tilde{p}\,(I_{\mu}\ast(\alpha \tilde{U})^{\tilde{p}+1})(\alpha \tilde{U})^{\tilde{p}-1}\rho \Big] \xi(\Phi-1)\,dx\Big|\\
             & =\Big|\int_{\mathbb{R}^{N}} \alpha^{2\tilde{p}}\Big[(\tilde{p}+1)(I_{\mu}\ast\tilde{U}^{\tilde{p}}\xi) \tilde{U}^{\tilde{p}} + \tilde{p}\,(I_{\mu}\ast \tilde{U}^{\tilde{p}+1}) \tilde{U}^{\tilde{p}-1}\xi \Big] \rho \,dx\\
             &\ \ +\int_{\mathbb{R}^{N}} \Big[(\tilde{p}+1)(I_{\mu}\ast(\alpha \tilde{U})^{\tilde{p}}\rho)(\alpha \tilde{U})^{\tilde{p}} + \tilde{p}\,(I_{\mu}\ast(\alpha \tilde{U})^{\tilde{p}+1})(\alpha \tilde{U})^{\tilde{p}-1}\rho \Big] \xi(\Phi-1)\,dx\Big|\\
             & =\Big|\int_{\mathbb{R}^{N}} \Big[(\tilde{p}+1)(I_{\mu}\ast(\alpha \tilde{U})^{\tilde{p}}\rho)(\alpha \tilde{U})^{\tilde{p}} + \tilde{p}\,(I_{\mu}\ast(\alpha \tilde{U})^{\tilde{p}+1})(\alpha \tilde{U})^{\tilde{p}-1}\rho \Big] \xi(\Phi-1)\,dx\Big|,
         \endaligned
     $$
      where the last equality is obtained by the orthogonality conditions \eqref{eq:3.9} and \eqref{eq:3.10}. Then
      \begin{equation}\label{eq:3.19}
         \aligned
            \mathcal{E}_{2}&\lesssim \int_{\mathbb{R}^{N}} (I_{\mu}\ast \tilde{U}^{\tilde{p}}|\rho|) \tilde{U}^{\tilde{p}}\tilde{U}(1-\Phi)\,dx + \int_{\mathbb{R}^{N}}(I_{\mu}\ast \tilde{U}^{\tilde{p}+1}) \tilde{U}^{\tilde{p}-1}|\rho| \tilde{U}(1-\Phi)\,dx \\
            &\lesssim \|\tilde{U}^{\tilde{p}}|\rho|\|_{L^{\frac{p+1}{\tilde{p}+1}}}\|\tilde{U}^{\tilde{p}+1}(1-\Phi)\|_{L^{\frac{p+1}{\tilde{p}+1}}}+ \|\tilde{U}^{\tilde{p}+1}\|_{L^{\frac{p+1}{\tilde{p}+1}}} \|\tilde{U}^{\tilde{p}}(1-\Phi) |\rho|  \|_{L^{\frac{p+1}{\tilde{p}+1}}}\\
            &\lesssim \|\tilde{U}\|_{L^{p+1}}^{\tilde{p}}\|\nabla \rho\|_{L^{2}} \|\tilde{U} \|_{L^{p+1}(\{\Phi<1\})}^{\tilde{p}+1}+\|\tilde{U}\|_{L^{p+1}}^{\tilde{p}+1}\|\tilde{U}\|_{L^{p+1}(\{\Phi<1\})}^{\tilde{p}}\|\nabla \rho\|_{L^{2}} \\
            &\lesssim o(1) \|\nabla \rho\|_{L^{2}} \quad\quad \hbox{(Lemma\; \ref{lem:localization}-(1)).}
         \endaligned
      \end{equation}
      \vskip0.16in
      \noindent$\bullet$ For $\mathcal{E}_{3},\mathcal{E}_{4},\mathcal{E}_{5}$, by  \eqref{eq:3.14}, \eqref{eq:3.15} and \eqref{eq:3.16}, we have known  that
      \begin{equation}\label{eq:3.20}
         \aligned
             \mathcal{E}_{3} +\mathcal{E}_{4}+\mathcal{E}_{5} &\lesssim \|\nabla \rho\|_{L^{2}}^{\min(\tilde{p},2)}+\|\nabla \rho\|_{L^{2}}^{\min(\tilde{p}+1,2)}+\|\nabla \rho\|_{L^{2}}^{2}\lesssim \|\nabla \rho\|_{L^{2}}^{\min(\tilde{p},2)}.
         \endaligned
      \end{equation}
    To achieve our goals,  we require the condition $\tilde{p}>1$ $(\iff \mu<4 )$ to guarantee that the exponent is strictly greater than $1$.
      \vskip0.16in
      \noindent$\bullet$ For $\mathcal{E}_{6}$, we observe that
    $$
      \aligned
      \mathcal{C}_{2} + \mathcal{A}_{3}=\ &(\tilde{p}+1)\big(I_{\mu}\ast \sigma^{\tilde{p}}\rho \big)\sigma^{\tilde{p}} - (\tilde{p}+1)(I_{\mu}\ast(\alpha \tilde{U})^{\tilde{p}}\rho)(\alpha \tilde{U})^{\tilde{p}}\\
      \approx\ &\big(I_{\mu}\ast \sigma^{\tilde{p}}\rho \big)\sigma^{\tilde{p}} - (I_{\mu}\ast(\alpha \tilde{U})^{\tilde{p}}\rho)(\alpha \tilde{U})^{\tilde{p}}\\
      =\ & \big(I_{\mu}\ast (\sigma^{\tilde{p}}\rho- (\alpha \tilde{U})^{\tilde{p}}\rho ) \big)\sigma^{\tilde{p}} + (I_{\mu}\ast(\alpha \tilde{U})^{\tilde{p}}\rho)\big(\sigma ^{\tilde{p}}-  (\alpha \tilde{U})^{\tilde{p}} \big)\\
      \eqqcolon  & \ \  \mathcal{F}_{1}+\mathcal{F}_{2}.
      \endaligned
  $$
       Since $|\sigma^{\tilde{p} }-(\alpha \tilde{U})^{\tilde{p}} |\lesssim \varepsilon \tilde{U}^{\tilde{p}}$ in the region $\{ \Phi >0 \}$ (see the estimate of $\mathcal{D}_{4}$), thus
   $$
      \aligned
          \Big|\int_{\mathbb{R}^{N}} \mathcal{F}_{2} \cdot \xi \Phi \,dx \Big|&= \Big|\int_{\mathbb{R}^{N}}(I_{\mu}\ast(\alpha \tilde{U})^{\tilde{p}}\rho)\big(\sigma ^{\tilde{p}}-  (\alpha \tilde{U})^{\tilde{p}} \big) \cdot \xi \Phi \,dx\Big|\lesssim  \int_{\mathbb{R}^{N}}(I_{\mu}\ast(\alpha \tilde{U})^{\tilde{p}}\rho)\,\varepsilon \tilde{U}^{\tilde{p}}  |\xi| \Phi \,dx\\
          & \lesssim \varepsilon \int_{\mathbb{R}^{N}}(I_{\mu}\ast(\alpha \tilde{U})^{\tilde{p}}\rho)\,\tilde{U}^{\tilde{p}+1}  \,dx \lesssim \varepsilon \|\nabla \rho\|_{L^{2}} \approx o(1) \|\nabla \rho\|_{L^{2}}.
      \endaligned
 $$
       On the other hand,
       \begin{equation}\label{eq:3.21}
      \aligned
          \Big|\int_{\mathbb{R}^{N}} \mathcal{F}_{1} \cdot \xi \Phi \,dx \Big|&=\Big| \int_{\mathbb{R}^{N}}\big(I_{\mu}\ast (\sigma^{\tilde{p}}\rho- (\alpha \tilde{U})^{\tilde{p}}\rho ) \big)\sigma^{\tilde{p}} \cdot \xi \Phi \,dx\Big|\\
          &\leq\Big| \int_{\mathbb{R}^{N}}\big(I_{\mu}\ast \mathds{1}_{\{\tilde{V}< \varepsilon\nu\tilde{U} \}}(\sigma^{\tilde{p}}\rho- (\alpha \tilde{U})^{\tilde{p}}\rho ) \big)\sigma^{\tilde{p}} \cdot \xi \Phi\,dx\Big| \\
   & +\Big| \int_{\mathbb{R}^{N}}\big(I_{\mu}\ast \mathds{1}_{\{\tilde{V}\geq \varepsilon\nu\tilde{U} \}}(\sigma^{\tilde{p}}\rho- (\alpha \tilde{U})^{\tilde{p}}\rho ) \big)\sigma^{\tilde{p}} \cdot \xi \Phi \,dx\Big|\\
   &\eqqcolon S_{1}+S_{2}.
   \endaligned
  \end{equation}
  For the  f{}irst term  {$S_{1}$} above,
  since $|\sigma^{\tilde{p}}-(\alpha \tilde{U})^{\tilde{p}}|\lesssim\varepsilon \tilde{U}^{\tilde{p}}$ in the region $\{\tilde{V}< \varepsilon\nu\tilde{U}\}$ (see the estimate of $\mathcal{D}_{4}$), and   $\sigma^{\tilde{p}}\lesssim \tilde{U}^{\tilde{p}}+ \tilde{V}^{\tilde{p}}\lesssim (1+\varepsilon^{\tilde{p}})
  \tilde{U}^{\tilde{p}}\lesssim \tilde{U}^{\tilde{p}}$ in the region $\{\Phi>0\}$ for $\tilde{p}>0$,  we obtain that
 $$
      \aligned
      S_{1}\lesssim \ & \int_{\mathbb{R}^{N}}\big(I_{\mu}\ast \mathds{1}_{\{\tilde{V}< \varepsilon\nu\tilde{U} \}}(|\sigma^{\tilde{p}}- (\alpha \tilde{U})^{\tilde{p}}|\cdot|\rho| ) \big)\sigma^{\tilde{p}} \cdot |\xi| \Phi \,dx\\
      \lesssim \ & \varepsilon \int_{\mathbb{R}^{N}}  \big(I_{\mu}\ast (\tilde{U}^{\tilde{p}} |\rho|)  \big) \tilde{U}^{\tilde{p}+1} \,dx \lesssim \varepsilon\|\nabla \rho\|_{L^{2}}\approx o(1)\|\nabla \rho\|_{L^{2}}.
      \endaligned
   $$
      As for the  second term  {$S_{2}$}, since $|\sigma^{\tilde{p}}-(\alpha \tilde{U})^{\tilde{p}} |\lesssim (\frac{1}{\varepsilon})^{\tilde{p}} \tilde{V}^{\tilde{p}} $ in the region $ \{\tilde{V}\geq \varepsilon\nu\tilde{U} \}$ (see the estimate of $\mathcal{D}_{4}$),  and $\sigma^{\tilde{p}}\lesssim \tilde{U}^{\tilde{p}}+ \tilde{V}^{\tilde{p}}\lesssim (1+\varepsilon^{\tilde{p}})
  \tilde{U}^{\tilde{p}}\lesssim \tilde{U}^{\tilde{p}}$ in the region $\{\Phi>0\}$ when $\tilde{p}>0$,  then we get that
      \begin{equation}\label{eq:3.22}
      \aligned
          S_{2}=\ &\Big| \int_{\mathbb{R}^{N}}\big(I_{\mu}\ast (\sigma^{\tilde{p}} \xi \Phi ) \big)  \mathds{1}_{\{\tilde{V}\geq \varepsilon\nu\tilde{U} \}}(\sigma^{\tilde{p}}\rho- (\alpha \tilde{U})^{\tilde{p}}\rho )      \,dx\Big|\\
          \lesssim\ & \int_{\mathbb{R}^{N}}\big(I_{\mu}\ast (\tilde{U}^{\tilde{p}} \tilde{U} ) \big)  \mathds{1}_{\{\tilde{V}\geq \varepsilon\nu\tilde{U} \}}\Big(\frac{1}{\varepsilon}\Big)^{\tilde{p}} \tilde{V}^{\tilde{p}}|\rho|       \,dx\\
          \lesssim \ & \Big(\frac{1}{\varepsilon}\Big)^{\tilde{p}}\int_{\mathbb{R}^{N}}\big(I_{\mu}\ast  \tilde{U}^{\tilde{p}+1} \big) \tilde{V}^{\tilde{p}}|\rho| \,dx\lesssim  \Big(\frac{1}{\varepsilon}\Big)^{\tilde{p}} \sum_{j\neq i} \int_{\mathbb{R}^{N}}  \tilde{U}^{p-\tilde{p}} \tilde{U}_{j}^{\tilde{p}}|\rho| \,dx \\
          \lesssim \ &  \Big(\frac{1}{\varepsilon}\Big)^{\tilde{p}} \sum_{j\neq i} \Big(\int_{\mathbb{R}^{N}} (\tilde{U}^{p-\tilde{p}} \tilde{U}_{j}^{\tilde{p}})^{\frac{p+1}{p}}  \,dx \Big)^{\frac{p}{p+1}} \|\nabla \rho\|_{L^{2}}.
      \endaligned
       \end{equation}
      By Remark \ref{power of epilson}, we know that  $\delta=\varepsilon^{\eta-\frac{2}{N-2}}$, where $\eta\geq 16$ will be determined later.
      \vskip0.08in
  {\bf Case 1:} If $p-\tilde{p}>\tilde{p}$ (i.e., $ \mu>\frac{N+2}{2}$), we take $\eta=16$ and  apply  Lemma \ref{prop:interaction_approx},
 $$
\aligned
\Big(\int_{\mathbb{R}^{N}} (\tilde{U}^{p-\tilde{p}} \tilde{U}_{j}^{\tilde{p}})^{\frac{p+1}{p}}  \,dx \Big)^{\frac{p}{p+1}} & \approx\Big( Q^{\frac{N-2}{2}\min((p-\tilde{p})\frac{p+1}{p}, \tilde{p}\frac{p+1}{p})}\Big)^{\frac{p}{p+1}}\\
&\approx Q^{\frac{N-2}{2}\tilde{p} }\approx Q^{\frac{N-\mu+2}{2}}\lesssim Q \lesssim \delta\approx\varepsilon^{16-\frac{2}{N-2}}.
\endaligned
 $$
Hence,the right hand side of \eqref{eq:3.22} is $o(1)\|\nabla \rho\|_{L^{2}}$ whenever $(16-\frac{2}{N-2})>\tilde{p}=2_{\mu}^{\ast}-1$, which is equivalent to  $ 15N-36>-\mu$. Evidently, this is always true  for $N\geq 3$ and $\mu>0$.

\vskip0.18in
       {\bf Case 2:} If $p-\tilde{p}=\tilde{p}$, i.e., $\mu=\frac{N+2}{2}$,  then  we let  $\eta=16$ and apply Lemma \ref{prop:interaction_approx} getting
 $$
      \aligned
          \Big(\int_{\mathbb{R}^{N}} (\tilde{U}^{p-\tilde{p}} \tilde{U}_{j}^{\tilde{p}})^{\frac{p+1}{p}}  \,dx \Big)^{\frac{p}{p+1}}&\approx\Big( Q^{\frac{N}{2} }\log(\frac{1}{Q})  \Big)^{\frac{p}{p+1}}\approx Q^{\frac{N+2}{4}} \Big(\log(\frac{1}{Q})  \Big)^{\frac{N+2}{2N}}\\
          &\lesssim \delta^{\frac{N+2}{4}} \Big(\log(\frac{1}{\delta})  \Big)^{\frac{N+2}{2N}},
          \text{\quad as \quad} Q<\delta \ll 1,
      \endaligned
  $$
       so the right hand side of \eqref{eq:3.22} is $o(1)\|\nabla \rho\|_{L^{2}}$, which is equivalent to $\frac{N+2}{4}(16-\frac{2}{N-2})>2_{\mu}^{\ast}-1 $ and also to $ 16N^{2}-6N-76>-4\mu$, which  always holds   when  $N\geq 3$ and $\mu>0$.
       \vskip0.18in
    {\bf Case 3:} If $p-\tilde{p}<\tilde{p}$ (equivalently $\mu<\frac{N+2}{2}$), by using Lemma \ref{prop:interaction_approx}, then we get that
 $$
          \Big(\int_{\mathbb{R}^{N}} (\tilde{U}^{p-\tilde{p}} \tilde{U}_{j}^{\tilde{p}})^{\frac{p+1}{p}}\,dx \Big)^{\frac{p}{p+1}}\approx\Big( Q^{\frac{N-2}{2}\min((p-\tilde{p})
          \frac{p+1}{p}, \tilde{p}\frac{p+1}{p}) }\Big)^{\frac{p}{p+1}}\approx Q^{\frac{N-2}{2}(p-\tilde{p}) }\approx Q^{\frac{\mu}{2}}\lesssim \delta^{\frac{\mu}{2}}\approx
          \varepsilon^{\frac{\mu}{2}
          (\eta-\frac{2}{N-2})}.
  $$
     It follows that the right hand side of \eqref{eq:3.22} is $o(1)\|\nabla \rho\|_{L^{2}}$ as long as $ \frac{\mu}{2}(\eta-\frac{2}{N-2})>2_{\mu}^{\ast}-1 $, which is equivalent to $ \mu >\frac{2}{\eta}(1+\frac{4}{N-2})$, where $\frac{2}{\eta}(1+\frac{4}{N-2})\leq \frac{10}{\eta}$ since $N\geq3$. In other words,  for a f{}ixed $\mu\in(0,N)$,  if we choose  $\eta\geq16$ suf{}f{}iciently large such that  $\mu >\frac{2}{\eta}(1+\frac{4}{N-2})$,
      then the right hand side of \eqref{eq:3.22} must be $o(1)\|\nabla \rho\|_{L^{2}}$.

\vskip0.12in
The above three cases show  that the  { second term  $S_{2}$  is $o(1)\|\nabla \rho\|_{L^{2}}$.  Combining $S_1$, $S_2$ and \eqref{eq:3.21}, we have $\big|\int_{\mathbb{R}^{N}} \mathcal{F}_{1} \cdot \xi \Phi \,dx \big|\lesssim o(1) \|\nabla \rho\|_{L^{2}}$.}  Therefore
      \begin{equation}\label{eq:3.23}
         \mathcal{E}_{6}=\Big|\int_{\mathbb{R}^{N}}( \mathcal{C}_{2} + \mathcal{A}_{3}) \cdot \xi \Phi \,dx \Big| \ \lesssim\ \Big|\int_{\mathbb{R}^{N}} \mathcal{F}_{1} \cdot \xi \Phi \,dx\Big|+\Big| \int_{\mathbb{R}^{N}} \mathcal{F}_{2} \cdot \xi \Phi \,dx\Big|\ \lesssim\  o(1) \|\nabla \rho\|_{L^{2}}.
      \end{equation}
      \vskip0.06in
      \noindent$\bullet$ For  $\mathcal{E}_{7}$, we f{}irstly observe that
 $$
      \aligned
      \mathcal{C}_{3} + \mathcal{A}_{4}=\ &\tilde{p} \big(I_{\mu}\ast \sigma^{\tilde{p}+1}\big) \sigma^{\tilde{p}-1}\rho- \tilde{p}\,(I_{\mu}\ast(\alpha \tilde{U})^{\tilde{p}+1})(\alpha \tilde{U})^{\tilde{p}-1}\rho\\
      \approx\ & \big(I_{\mu}\ast \sigma^{\tilde{p}+1}\big) \sigma^{\tilde{p}-1}\rho- (I_{\mu}\ast(\alpha \tilde{U})^{\tilde{p}+1})(\alpha \tilde{U})^{\tilde{p}-1}\rho\\
      =\ & \big(I_{\mu}\ast \sigma^{\tilde{p}+1}\big) (\sigma^{\tilde{p}-1}- (\alpha \tilde{U})^{\tilde{p}-1})\rho+\big(I_{\mu}\ast  (\sigma^{\tilde{p}+1}- (\alpha \tilde{U})^{\tilde{p}+1})\big) (\alpha \tilde{U})^{\tilde{p}-1}\rho\\
      \eqqcolon  & \ \ \mathcal{G}_{1}+\mathcal{G}_{2}.
      \endaligned
 $$
       Since  $|\sigma^{\tilde{p}-1 }-(\alpha \tilde{U})^{\tilde{p}-1} |\lesssim \varepsilon \tilde{U}^{\tilde{p}-1}$ in the region $\{\Phi >0\}$, then
 $$
      \aligned
          \Big|\int_{\mathbb{R}^{N}} \mathcal{G}_{1} \cdot \xi \Phi \,dx \Big|&= \Big|\int_{\mathbb{R}^{N}}\big(I_{\mu}\ast \sigma^{\tilde{p}+1}\big) (\sigma^{\tilde{p}-1}- (\alpha \tilde{U})^{\tilde{p}-1})\rho \cdot \xi \Phi \,dx\Big|\\
          &\lesssim  \varepsilon\int_{\mathbb{R}^{N}}(I_{\mu}\ast\sigma^{\tilde{p}+1})\, \tilde{U}^{\tilde{p}}  |\rho|  \,dx  \lesssim \varepsilon \|\nabla \rho\|_{L^{2}} \approx o(1) \|\nabla \rho\|_{L^{2}}.
      \endaligned
 $$
 On the other hand,
 \begin{equation}\label{eq:3.24}
  \begin{aligned}
  &\Big|\int_{\mathbb{R}^{N}} \mathcal{G}_{2} \cdot \xi \Phi \,dx \Big|\\
  &= \Big|\int_{\mathbb{R}^{N}}\big(I_{\mu}\ast  (\sigma^{\tilde{p}+1}- (\alpha \tilde{U})^{\tilde{p}+1})\big) (\alpha \tilde{U})^{\tilde{p}-1}\rho \cdot \xi \Phi \,dx\Big|\\
  &\leq\Big| \int_{\mathbb{R}^{N}}\big(I_{\mu}\ast \mathds{1}_{\{\tilde{V}< \varepsilon\nu\tilde{U} \}}(\sigma^{\tilde{p}+1}- (\alpha \tilde{U})^{\tilde{p}+1} ) \big)(\alpha \tilde{U})^{\tilde{p}-1}\rho \cdot \xi \Phi \,dx\Big| \\
  &\quad +\Big| \int_{\mathbb{R}^{N}}\big(I_{\mu}\ast \mathds{1}_{\{\tilde{V}\geq \varepsilon\nu\tilde{U} \}}(\sigma^{\tilde{p}+1}- (\alpha \tilde{U})^{\tilde{p}+1} ) \big)(\alpha \tilde{U})^{\tilde{p}-1}\rho \cdot \xi \Phi \,dx\Big|\\
  &\eqqcolon T_1+T_2.
 \end{aligned}
 \end{equation}
To estimate  the f{}irst term $T_1$ above, we note that $|\sigma^{\tilde{p}+1}-(\alpha \tilde{U})^{\tilde{p}+1}|\lesssim \varepsilon \tilde{U}^{\tilde{p}+1}$ in the region $\{\tilde{V}< \varepsilon\nu\tilde{U}\}$,  therefore
 $$
      \aligned
       T_1
      \lesssim   \varepsilon \int_{\mathbb{R}^{N}}\big(I_{\mu}\ast  \tilde{U}^{\tilde{p}+1}\big)  \tilde{U}^{\tilde{p}}|\rho|  \,dx \lesssim \varepsilon\|\nabla \rho\|_{L^{2}}\approx o(1)\|\nabla \rho\|_{L^{2}}.
      \endaligned
 $$
         For the second term $T_2$ above, since $|\sigma^{\tilde{p}+1}-(\alpha \tilde{U})^{\tilde{p}+1}|\lesssim (\frac{1}{\varepsilon})^{\tilde{p}+1} \tilde{V}^{\tilde{p}+1}$   in the region $ \{\tilde{V}\geq  \varepsilon\nu\tilde{U}\}$,  we have that
      \begin{equation}\label{eq:3.25}
      \aligned
          T_2&\lesssim    \int_{\mathbb{R}^{N}}\big(I_{\mu}\ast \mathds{1}_{\{\tilde{V}\geq \varepsilon\nu\tilde{U} \}} \Big(\frac{1}{\varepsilon}\Big)^{\tilde{p}+1} \tilde{V}^{\tilde{p}+1}   \big)(\alpha \tilde{U})^{\tilde{p}-1}|\rho|\cdot|\xi|\Phi  \,dx \\
          &\lesssim    \Big(\frac{1}{\varepsilon}\Big)^{\tilde{p}+1}\sum_{j\neq i}\int_{\mathbb{R}^{N}}\big(I_{\mu}\ast \tilde{U}_{j}^{\tilde{p}+1}\big)\tilde{U}^{\tilde{p}} |\rho|\,dx \approx \Big(\frac{1}{\varepsilon}\Big)^{\tilde{p}+1}\sum_{j\neq i}\int_{\mathbb{R}^{N}} \tilde{U}_{j}^{p-\tilde{p}}\tilde{U}^{\tilde{p}} |\rho|\,dx \\
          &\lesssim   \Big(\frac{1}{\varepsilon}\Big)^{\tilde{p}+1}\sum_{j\neq i}\Big(\int_{\mathbb{R}^{N}} (\tilde{U}_{j}^{p-\tilde{p}} \tilde{U}^{\tilde{p}})^{\frac{p+1}{p}}  \,dx \Big)^{\frac{p}{p+1}} \|\nabla \rho\|_{L^{2}}.
      \endaligned
       \end{equation}
      Recall the relation between $\delta$ and $\varepsilon$(see Remark \ref{power of epilson}), we have $\delta=\varepsilon^{\eta-\frac{2}{N-2}}$, where $\eta\geq 16$ will be determined later.
    \vskip0.18in
  {\bf Case 1:} If $p-\tilde{p}>\tilde{p}$ (equivalently, $ \mu>\frac{N+2}{2}$), then we may take $\eta=16$ and by Lemma \ref{prop:interaction_approx}, we get that
 $$
      \aligned
          \Big(\int_{\mathbb{R}^{N}} (\tilde{U}_{j}^{p-\tilde{p}} \tilde{U}^{\tilde{p}})^{\frac{p+1}{p}}  \,dx \Big)^{\frac{p}{p+1}} & \approx\Big( Q^{\frac{N-2}{2}\min((p-\tilde{p})\frac{p+1}{p}, \tilde{p}\frac{p+1}{p})}\Big)^{\frac{p}{p+1}}\\
          &\approx Q^{\frac{N-2}{2}\tilde{p} }\approx Q^{\frac{N-\mu+2}{2}}\lesssim Q \lesssim \delta \approx\varepsilon^{16-\frac{2}{N-2}},
      \endaligned
  $$
    therefore, the last term of   \eqref{eq:3.25} is $o(1)\|\nabla \rho\|_{L^{2}}$ whenever  $ (16-\frac{2}{N-2})>\tilde{p}+1=2_{\mu}^{\ast},$ which is equivalent to $ 14N-34>-\mu$ and  always holds as long as $N\geq 3$ and $\mu>0$.
        \vskip0.18in
    {\bf Case 2:} If $p-\tilde{p}=\tilde{p}$, i.e., $\mu=\frac{N+2}{2}$, then we  may choose  $\eta=16$. Then by the Lemma \ref{prop:interaction_approx}, we have that
 $$
      \aligned
          \Big(\int_{\mathbb{R}^{N}} (\tilde{U}_{j}^{p-\tilde{p}} \tilde{U}^{\tilde{p}})^{\frac{p+1}{p}}  \,dx \Big)^{\frac{p}{p+1}}&\approx\Big( Q^{\frac{N}{2} }\log(\frac{1}{Q})  \Big)^{\frac{p}{p+1}}\approx Q^{\frac{N+2}{4}} \Big(\log(\frac{1}{Q})  \Big)^{\frac{N+2}{2N}}\\
          &\lesssim \delta^{\frac{N+2}{4}} \Big(\log(\frac{1}{\delta})  \Big)^{\frac{N+2}{2N}}, \text{\quad as \quad} Q<\delta \ll 1.
      \endaligned
 $$
 It follows that the right hand side of \eqref{eq:3.25} is $o(1)\|\nabla \rho\|_{L^{2}}$ whenever $ \frac{N+2}{4}(16-\frac{2}{N-2})>2_{\mu}^{\ast}$, that is $ 16N^{2}-4N-68>-\mu$, which  holds  as long as $N\geq 3$ and $\mu>0$.
        \vskip0.08in
    {\bf Case 3:}  If  $p-\tilde{p}<\tilde{p}$ (we need $\mu<\frac{N+2}{2}$), then by the Lemma
    \ref{prop:interaction_approx}, we see that
 $$
          \Big(\int_{\mathbb{R}^{N}} (\tilde{U}_{j}^{p-\tilde{p}}
          \tilde{U}^{\tilde{p}})^{\frac{p+1}{p}}  \,dx \Big)^{\frac{p}{p+1}}\approx\Big(
             Q^{\frac{N-2}{2}\min((p-\tilde{p})\frac{p+1}{p},
             \tilde{p}\frac{p+1}{p}) }\Big)^{\frac{p}{p+1}}\approx
             Q^{\frac{N-2}{2}(p-\tilde{p}) }\approx Q^{\frac{\mu}{2}}\lesssim
             \delta^{\frac{\mu}{2}}\approx\varepsilon^{\frac{\mu}{2}(\eta-\frac{2}{N-2})}.
 $$
 Note that  $\frac{\mu}{2}(\eta-\frac{2}{N-2})>2_{\mu}^{\ast} \Leftrightarrow  \mu >\frac{4}{\eta}(1+\frac{2}{N-2})$  and $\frac{4}{\eta}(1+\frac{2}{N-2})\leq \frac{12}{\eta}$. Hence,  for a f{}ixed $\mu\in(0,N)$, if we take $\eta\geq 16 $ suf{}f{}iciently large such that $\mu>\frac{12}{\eta}$, then we must have  $\mu >\frac{4}{\eta}(1+\frac{2}{N-2})$ and it follows that  the right hand side of \eqref{eq:3.25} is $o(1)\|\nabla \rho\|_{L^{2}}$.

 \vskip0.23in
 To sum up, the above three cases  imply  that the second term $T_2$  is also  $ o(1)\|\nabla \rho\|_{L^{2}}$.  Combining $T_1$, $T_2$ and \eqref{eq:3.24}, we have $\big|\int_{\mathbb{R}^{N}} \mathcal{G}_{2} \cdot \xi \Phi \,dx \big|\lesssim o(1) \|\nabla \rho\|_{L^{2}}$.  Therefore,
      \begin{equation}\label{eq:3.26}
         \mathcal{E}_{7}=\Big|\int_{\mathbb{R}^{N}}( \mathcal{C}_{3} + \mathcal{A}_{4}) \cdot \xi \Phi \,dx \Big| \ \lesssim\ \Big|\int_{\mathbb{R}^{N}} \mathcal{G}_{1} \cdot \xi \Phi \,dx\Big|+\Big| \int_{\mathbb{R}^{N}} \mathcal{G}_{2} \cdot \xi \Phi \,dx\Big|\ \lesssim\  o(1) \|\nabla \rho\|_{L^{2}}.
      \end{equation}
     \vskip0.06in
      \noindent$\bullet$ For $\mathcal{E}_{8}$, using \eqref{eq:2.7}, we get that
 $$
     \aligned
         \|\nabla(\xi \Phi)\|_{L^{2}}= \|  \Phi \nabla \xi + \xi \nabla \Phi \|_{L^{2}}&\lesssim
         ( \| \nabla \xi\|_{L^{2}} + \|\tilde{U}\|_{L^{p+1}}\| \nabla \Phi \|_{L^{N}(\{\Phi<1\})} )\\
         &\lesssim ( \| \nabla \tilde{U}\|_{L^{2}}+\| \tilde{U}\|_{L^{p+1}}+o(1) \|\tilde{U}\|_{L^{p+1}})\lesssim 1.
     \endaligned
 $$
     Therefore,
     \begin{equation}\label{eq:3.27}
     \aligned
         \mathcal{E}_{8}&=\Big| \int_{\mathbb{R}^{N}}  \mathcal{A}_{5}\cdot \xi\Phi  \,dx \Big| =\Big| \int_{\mathbb{R}^{N}}  f\cdot \xi\Phi  \,dx \Big|\\
         &= \Big| \int_{\mathbb{R}^{N}}   (-\Delta u-(I_{\mu}\ast |u|^{\tilde{p}+1})|u|^{\tilde{p}-1}u)\cdot \xi\Phi\,dx \Big|\lesssim \Theta(u)\|\nabla(\xi \Phi)\|_{L^{2}} \lesssim  \Theta(u).
     \endaligned
     \end{equation}
     Unite the above estimates of $\mathcal{E}_{1}\sim\mathcal{E}_{8}$, i.e., \eqref{eq:3.18}, \eqref{eq:3.19}, \eqref{eq:3.20}, \eqref{eq:3.23}, \eqref{eq:3.26} and \eqref{eq:3.27}, and note that only the term $\mathcal{E}_{3}$ requires  additional condition $\mu<4$ for our purpose, then it follows from \eqref{eq:3.17} that
     \begin{equation}\label{eq:3.28}
      \aligned
      & \Big|\int_{\mathbb{R}^{N}}  \Big[(\alpha-\alpha^{2\tilde{p}+1}+o(1) )  \tilde{U}^{p} - (\tilde{p}\alpha^{2\tilde{p}}+o(1))\tilde{U}^{p-1}\tilde{V} \Big]  \xi \Phi\,dx \Big| \\
       \lesssim \ \ & o(1)\|\nabla \rho\|_{L^{2}}+\|\nabla \rho\|_{L^{2}}^{\min(\tilde{p}, 2) }+\Theta(u).
      \endaligned
      \end{equation}
      Let us now split $\tilde{V}=\tilde{V}_1+\tilde{V}_2$ where $\tilde{V}_1\coloneqq\sum\limits_{j<i}\alpha_j \tilde{U}_j$ and
     $\tilde{V}_2\coloneqq\sum\limits_{j>i}\alpha_j \tilde{U}_j$.
     If $i=1$, then $\tilde{V}_1=0, \tilde{V}=\tilde{V}_{2}$ and  the following arguments will prove the results for the case $i=1$.
     If $i>1$, then by induction assumption that for any f{}ixed $j<i$, \eqref{eq:3.12} holds for any $i\neq j$. Then
     \begin{equation}\label{eq:3.29}
         \aligned
      \int_{\R^N}\tilde{U}^{p}\tilde{U}_j\,dx&= \int_{\R^N} \tilde{U}_j^{p}\tilde{U}\,dx \lesssim o(1)\|\nabla\rho\|_{L^2}
          +\|\nabla \rho\|_{L^{2}}^{\min(\tilde{p}, 2) }+\Theta(u).
         \endaligned
     \end{equation}
    After taking the  summation with respect to $j<i$ and  using $|\xi|\lesssim \tilde{U}$, we obtain that
     $$
     \aligned \int_{\R^N}\tilde{U}^{p-1}\tilde{V}_1\abs{\xi}\Phi\,dx
     &\lesssim \int_{\R^N}\tilde{U}^{p}\tilde{V}_1\,dx \lesssim  o(1)\|\nabla\rho\|_{L^2}+\|\nabla \rho\|_{L^{2}}^{\min(\tilde{p}, 2) } +\Theta(u).\endaligned
      $$
     Thus it follows from $\eqref{eq:3.28}$ that
     \begin{equation}\label{eq:3.30}
         \aligned
         & \Big|\int_{\mathbb{R}^{N}}  \Big[(\alpha-\alpha^{2\tilde{p}+1}+o(1) )  \tilde{U}^{p} - (\tilde{p}\alpha^{2\tilde{p}}+o(1))\tilde{U}^{p-1}\tilde{V}_{2} \Big]  \xi \Phi\,dx \Big| \\
       \lesssim \ \ & o(1)\|\nabla \rho\|_{L^{2}}+\|\nabla \rho\|_{L^{2}}^{\min(\tilde{p},2)}+\Theta(u).
         \endaligned
     \end{equation}
     Next we verify that
     \begin{equation}\label{eq:3.31}
     \tilde{V}_2(x)\Phi(x) = (1+o(1))\tilde{V}_2(0)\Phi(x), \quad\quad  \hbox{ for\, any } x\in\R^N.
     \end{equation}

     \textbf{Case 1:} If $\Phi(x)=0$, then \eqref{eq:3.31} holds.

     \textbf{Case 2:} If $\Phi(x)\neq0$, by the construction of $\Phi$, we know $\{\Phi>0\}\subseteq B(0,R)$, where the notation $R$ is from Lemma \ref{lem:localization}, then it follows from $0,x\in B(0,R)$ and the Claim 2 that
$$
         \frac{U_{j}(x)}{U_{j}(0)},\ \frac{U_{j}(0)}{U_{j}(x)} \leq \frac{\sup_{B(0,R)}U_j}{\inf_{B(0,R)}U_j}\le 1+\varepsilon
 $$
     for any $j\neq i$ satisfying $\lambda_{j}\leq \lambda_{i}$. Therefore,  $U_{j}(x)=(1+o(1))U_{j}(0)$ and hence $$\sum_{j>i}\alpha_{j}U_{j}(x)=(1+o(1))\sum_{j>i}\alpha_{j}U_{j}(0).$$
  It follows that $\tilde{V}_2(x)\Phi(x)=(1+o(1))\tilde{V}_2(0)\Phi(x)$.

 \vskip0.2in

Thus, the above two cases show that $\tilde{V}_2(x)\Phi(x) = (1+o(1))\tilde{V}_2(0)\Phi(x)$. Then it follows from $\eqref{eq:3.30}$ that
     \begin{equation}\label{eq:3.32}
         \aligned
         & \Big|\int_{\mathbb{R}^{N}}  \Big[(\alpha-\alpha^{2\tilde{p}+1}+o(1) )  \tilde{U}^{p} - (\tilde{p}\alpha^{2\tilde{p}}+o(1))\tilde{U}^{p-1}\tilde{V}_{2}(0) \Big]  \xi \Phi\,dx \Big| \\
       & \lesssim  o(1)\|\nabla \rho\|_{L^{2}}+\|\nabla \rho\|_{L^{2}}^{\min(\tilde{p}, 2)}+\Theta(u).
         \endaligned
     \end{equation}

\vskip0.3in

We now prove \eqref{eq:3.11} of Lemma \ref{prop:interaction_and_coef}. If $\alpha=1$, then there is nothing to
     prove. Therefore, we assume that $\alpha \neq 1$ and def{}ine $\displaystyle \theta\coloneqq\frac{\tilde{p}\alpha^{2\tilde{p}}\tilde{V}_2(0)}{\alpha-\alpha^{2\tilde{p}+1}}$, then $\theta$ is independent of  $\xi=\tilde{U}$ or $\xi=\partial_\lambda\tilde{U}$.
     It follows from \eqref{eq:3.32} that
     \begin{equation}\label{eq:3.33}
     \aligned
     \big|\alpha-\alpha^{2\tilde{p}+1}\big|\cdot
     \Big|\int_{\R^N}(1+o(1))(\tilde{U}^{p}-\theta \tilde{U}^{p-1})
     \xi\Phi\,dx \Big|  \lesssim
      o(1)\|\nabla \rho\|_{L^{2}}+\|\nabla \rho\|_{L^{2}}^{\min(\tilde{p},2)}+\Theta(u) .
     \endaligned
     \end{equation}
By the Lemma \ref{lem:localization}-\ref{it:localization1}, we have
\begin{equation}\label{eq:3.34}
 \int_{\R^N}\bigl(1+o(1))(\tilde{U}^{p}-\theta \tilde{U}^{p-1}\bigr)\xi\Phi\,dx
  =\int_{\R^N} \tilde{U}^{p}\xi\,dx-\theta\int_{\R^N}\tilde{U}^{p-1}\xi\,dx+ o(1).
 \end{equation}
 Take $\xi=\partial_{\lambda}\tilde{U}$, then $$\aligned \int_{\R^N}
         \tilde{U}^{p}\xi\,dx =\int_{\R^N}
         \tilde{U}^{p}\partial_{\lambda}\tilde{U} \,dx\approx\partial_{\lambda}\Big( \int_{\R^N}
         U[0,\lambda]^{p+1} \,dx \Big)=\partial_{\lambda}\Big( \int_{\R^N}
         U[0,1]^{p+1} \,dx \Big)=0, \endaligned $$ and
         $$\aligned\int_{\R^N}
         \tilde{U}^{p-1}\xi\,dx &=\int_{\R^N}
         \tilde{U}^{p-1}\partial_{\lambda}\tilde{U} \,dx\approx\partial_{\lambda}\Big( \int_{\R^N}
         U[0,\lambda]^{p} \,dx \Big)\\
         &=\partial_{\lambda}\Big( \int_{\R^N} \lambda^{\frac{N-2}{2}p-N}
         U[0,1]^{p} \,dx \Big)= \frac{2-N}{2}\int_{\R^N}U^{p}\,dx \not = 0. \endaligned$$
Then it follows from \eqref{eq:3.33} that
         \begin{equation}\label{eq:3.35}
         \aligned
      o(1)\|\nabla \rho\|_{L^{2}}+\|\nabla \rho\|_{L^{2}}^{\min(\tilde{p},2)}+\Theta(u)
     &\gtrsim \big|\alpha-\alpha^{2\tilde{p}+1}\big|\cdot
     \Big|0-\theta \int_{\R^N}
         \tilde{U}^{p-1}\xi\,dx+o(1)\Big|\\
         &\gtrsim |\theta|\cdot \big|\alpha-\alpha^{2\tilde{p}+1}\big|.\\
         \endaligned
         \end{equation}
          Take $\xi=\tilde{U}$ in \eqref{eq:3.33}, by \eqref{eq:3.33} \eqref{eq:3.34} and \eqref{eq:3.35}, we have that
        $$
         \aligned
          &\big|\alpha-\alpha^{2\tilde{p}+1}\big|\cdot \Big|\frac{1}{2}\int_{\R^N} \tilde{U}^{p+1}  \,dx \Big|\\
          &\lesssim \big|\alpha-\alpha^{2\tilde{p}+1}\big|\cdot \Big|\int_{\R^N} \tilde{U}^{p+1} \,dx-\theta \int_{\R^N} \tilde{U}^{p} \,dx+o(1)\Big| + |\theta|\cdot \big|\alpha-\alpha^{2\tilde{p}+1}\big|\cdot\Big|\int_{\R^N} \tilde{U}^{p} \,dx\Big| \\
          &\lesssim   o(1)\|\nabla \rho\|_{L^{2}}+\|\nabla \rho\|_{L^{2}}^{\min(\tilde{p}, 2)}+\Theta(u).\\
         \endaligned
 $$
Therefore
\begin{equation}\label{eq:3.36}
\big|\alpha-\alpha^{2\tilde{p}+1}\big|\lesssim  o(1)\|\nabla \rho\|_{L^{2}}+\|\nabla \rho\|_{L^{2}}^{\min(\tilde{p},2) }+\Theta(u).
\end{equation}
Finally, we observe  that the function $g(\alpha)\coloneqq\frac{|\alpha-1|}{|\alpha-\alpha^{2\tilde{p}+1}|}$ is
      bounded in $\big[\frac{1}{2},\frac{3}{2}\big]\setminus\{1\}$, thus
      $$
          |\alpha-1|= \frac{|\alpha-1|}{|\alpha-\alpha^{2\tilde{p}+1}|} \cdot \big|\alpha-\alpha^{2\tilde{p}+1}\big|\lesssim  o(1)\|\nabla \rho\|_{L^{2}}+\|\nabla \rho\|_{L^{2}}^{\min(\tilde{p}, 2)}+\Theta(u).
  $$
We then  get  \eqref{eq:3.11} of Lemma \ref{prop:interaction_and_coef}.

\vskip0.12in
We now prove \eqref{eq:3.12} of Lemma \ref{prop:interaction_and_coef}. Since for any $j\neq i$, $\alpha_{j}\geq \frac{1}{2}$, using  \eqref{eq:3.28} with $\xi= \tilde{U}$ and $\eqref{eq:3.36}$, we obtain that
     \begin{equation}\label{eq:3.37}
         \aligned
         &\int_{\R^N}\tilde{U}^{p}\tilde{U}_{j}\Phi\,dx \lesssim \left|\int_{\R^N}\tilde{U}^{p}\tilde{V}\Phi\,dx\right|\lesssim \big(\tilde{p}\alpha^{2\tilde{p}}+o(1)\big) \left|\int_{\R^N}\tilde{U}^{p}\tilde{V}\Phi\,dx\right|\\
        &  \lesssim  \big|\alpha-\alpha^{2\tilde{p}+1}+o(1) \big|\Big|\int_{\mathbb{R}^{N}}  \tilde{U}^{p+1} \,dx\Big|+\big|\alpha-\alpha^{2\tilde{p}+1}+o(1) \big|\Big|\int_{\mathbb{R}^{N}}  \tilde{U}^{p} \tilde{U}(\Phi-1)\,dx\Big| \\
          &  \quad + \Big|\int_{\mathbb{R}^{N}}  \Big[(\alpha-\alpha^{2\tilde{p}+1}+o(1) )  \tilde{U}^{p} - (\tilde{p}\alpha^{2\tilde{p}}+o(1))\tilde{U}^{p-1}\tilde{V} \Big]  \tilde{U}\Phi\,dx \Big| \\
         & \lesssim  2(\alpha-\alpha^{2\tilde{p}+1} )  \|\tilde{U}\|_{L^{p+1}}^{p+1} + o(1)\|\nabla \rho\|_{L^{2}}+\|\nabla \rho\|_{L^{2}}^{\min(\tilde{p}, 2)}+\Theta(u)\\
         & \lesssim  o(1)\|\nabla \rho\|_{L^{2}}+\|\nabla \rho\|_{L^{2}}^{\min(\tilde{p}, 2)}+\Theta(u).
         \endaligned
     \end{equation}
     On the other hand, we recall that 
     $\Phi\geq 0 $ on $\mathbb{R}^{N}$
      and that $$\{\Phi=1\}=B(0,\varepsilon R)
      \setminus \bigcup_{j\in G}B(z_{j},
      \varepsilon^{-1}R_{j}),$$ where the 
      notations come from Lemma 
      \ref{lem:localization}. For 
      any  $\lambda_{j}\leq\lambda_{i}= 
      1 (\iff j>i)$, using the facts 
      $\tilde{U}\lesssim 1$, $\tilde{U}_{j}
      \lesssim \lambda_{j}^{\frac{N-2}{2}}$, 
      the Claim 3 and the choice of $R=\epsilon^{-2}\gg \varepsilon^{-2}$ in \eqref{eq:3.1}, we obtain that
 $$
         \aligned
         \int_{\R^N}\tilde{U}^{p}\tilde{U}_{j}\Phi\,dx &\geq \int_{B(0,\varepsilon R)\setminus \bigcup\limits_{j\in G}B(z_{j},\varepsilon^{-1}R_{j})}  \tilde{U}^{p}\tilde{U}_{j} \,dx =\int_{B(0,\varepsilon R)} \tilde{U}^{p}\tilde{U}_{j} \,dx -\sum_{j\in G}\int_{B(z_{j},\varepsilon^{-1}R_{j})} \tilde{U}^{p}\tilde{U}_{j} \,dx\\
         &\gtrsim \int_{B(0,\varepsilon R)} \tilde{U}^{p}\tilde{U}_{j} \,dx-\sum_{j\in G} \lambda_{j}^{\frac{N-2}{2}}(\varepsilon^{-1}R_{j})^{N} \geq \int_{B(0,\varepsilon^{-1})} \tilde{U}^{p}\tilde{U}_{j} \,dx-\nu (\varepsilon^{-1}\varepsilon^{2})^{N} \\
         &\geq \frac{1}{2}\int_{B(0,1)} \tilde{U}^{p}\tilde{U}_{j} \,dx
         \text{\quad \quad as}\;\varepsilon\ll 1.
         \endaligned
   $$
  Further, combining  with  \eqref{eq:3.37}, we deduce that for any $j>i$, it holds
   $$
         \aligned
         \int_{B(0,1)} \tilde{U}^{p}\tilde{U}_{j} \,dx\lesssim o(1)\|\nabla \rho\|_{L^{2}}+\|\nabla \rho\|_{L^{2}}^{\min(\tilde{p}, 2) }+\Theta(u).
         \endaligned
   $$
      Thanks to Lemma \ref{cor:interaction_integral_localized}, we deduce
       \eqref{eq:3.12} of Lemma  \ref{prop:interaction_and_coef} holds for all $\lambda_{j}\leq\lambda_{i}=1\; (j > i)$. If $i=1$, then we f{}inish the proof of   \eqref{eq:3.12} of Lemma  \ref{prop:interaction_and_coef}. If $i>1$, then by \eqref{eq:3.29} we also deduce
       \eqref{eq:3.12} holds for all $j<i$. In summary, \eqref{eq:3.12}   of Lemma  \ref{prop:interaction_and_coef} holds true for all  $j\neq i$.  We now f{}inish the proof of Lemma  \ref{prop:interaction_and_coef}.
     \end{proof}


\vskip0.3in

\section{Quantitative stability estimate}\label{sec-quantitative}

This section is devoted to the proof of Theorem \ref{thm:main_close}. Since there appears convolution type nonlinear term, some new estimates need to be established. Our aim is to show, in a neighborhood of a $\delta$-interacting family of Talenti bubbles, that the quantity $\big\|\Delta u+(I_\mu\ast|u|^{\tilde{p}+1} )|u|^{\tilde{p}-1}u\big\|_{(\mathcal{D}^{1,2})^{-1}}\eqqcolon \Theta(u)$ controls the $\mathcal{D}^{1,2}$-distance of $u$ from the manifold of sums of the Talenti bubbles. Let us brief{}ly explain the main steps of the proof. Firstly we consider the sum of the Talenti bubbles $\sigma$ that minimizes the distance from $u$. Then, setting $u=\sigma+\rho$ and testing  $\Delta u +(I_{\mu}\ast|u|^{2_\mu^*})|u|^{2_\mu^*-2}u$ against $\rho$. By doing so we obtain an estimate on $\|\nabla\rho\|_{L^2}$, namely \eqref{eq:4.4}. From there, we expand the nonlinear term on  the right hand side of \eqref{eq:4.4}, which has double integral nonlinearity, as the nineteen terms $(\mathcal{H}_{1}\sim \mathcal{H}_{19})$. We estimate them separately, where the Lemma \ref{prop:interaction_approx} and the interaction integral estimate \eqref{eq:3.12} of Lemma \ref{prop:interaction_and_coef} are the main tools.
\vskip0.18in
\noindent{\bf Proof of Theorem \ref{thm:main_close}.}
    Observe that the left hand side of  $\eqref{eq:1.18}$ and $\eqref{eq:1.19}$ is bounded from above, so up to enlarging the constant $C$, we can assume
    \begin{equation}\label{eq:4.1}
        \Theta(u)\leq 1.
    \end{equation}
Let $\sigma=\sum\limits_{i=1}^\nu \alpha_i \tilde{U}[z_i,\lambda_i]$ be the linear combination of the Talenti bubbles from \eqref{eq:1.14} which is closest to $u$ in the $\mathcal{D}^{1,2}(\R^N)$-norm, that is,
  $$ \big\|\nabla u-\nabla\sigma\big\|_{L^2} =
  \min_{\substack{\hat{\alpha}_1,\dots,\hat{\alpha}_\nu\in\R\\
  \hat{ z}_1,\dots,\dots,\hat{ z}_\nu\in\R^N\\
  \hat{ \lambda}_1,\dots,\hat{\lambda}_\nu>0}} \Big\|\nabla u-\nabla
  \Big(\sum_{i=1}^\nu \hat{\alpha}_i \tilde U
  [\hat{z}_i, \hat{\lambda}_i]\Big)\Big\|_{L^2}.$$
Let $\rho\coloneqq u-\sigma$ and denote $\tilde{U}_i\coloneqq \tilde{U}[z_i,\lambda_i]$. It follows directly from the condition \eqref{eq:1.17} that $\|\nabla\rho\|_{L^2}\le \delta$, which implies $\|\nabla\rho\|_{L^2}=o(1)$. By the def{}inition \eqref{eq:1.15}, the quantity \eqref{eq:2.11} satisf{}ies   $Q=o(1)$. Furthermore, since the bubbles $\hat{ U}_i$ are $\delta$-interacting, the family $(\alpha_i,\tilde{U}_i)_{1\le i \le \nu}$ is $\delta^{\prime}$-interacting for some $\delta'$ which goes to zero as $\delta$ goes to $0$.

  \vskip0.18in
Since $\sigma$ minimizes the $\mathcal{D}^{1,2}(\R^N)$-distance from $u$, $\rho$ is $\mathcal{D}^{1,2}(\R^N)$-orthogonal
 to the manifold composed of linear combinations of  the $\nu$ Talenti bubbles. Def{}ine
  $$h=h(\hat{\alpha}_1,\cdots,\hat{\alpha}_\nu,\hat{\lambda}_1,
  \cdots,\hat{\lambda}_\nu,\hat{z}_1,\cdots,\hat{z}_\nu)\coloneqq u- \sum_{i=1}^{\nu}
  \hat{\alpha}_i \tilde{U}[\hat{z}_{i},\hat{\lambda}_{i}] ,$$
  in particular $ h(\alpha_1,\cdots,\alpha_\nu,\lambda_1,\cdots,\lambda_\nu,z_1,\cdots,z_\nu)=\rho$, then by the def{}inition of $\sigma$ and $\rho$, $\|\nabla \rho\|^{2}_{L^{2}}$ is the minimum of $\|\nabla h\|^{2}_{L^{2}}$ , which is a function over $\mathbb{R}^{\nu(N+2)}$. Thus $(\alpha_1,\cdots,\alpha_\nu,\lambda_1,\cdots,\lambda_\nu,z_1,\cdots,z_\nu)$ is a critical point of the function $\|\nabla h\|^{2}_{L^{2}}$.  For f{}ixed $1\leq i\leq \nu $, let
$$\partial\coloneqq\partial_{\hat\alpha_{i}}\big|_{(\alpha_1,\cdots,\alpha_\nu,z_1,
\cdots,z_\nu,\lambda_1,\cdots,\lambda_\nu)}\ \ \hbox{or}\ \
\partial_{\hat \lambda_{i}}\big|_{(\alpha_1,\cdots,\alpha_\nu,z_1,
\cdots,z_\nu,\lambda_1,\cdots,\lambda_\nu)}\ \ \hbox{or}\ \
\partial_{(\hat z_{i})_{j}}\big|_{(\alpha_1,\cdots,\alpha_\nu,z_1,
\cdots,z_\nu,\lambda_1,\cdots,\lambda_\nu)},$$ where $1\leq j\leq N$, then
\begin{align}
0=\partial(\|\nabla h\|^{2}_{L^{2}})=\partial\int_{\mathbb{R}^{N}}|\nabla h|^{2}\,dx
=\int_{\mathbb{R}^{N}} \partial|\nabla h|^{2}\,dx=
2\int_{\mathbb{R}^{N}}\nabla \rho\cdot \nabla \partial h\,dx,
\end{align}
Further, the following $N+2$ orthogonality conditions hold:
  \begin{equation} \label{eq:4.2}
  \int_{\R^N}\nabla\rho\cdot\nabla \tilde{U}_i\,dx= 0,
  \
  \int_{\R^N}\nabla\rho\cdot\nabla \partial_\lambda \tilde{U}_i\,dx= 0,
  \
  \int_{\R^N}\nabla\rho\cdot\nabla \partial_{z_j}\tilde{U}_i\,dx= 0\ \  \text{for any}\ \  1\le j\le N.
  \end{equation}
Moreover,  by using integration by parts and noting that the functions $\tilde{U}_i, \partial_\lambda \tilde{U}_i, \partial_{z_j}\tilde{U}_i$ satisfy \eqref{eq:2.1}, \eqref{eq:2.4} and \eqref{eq:2.5} respectively, we know that the above orthogonality conditions are equivalent to
 \begin{equation}\label{eq:4.3}
   \int_{\R^N}(I_\mu\ast\tilde{U_{i}}^{\tilde{p}+1} )\tilde{U_{i}}^{\tilde{p}}\rho\,dx=0,
 \end{equation}
$$
      \int_{\R^N}\big[(\tilde{p}+1)(I_{\mu}\ast\tilde{U}_{i}^{\tilde{p}}\partial_\lambda\tilde{U}_{i})\tilde{U}_{i}^{\tilde{p}}+\tilde{p}(I_{\mu}\ast\tilde{U}_{i}^{\tilde{p}+1})\tilde{U}_{i}^{\tilde{p}-1}\partial_\lambda\tilde{U}_{i}\big]\, \rho\,dx= 0,
  $$
    $$
      \int_{\R^N}\big[  (\tilde{p}+1)(I_{\mu}\ast\tilde{U}_{i}^{\tilde{p}}\partial_{z_j}\tilde{U}_{i})\tilde{U}_{i}^{\tilde{p}}+\tilde{p}(I_{\mu}\ast\tilde{U}_{i}^{\tilde{p}+1})\tilde{U}_{i}^{\tilde{p}-1}\partial_{z_j}\tilde{U}_{i} \big]\, \rho\,dx= 0 \quad\text {for any}\ \ 1\le j\le N.
     $$
Our goal is to show that $\|\nabla\rho\|_{L^2}$ is controlled by
  $$ \big\|\Delta u +(I_\mu\ast|u|^{\tilde{p}+1} )
  |u|^{\tilde{p}-1}u\big\|_{(\mathcal{D}^{1,2})^{-1}}\eqqcolon \Theta(u).$$
  To achieve this, testing $\Delta u +(I_\mu\ast|u|^{\tilde{p}}u )
  |u|^{\tilde{p}-1}u$ by  $\rho$ and exploiting the orthogonality condition \eqref{eq:4.2} yield
  \begin{equation}\label{eq:4.4}
  \begin{aligned}
   \int_{\R^N}\abs{\nabla\rho}^2\,dx&=
   \int_{\R^N} \nabla (\sigma +\rho)\cdot\nabla\rho\,dx =\int_{\R^N} \nabla u\cdot\nabla\rho\,dx\\
   &=\int_{\R^N}(I_\mu\ast|u|^{\tilde{p}}u )|u|^{\tilde{p}-1}u\rho\,dx -\int_{\R^N}\left(
     \Delta u+(I_\mu\ast|u|^{\tilde{p}}u )
    |u|^{\tilde{p}-1}u\right)\rho\,dx\\
&\le\int_{\R^N}(I_\mu\ast|u|^{\tilde{p}}u )|u|^{\tilde{p}-1}u\rho\,dx + \Theta(u)\|\nabla\rho\|_{L^2}.
    \end{aligned}
  \end{equation}
To control the f{}irst term
  $$
  \int_{\R^N}(I_\mu\ast|u|^{\tilde{p}}u )
      |u|^{\tilde{p}-1}u\rho\,dx,
   $$
  we apply \eqref{eq:2.13} and  \eqref{eq:2.18} with $r=\tilde{p}+1$ (here we require $\mu\leq N+2$), $l=0$, $a=\sigma$, $b=\rho$  and $a_i=\alpha_i \tilde{U}_i$ and  deduce that
  \begin{equation}\label{eq:4.5}
   \Bigl||u|^{\tilde{p}}u - \sum_{i=1}^\nu\alpha_i|\alpha_i|^{\tilde{p}}
  \tilde{U}_i^{\tilde{p}+1}-(\tilde{p}+1)\sigma^{\tilde{p}}\rho\Bigr|\le
   C_{N,\nu}\Bigl(\{\sigma^{\tilde{p}-1}|\rho|^2\}_{\tilde{p}+1>2} + |\rho|^{\tilde{p}+1} +
  \sum_{ i\not= j}
  \tilde{U}_i^{\tilde{p}}\tilde{U}_j\Bigr).
  \end{equation}
Similarly, we combine  \eqref{eq:2.13} and \eqref{eq:2.18} with $r=\tilde{p}$ (we need that $\mu\leq 4$), $l=0$, $a=\sigma$, $b=\rho,$
and $a_i=\alpha_i \tilde{U}_i$, it follows that
  \begin{equation}\label{eq:4.6}
   \Bigl||u|^{\tilde{p}-1}u - \sum_{i=1}^\nu\alpha_i|\alpha_i|^{\tilde{p}-1}
  \tilde{U}_i^{\tilde{p}}-\tilde{p}\sigma^{\tilde{p}-1}\rho \Bigr|\le C_{N,\nu}\Bigl(\{\sigma^{\tilde{p}-2}|\rho|^2\}_{\tilde{p}>2} + |\rho|^{\tilde{p}} +
   \sum_{ i\not= j}
  \tilde{U}_i^{\tilde{p}-1}\tilde{U}_j\Bigr),
 \end{equation}
 here we recall  the following equivalent conditions:
  $$\tilde{p}+1>2\iff \mu<4, \quad\quad\quad   \tilde{p}>2\iff \mu<6-N ( \hbox{see   Table} \ref{tab:Equivalent with-p wan-mu}).$$
 Therefore,  we f{}ind that
 \begin{equation}\label{eq:4.7}
   \aligned
 &\ \int_{\R^N} (I_\mu\ast |u|^{\tilde{p}}u )|u|^{\tilde{p}-1}u\rho\,dx\\
 & =   \int_{\R^N}\Big(I_{\mu}\ast \bigl(|u|^{\tilde{p}}u - \sum_{i=1}^\nu\alpha_i|\alpha_i
    |^{\tilde{p}}\tilde{U}_i^{\tilde{p}+1}
    \bigr)\Big)\Big(|u|^{\tilde{p}-1}u - \sum_{i=1}^\nu\alpha_i|\alpha_i|^{\tilde{p}-1}
    \tilde{U}_i^{\tilde{p}}   \Big)\rho\,dx \\
 &\quad +  \int_{\R^N}\Big(I_{\mu}\ast\big(|u|^{\tilde{p}}u - \sum_{i=1}^\nu\alpha_i|\alpha_i|^{\tilde{p}}
    \tilde{U}_i^{\tilde{p}+1}
    \big)\Big)\Big( \sum_{i=1}^\nu\alpha_i|\alpha_i|^{\tilde{p}-1}
    \tilde{U}_i^{\tilde{p}}   \Big)\rho\,dx\\
 &\quad +  \int_{\R^N}\Big(I_{\mu}\ast\sum_{i=1}^\nu\alpha_i|\alpha_i|^{\tilde{p}}
    \tilde{U}_i^{\tilde{p}+1}\Big)\Big(|u|^{\tilde{p}-1}u - \sum_{i=1}^\nu\alpha_i|\alpha_i|^{\tilde{p}-1}
    \tilde{U}_i^{\tilde{p}}   \Big)\rho\,dx\\
 &\quad + \int_{\R^N}\Big(I_{\mu}\ast\sum_{i=1}^\nu\alpha_i|\alpha_i|^{\tilde{p}}
    \tilde{U}_i^{\tilde{p}+1}\Big)\Big( \sum_{i=1}^\nu\alpha_i|\alpha_i|^{\tilde{p}-1}
    \tilde{U}_i^{\tilde{p}}   \Big)\rho\,dx \\
 &\lesssim    \int_{\R^N}\Big(I_{\mu}\ast\big(\sigma^{\tilde{p}}|\rho|+
    \{\sigma^{\tilde{p}-1}|\rho|^2\}_{\tilde{p}+1>2} + |\rho|^{\tilde{p}+1} +
    \sum\limits_{i\not= j}
    \tilde{U}_i^{\tilde{p}}\tilde{U}_j\big)\Big)  \Big(\sigma^{\tilde{p}-1}|\rho|+ \{
    \sigma^{\tilde{p}-2}|\rho|^2\}_{\tilde{p}>2}\\
 & \quad  + |\rho|^{\tilde{p}}  +\sum_{ i\not= j}
    \tilde{U}_i^{\tilde{p}-1}\tilde{U}_j  \Big)|\rho|\,dx + \int_{\R^N}\Big(I_{\mu}\ast\big(|u|^{\tilde{p}}u - \sum_{i=1}^\nu\alpha_i|\alpha_i|^{\tilde{p}}
    \tilde{U}_i^{\tilde{p}+1}
    \big)\Big)\Big( \sum_{i=1}^\nu\alpha_i|\alpha_i|^{\tilde{p}-1}
    \tilde{U}_i^{\tilde{p}}   \Big)\rho\,dx \\
 & \quad+  \int_{\R^N}\Big(I_{\mu}\ast\sum_{i=1}^\nu\alpha_i|\alpha_i|^{\tilde{p}}
    \tilde{U}_i^{\tilde{p}+1}\Big)\Big(|u|^{\tilde{p}-1}u - \sum_{i=1}^\nu\alpha_i|\alpha_i|^{\tilde{p}-1}
    \tilde{U}_i^{\tilde{p}}   \Big)\rho\,dx\\
 &\quad +  \int_{\R^N}\Big(I_{\mu}\ast\sum_{i=1}^\nu\alpha_i|\alpha_i|^{\tilde{p}}
    \tilde{U}_i^{\tilde{p}+1}\Big)\Big( \sum_{i=1}^\nu\alpha_i|\alpha_i|^{\tilde{p}-1}
    \tilde{U}_i^{\tilde{p}}   \Big)\rho\,dx \\
 &  \approx    \int_{\R^N} (I_{\mu}\ast\sigma^{\tilde{p}}|\rho| ) \sigma^{\tilde{p}-1}|\rho|^2  \,dx
        +\Big\{\int_{\R^N}  (I_\mu\ast\sigma^{\tilde{p}}|\rho| ) \sigma^{\tilde{p}-2}|\rho|^3   \,dx\Big\}_{\tilde{p}>2}
         +\int_{\R^N}  (I_\mu\ast\sigma^{\tilde{p}}|\rho| )|\rho|^{\tilde{p}+1}
        \,dx\\
 &\quad +\sum_{ i\not= j}\int_{\R^N} (I_\mu\ast \sigma^{\tilde{p}}|\rho| )
        \tilde{U}_i^{\tilde{p}-1}\tilde{U}_j|\rho|  \,dx +\Big\{\int_{\R^N} (I_\mu\ast
        \sigma^{\tilde{p}-1}|\rho|^2 )\sigma^{\tilde{p}-1}|\rho|^2  \,dx\Big\}_{\tilde{p}+1>2}    \\
 &\quad +\Big\{\int_{\R^N}  (I_\mu\ast \sigma^{\tilde{p}-1}|\rho|^2 )
        \sigma^{\tilde{p}-2}|\rho|^3   \,dx \Big\}_{\tilde{p}+1>2,\, \tilde{p}>2}
          + \Big\{\int_{\R^N}  (I_\mu\ast 
          \sigma^{\tilde{p}-1}|\rho|^2 )
          |\rho|^{\tilde{p}+1}  \,dx \Big\}_{\tilde{p}
          +1>2}\\
 \ &\quad +\sum_{ i\not= j} \Big\{\int_{\R^N}  (I_\mu\ast  \sigma^{\tilde{p}-1}|\rho|^2 )
          \tilde{U}_i^{\tilde{p}-1}\tilde{U}_j|\rho| 
           \,dx\Big\}_{\tilde{p}+1>2}  
           +\int_{\R^N} (I_\mu\ast
            |\rho|^{\tilde{p}+1} )
            \sigma^{\tilde{p}-1}|\rho|^2  \,dx
           \endaligned
    \end{equation}
   $$
   \aligned
 \ &\quad +\Big\{\int_{\R^N}  ( I_{\mu}\ast |\rho|^{\tilde{p}+1}  )\sigma^{\tilde{p}-2}|\rho|^3
        \,dx\Big\}_{\tilde{p}>2}  +\int_{\R^N} ( I_{\mu}\ast |\rho|^{\tilde{p}+1} )|\rho|^{\tilde{p}+1}  \,dx  \\
\ &\quad +\sum_{ i\not= j}\int_{\R^N} ( I_{\mu}\ast
           |\rho|^{\tilde{p}+1} )
        \tilde{U}_i^{\tilde{p}-1}\tilde{U}_j|\rho|  \,dx
       +\sum_{ i\not= j}\int_{\R^N}
         \bigl(I_{\mu}\ast
        \tilde{U}_i^{\tilde{p}}\tilde{U}_j \bigr)
        \sigma^{\tilde{p}-1}|\rho|^2  \,dx   \\
        \ & +\sum_{i\neq j}\Big\{\int_{\R^N}  \big(I_\mu\ast \tilde{U}_i^{\tilde{p}}\tilde{U}_j   \big)
        \sigma^{\tilde{p}-2}|\rho|^3   \,dx\Big\}_{\tilde{p}>2} +\sum_{i\neq j}\int_{\R^N}
        \Big( I_{\mu}\ast \tilde{U}_i^{\tilde{p}}\tilde{U}_j  \Big)
        |\rho|^{\tilde{p}+1}  \,dx  \\
  \ & +\sum_{i\not= j,k\not= l }   \int_{\R^N}  \big( I_{\mu}\ast
        \tilde{U}_i^{\tilde{p}}\tilde{U}_j \big)
        \tilde{U}_k^{\tilde{p}-1}\tilde{U}_l|\rho|  \,dx  \\
              \ & + \int_{\R^N}\Big(I_{\mu}\ast\big(|u|^{\tilde{p}}u - \sum_{i=1}^\nu\alpha_i|\alpha_i|^{\tilde{p}}
    \tilde{U}_i^{\tilde{p}+1}
    \big)\Big)\Big( \sum_{i=1}^\nu\alpha_i|\alpha_i|^{\tilde{p}-1}
    \tilde{U}_i^{\tilde{p}}   \Big)\rho\,dx\\
    \ &  +  \int_{\R^N}\Big(I_{\mu}\ast\sum_{i=1}^\nu\alpha_i|\alpha_i|^{\tilde{p}}
    \tilde{U}_i^{\tilde{p}+1}\Big)\Big(|u|^{\tilde{p}-1}u - \sum_{i=1}^\nu\alpha_i|\alpha_i|^{\tilde{p}-1}
    \tilde{U}_i^{\tilde{p}}   \Big)\rho\,dx\\
    \ & +  \sum_{i\neq j}\alpha_i|\alpha_i|^{\tilde{p}}\alpha_j|\alpha_j|^{\tilde{p}-1} \int_{\R^N}\Big(I_{\mu}\ast
    \tilde{U}_i^{\tilde{p}+1}\Big)
    \tilde{U}_j^{\tilde{p}}  \rho\,dx \\
    \ \eqqcolon&\   \mathcal{H}_1+\cdots +\mathcal{H}_{19},
        \endaligned
 $$
  where we have used the orthogonality condition \eqref{eq:4.3} in $\mathcal{H}_{19}$.
  \vskip0.25in
  $\bullet$ The above terms $\mathcal{H}_1\sim\mathcal{H}_{16}$   can be controlled easily by using the method of Lemma \ref{double-integral-estimate}.
  Applying the HLS inequality \eqref{eq:1.9} with $ s=t=\frac{2N}{2N-\mu}=\frac{p+1}{\tilde{p}+1}$  (see the Remark \ref{rmk3-double-integral-estimate}), H\"older inequality and Sobolev inequality,  we get that
  $$\aligned \mathcal{H} _1&=\int_{\R^N}(I_\mu\ast\sigma^{\tilde{p}}|\rho|)\sigma^{\tilde{p}-1}|\rho|^2  \,dx
 \lesssim \|\sigma^{\tilde{p}}|\rho|\|_{L^{\frac{p+1}{\tilde{p}+1}}}
 \|\sigma^{\tilde{p}-1}|\rho|^2
 \|_{L^{\frac{p+1}{\tilde{p}+1}}}\\
 &\lesssim \|\sigma\|_{L^{p+1}}^{\tilde{p}}
 \|\rho\|_{L^{p+1}} \|\sigma\|_{L^{p+1}}^{\tilde{p}-1}
 \|\rho\|_{L^{p+1}}^2 \lesssim \|\nabla\rho\|_{L^2}^3,\ \mu<4.
 \endaligned $$
 Note that the range of $\mu$ is the pre-condition  of using H\"older inequality (see the Remark \ref{rmk2-double-integral-estimate} and the Table \ref{tab:Equivalent with-p wan-mu}).
 Similar to the estimate of $\mathcal{A}_1$, we have that
 $$\aligned\mathcal{H} _2= \Big\{\int_{\R^N}(I_\mu\ast\sigma^{\tilde{p}}|\rho|)\sigma^{\tilde{p}-2}|\rho|^3\,dx\Big\}_{\tilde{p}>2}
 &\lesssim \Big\{\|\sigma^{\tilde{p}}
 |\rho|\|_{L^{\frac{p+1}{\tilde{p}+1}}}\|\sigma^{\tilde{p}-2}|\rho|^3\|_{L^{\frac{p+1}{\tilde{p}+1}}}\Big\}_{\mu<6-N}\\
 &\lesssim  \big\{\|\sigma\|_{L^{p+1}}^{2\tilde{p}-2}
 \|\rho\|_{L^{p+1}}^4\big\}_{\mu<6-N}\lesssim \big\{\|\nabla\rho\|_{L^2}^4\big\}_{\mu<6-N}, \endaligned $$
 $$\aligned \mathcal{H}_3 =\int_{\R^N}(I_\mu\ast\sigma^{\tilde{p}}|\rho|)|\rho|^{\tilde{p}+1}
 &\lesssim   \|\sigma^{\tilde{p}}|\rho|\|_{L^{\frac{p+1}{\tilde{p}+1}}}
 \||\rho|^{\tilde{p}+1}\|_{L^{\frac{p+1}{\tilde{p}+1}}}\\
 &\lesssim   \|\sigma\|_{L^{p+1}}^{\tilde{p}}
 \|\rho\|_{L^{p+1}}^{\tilde{p}+2}\lesssim \|\nabla\rho\|_{L^2}^{\tilde{p}+2},\ \mu<N+2,\endaligned $$
 $$\aligned\mathcal{H}_5  =\Big\{\int_{\R^N}(I_\mu\ast
 \sigma^{\tilde{p}-1}|\rho|^2)\sigma^{\tilde{p}-1}|\rho|^2  \,dx  \Big\}_{\tilde{p}+1>2}
&\lesssim \Big\{ \|\sigma^{\tilde{p}-1}|\rho|^2\|_{L^{\frac{p+1}{\tilde{p}+1}}}
 \| \sigma^{\tilde{p}-1}|\rho|^2\|_{L^{\frac{p+1}{\tilde{p}+1}}}\Big\}_{\mu<4}\\
 &\lesssim \big\{\|\sigma\|_{L^{p+1}}^{2\tilde{p}-2}
\|\rho\|_{L^{p+1}}^{4}\big\}_{\mu<4}\lesssim \big\{\|\nabla\rho\|_{L^2}^{4}\big\}_{\mu<4}, \endaligned$$
 $$\aligned\mathcal{H} _6=\Big\{
 \int_{\R^N} (I_\mu\ast \sigma^{\tilde{p}-1}|\rho|^2)
 \sigma^{\tilde{p}-2}|\rho|^3\,dx \Big\}_{\tilde{p}+1>2,\tilde{p}>2}
 &\lesssim  \Big\{\|\sigma^{\tilde{p}-1}|\rho|^2\|_{L^{\frac{p+1}{\tilde{p}+1}}}
 \|\sigma^{\tilde{p}-2}|\rho|^3\|_{L^{\frac{p+1}{\tilde{p}+1}}}\Big\}_{\mu<6-N}\\
 &\lesssim
 \big\{\|\sigma\|_{L^{p+1}}^{2\tilde{p}-3}\|\rho\|_{L^{p+1}}^{5}\big\}_{\mu<6-N}\lesssim \big\{\|\nabla\rho\|_{L^2}^{5}\big\}_{\mu<6-N}, \endaligned$$
 $$\aligned\mathcal{H} _7=\Big\{\int_{\R^N}(I_\mu\ast \sigma^{\tilde{p}-1}|\rho|^2)|\rho|^{\tilde{p}+1}  \,dx\Big\}_{\tilde{p}+1>2}
 &\lesssim  \Big\{\|\sigma^{\tilde{p}-1}|\rho|^2\|_{L^{\frac{p+1}{\tilde{p}+1}}}
 \||\rho|^{\tilde{p}+1}\|_{L^{\frac{p+1}{\tilde{p}+1}}}\Big\}_{\mu<4}\\
 &\lesssim  \big\{\|\sigma\|_{L^{p+1}}^{\tilde{p}-1}
 \|\rho\|_{L^{p+1}}^{\tilde{p}+3}\big\}_{\mu<4}\lesssim \big\{\|\nabla\rho\|_{L^2}^{\tilde{p}+3}\big\}_{\mu<4},\endaligned $$
 $$\aligned\mathcal{H} _9=
 \int_{\R^N}(I_\mu\ast |\rho|^{\tilde{p}+1})\sigma^{\tilde{p}-1}|\rho|^2  \,dx
 &\lesssim  \||\rho|^{\tilde{p}+1}\|_{L^{\frac{p+1}{\tilde{p}+1}}}
 \|\sigma^{\tilde{p}-1}|\rho|^2\|_{L^{\frac{p+1}{\tilde{p}+1}}}\\
 &\lesssim  \|\sigma\|_{L^{p+1}}^{
   \tilde{p}-1}\|\rho\|_{L^{p+1}}^{\tilde{p}+3}\lesssim \|\nabla\rho\|_{L^2}^{\tilde{p}+3},\  \mu<4,\endaligned$$
 $$\aligned\mathcal{H} _{10}=\Big\{\int_{\R^N}(I_\mu\ast
 |\rho|^{\tilde{p}+1}) \sigma^{\tilde{p}-2}|\rho|^3\Big\}_{\tilde{p}>2}
 &\lesssim  \Big\{\||\rho|^{\tilde{p}+1}\|_{L^{\frac{p+1}{\tilde{p}+1}}}
 \|\sigma^{\tilde{p}-2}|\rho|^3\|_{L^{\frac{p+1}{\tilde{p}+1}}}\Big\}_{\mu<6-N}\\
 &\lesssim  \big\{\|\sigma\|_{L^{p+1}}^{\tilde{p}-2}
 \|\rho\|_{L^{p+1}}^{\tilde{p}+4}\big\}_{\mu<6-N}  \lesssim\big\{ \|\nabla\rho\|_{L^2}^{\tilde{p}+4}\big\}_{\mu<6-N},\endaligned$$
 $$\aligned\mathcal{H} _{11}&=
 \int_{\R^N}( I_{\mu}\ast |\rho|^{\tilde{p}+1}) |\rho|^{\tilde{p}+1}\,dx
 \lesssim  \||\rho|^{\tilde{p}+1}\|_{L^{\frac{p+1}{\tilde{p}+1}}}^2\lesssim \|\rho\|_{L^{p+1}}^{2(\tilde{p}+1)}
 \lesssim\|\nabla\rho\|_{L^2}^{2(\tilde{p}+1)}.
 \endaligned$$
 In order to estimate the other terms, let us note  that if $\mu<4$(iff $\tilde{p}-1>0$), then by Lemma
 \ref{prop:interaction_approx}, for
 any $i\not=j$ it holds that
 \begin{equation}\label{eq:4.8}
 \aligned
 \|\tilde{U}_{i}^{\tilde{p}-1}\tilde{U}_{j}\|_{L^{\frac{p+1}{\tilde{p}}}}&=
 \bigg(\int_{\mathbb{R}^N}\tilde{U}_i^{\frac{(\tilde{p}-1)(p+1)}{\tilde{p}}}
 \tilde{U}_j^{\frac{p+1}{\tilde{p}}}\,dx
 \bigg)^{\frac{\tilde{p}}{p+1}}
 \approx \begin{cases}
     Q^{\frac{N-2}{2}\min (\tilde{p}-1,1)}, &\text{if } \tilde{p}-1\neq 1\\
     \Big(Q^{\frac{N}{2}} \log(\frac{1}{Q})\Big)^{\frac{\tilde{p} }{ p+1}}, &\text{if } \tilde{p}-1=1
     \end{cases}\\
 &\approx \   o(1),
 \endaligned
 \end{equation}
 \begin{equation}\label{eq:4.9}
   \aligned
   \|\tilde{U}_{i}^{\tilde{p}}\tilde{U}_{j}\|_{L^{\frac{p+1}{\tilde{p}+1}}}&=\bigg(\int_{\mathbb{R}^N}\tilde{U}_i^{\frac{\tilde{p}(p+1) }{\tilde{p}+1}}
   \tilde{U}_j^{\frac{p+1}{\tilde{p}+1}}\,dx\bigg)^{\frac{\tilde{p}+1}{p+1}}
   \approx \Big(Q^{\frac{N-2}{2}
   \min (\frac{\tilde{p}(p+1) }{\tilde{p}+1},\frac{p+1}{\tilde{p}+1} )}\Big)^{\frac{\tilde{p}+1}{p+1}}\\
   &\approx \Big(Q^{\frac{N-2}{2}\frac{p+1}{\tilde{p}+1}}\Big)^{\frac{\tilde{p}+1}{p+1}}
=Q^{\frac{N-2}{2}}\approx o(1).\endaligned
   \end{equation}
Thus by \eqref{eq:4.9} and \eqref{eq:2.12}, we have that
\begin{equation}\label{eq:4.10}
\aligned\|\tilde{U}_{i}^{\tilde{p}}\tilde{U}_{j}\|_{L^{\frac{p+1}{\tilde{p}+1}}}\approx Q^{\frac{N-2}{2}}\approx \int_{\R^N} \tilde{U}_{i}^{p} \tilde{U}_{j}\,dx.\endaligned
   \end{equation}
 Therefore, using \eqref{eq:4.8}, \eqref{eq:4.9}, the HLS inequality \eqref{eq:1.9} with $ s=t=\frac{2N}{2N-\mu}=\frac{p+1}{\tilde{p}+1}$, H\"older and Sobolev inequalities, we obtain that
 $$\aligned\mathcal{H}_{4}&=
 \sum_{i\not=j}\int_{\R^N}(I_\mu\ast\sigma^{\tilde{p}}
 |\rho|)
 \tilde{U}_i^{\tilde{p}-1}\tilde{U}_j|\rho|\,dx
 \lesssim \|\sigma^{\tilde{p}}|\rho|\|_{L^{\frac{p+1}{\tilde{p}+1}}}
 \|\tilde{U}_i^{\tilde{p}-1}\tilde{U}_j|\rho|\|_{L^{\frac{p+1}{\tilde{p}+1}}}\\
 &\lesssim\|\sigma\|_{L^{p+1}}^{\tilde{p}}\|\rho\|_{L^{p+1}}
 \|\tilde{U}_{i}^{\tilde{p}-1}\tilde{U}_{j}\|_{L^{\frac{p+1}{\tilde{p}}}}\|\rho\|_{L^{p+1}}\lesssim o(1)\|\nabla\rho\|_{L^2}^{2},\  \mu<N+2\text{ and } \mu<4,\endaligned$$
 $$\aligned\mathcal{H} _{8}&=
 \sum_{i\not=j} \Big\{\int_{\R^N}(I_\mu\ast\sigma^{\tilde{p}-1}|\rho|^2)
 \tilde{U}_i^{\tilde{p}-1}\tilde{U}_j|\rho|\,dx\Big\}_{\tilde{p}+1>2}
 \lesssim\Big\{\|\sigma^{\tilde{p}-1}|\rho|^2\|_{L^{\frac{p+1}{\tilde{p}+1}}}
 \|\tilde{U}_i^{\tilde{p}-1}\tilde{U}_j|\rho|\|_{L^{\frac{p+1}{\tilde{p}+1}}}\Big\}_{\mu<4}\\
 &\lesssim\big\{\|\sigma\|_{L^{p+1}}^{\tilde{p}-1}\|\rho\|_{L^{p+1}}^2
 \|\tilde{U}_{i}^{\tilde{p}-1}\tilde{U}_{j}\|_{L^{\frac{p+1}{\tilde{p}}}}\|\rho\|_{L^{p+1}}\big\}_{\mu<4} \lesssim \big\{o(1) \|\nabla\rho\|_{L^2}^{3}\big\}_{\mu<4}, \endaligned$$
 $$\aligned\mathcal{H} _{12}&=
  \sum_{i\not=j}\int_{\R^N}\left( I_{\mu}\ast
 |\rho|^{\tilde{p}+1} \right)\tilde{U}_i^{\tilde{p}-1}\tilde{U}_j|\rho|  \,dx
 \lesssim\||\rho|^{\tilde{p}+1}\|_{L^{\frac{p+1}{\tilde{p}+1}}}
 \|\tilde{U}_i^{\tilde{p}-1}\tilde{U}_j|\rho|\|_{L^{\frac{p+1}{\tilde{p}+1}}}\\
 &\lesssim\|\rho\|_{L^{p+1}}^{\tilde{p}+2}
 \|\tilde{U}_{i}^{\tilde{p}-1}\tilde{U}_{j}\|_{L^{\frac{p+1}{\tilde{p}}}} \lesssim \|\rho\|_{L^{p+1}}^{\tilde{p}+2}
 \|\tilde{U}_{i}\|_{L^{p+1}}^{\tilde{p}-1}\|\tilde{U}_{j}\|_{L^{p+1}}
 \lesssim \|\nabla\rho\|_{L^2}^{\tilde{p}+2},\  \mu<N+2\text{ and } \mu<4,\endaligned $$
 $$\aligned\mathcal{H} _{13}&=
 \sum_{i\not=j}\int_{\R^N}
 \big(I_\mu\ast\tilde{U}_i^{\tilde{p}}\tilde{U}_j\big)\sigma^{\tilde{p}-1}|\rho|^2\,dx\lesssim
 \|\tilde{U}_i^{\tilde{p}}\tilde{U}_j\|_{L^{\frac{p+1}{\tilde{p}+1}}}
 \|\sigma^{\tilde{p}-1}|\rho|^2\|_{L^{\frac{p+1}{\tilde{p}+1}}}\\
 &\lesssim \|\tilde{U}_i^{\tilde{p}}\tilde{U}_j\|_{
 L^{\frac{p+1}{\tilde{p}+1}}}\|\sigma\|_{L^{p+1}}^{\tilde{p}-1}\|\rho\|_{L^{p+1}}^2
  \lesssim o(1)\|\nabla\rho\|_{L^2}^{2},\  \mu<4,\endaligned $$
 $$\aligned\mathcal{H} _{14}&=\sum_{i\not=j}\Big\{
 \int_{\R^N}  \big(I_\mu\ast \tilde{U}_i^{\tilde{p}}\tilde{U}_j  \big)
 \sigma^{\tilde{p}-2}|\rho|^3   \,dx\Big\}_{\tilde{p}>2}
 \lesssim \Big\{\|\tilde{U}_i^{
 \tilde{p}}\tilde{U}_j \|_{L^{\frac{p+1}{\tilde{p}+1}}}
 \|\sigma^{\tilde{p}-2}|\rho|^3\|_{L^{\frac{p+1}{\tilde{p}+1}}}\Big\}_{\mu<6-N}\\
 &\lesssim \big\{\|\tilde{U}_i^{\tilde{p}}\tilde{U}_j\|_{
 L^{\frac{p+1}{\tilde{p}+1}}} \|\sigma\|_{L^{p+1}}^{\tilde{p}-2}\|\rho\|_{L^{p+1}}^3\big\}_{\mu<6-N}
 \lesssim \big\{o(1)\|\nabla\rho\|_{L^2}^{3}\big\}_{\mu<6-N},\endaligned$$
 $$\aligned\mathcal{H} _{15}&=\sum_{i\not=j}
  \int_{\R^N}\big(I_{\mu}
  \ast\tilde{U}_i^{\tilde{p}}\tilde{U}_j\big)|\rho|^{\tilde{p}+1}\,dx
 \lesssim\|\tilde{U}_i^{
 \tilde{p}}\tilde{U}_j\|_{L^{\frac{p+1}{\tilde{p}+1}}}
 \||\rho|^{\tilde{p}+1}\|_{L^{\frac{p+1}{\tilde{p}+1}}}\\
 &\approx\|\tilde{U}_i^{\tilde{p}}\tilde{U}_j\|_{
 L^{\frac{p+1}{\tilde{p}+1}}}\|\rho\|_{L^{p+1}}^{\tilde{p}+1}
 \lesssim o(1)
 \|\nabla\rho\|_{L^2}^{\tilde{p}+1}.\endaligned$$
 Therefore, if $N\geq3$ and $\mu<4$, then all the above terms   are of  $o(1)\|\nabla \rho\|_{L^{2}}^{2}$ when  the exponents are strictly larger than 2, namely
  $$
     \mathcal{H}_{1}+\cdots +\mathcal{H}_{15}\lesssim o(1)\|\nabla \rho\|_{L^{2}}^{2}.
 $$
We remark that the above estimates are heavily dependent on the condition $\mu<4$.
 Next, according to the HLS inequality \eqref{eq:1.9} with $ s=t=\frac{2N}{2N-\mu}=\frac{p+1}{\tilde{p}+1}$, H\"older and Sobolev inequalities, \eqref{eq:4.10}, \eqref{eq:4.8}, and the estimate \eqref{eq:3.12} in Lemma \ref{prop:interaction_and_coef} (here we need $3\leq N\leq 5$ and $\frac{N+2}{2}<\mu<\min(N,4)$),  we obtain  that
 $$\aligned\mathcal{H}_{16}= &\ \sum_{i\not= j,k\not= l} \int_{\R^N}\big(I_\mu\ast\tilde{U}_i^{\tilde{p}}\tilde{U}_j\big)
 \tilde{U}_k^{\tilde{p}-1}\tilde{U}_l|\rho|\,dx\lesssim
 \|\tilde{U}_i^{\tilde{p}}\tilde{U}_j\|_{
 L^{\frac{p+1}{\tilde{p}+1}}}
 \|\tilde{U}_k^{\tilde{p}-1}\tilde{U}_l|\rho|\|_{L^{\frac{p+1}{\tilde{p}+1}}}\\
 \ \lesssim&\
 \|\tilde{U}_i^{\tilde{p}}\tilde{U}_j\|_{
 L^{\frac{p+1}{\tilde{p}+1}}} \|\tilde{U}_{k}^{\tilde{p}-1}\tilde{U}_{l}\|_{L^{\frac{p+1}{\tilde{p}}}}  \|\rho\|_{L^{p+1}}  \lesssim \Big(\int_{\R^N} \tilde{U}_{i}^{p} \tilde{U}_{j} \,dx\Big)o(1) \|\nabla\rho\|_{L^2} \\
 \ \lesssim&\ o(1) \left(\hat\varepsilon\|\nabla\rho\|_{L^2}
 +\|\nabla \rho\|_{L^{2}}^{\min(\tilde{p}, 2) }+\Theta(u)\right) \|\nabla\rho\|_{L^2}.\\
 \endaligned$$
 Thus,  if  $3\leq N\leq 5$ and $\frac{N+2}{2}<\mu<\min(N,4)$ ( it implies $\tilde{p}>1$), then
 $$
     \aligned
     \mathcal{H}_{16}\lesssim \hat\varepsilon\|\nabla\rho\|_{L^2}^{2}+ o(1)\|\nabla\rho\|_{L^2}^{2}+\Theta(u)\|\nabla\rho\|_{L^2}.
     \endaligned
  $$
 Summing up, if $3\leq N\leq 5$ and $\frac{N+2}{2}<\mu<\min(N,4)$, then
 \begin{equation}\label{eq:4.11}
     \mathcal{H}_{1}+\cdots+\mathcal{H}_{16}\lesssim \hat\varepsilon\|\nabla\rho\|_{L^2}^{2}+o(1)\|\nabla\rho\|_{L^2}^{2}+\Theta(u)\|\nabla\rho\|_{L^2}.
 \end{equation}
 \vskip0.05in
$\bullet$ Next we consider $\mathcal{H} _{19}$. By  applying  the relation \eqref{eq:2.3} and the H\"older inequality, we have that
$$
\aligned \   \mathcal{H}_{19}&=\sum_{i\neq j}\alpha_i|\alpha_i
  |^{\tilde{p}}\alpha_j|\alpha_j|^{\tilde{p}-1}\int_{\R^N}\Big(I_\mu\ast
      \tilde{U}_i^{\tilde{p}+1}\Big)\tilde{U}_j^{\tilde{p}} \rho\,dx \\
      &\approx \sum_{i\neq j}\int_{\R^N}\tilde{U}_i^{p-\tilde{p}}\tilde{U}_j^{\tilde{p}} \rho\,dx
      \lesssim \max_{i\neq j}  \|\tilde{U}_i^{p-\tilde{p}}\tilde{U}_j^{\tilde{p}}
      \|_{L^{\frac{p+1}{p}}} \| \nabla\rho\|_{L^{2}}.
\endaligned
 $$
For $i\neq j$, noting that $\mu>\frac{N+2}{2}$ is equivalent to $\frac{(p-\tilde{p})(p+1)}{p}>\frac{\tilde{p}(p+1)}{p}$, we combine  Lemma \ref{prop:interaction_approx} and Lemma \ref{prop:interaction_and_coef} (here we ask $3\leq N\leq 5$ and $\frac{N+2}{2}<\mu<\min(N,4)$) and  get that
$$
   \aligned \ &\  \|\tilde{U}_i^{p-\tilde{p}}\tilde{U}_j^{\tilde{p}}\|_{L^{\frac{p+1}{p}}} \approx \Big(Q^{\frac{N-2}{2}\min(\frac{(p-\tilde{p})(p+1)}{p},\frac{\tilde{p}(p+1)}{p} )}\Big)^{\frac{p}{p+1}} \approx  \Big(\int_{\R^N}\tilde{U}_i^p
\tilde{U}_j\Big)^{\tilde{p}}\\
\ \lesssim&\  \Big(\hat  \varepsilon\| \nabla \rho\|_{L^{2}} +\| \nabla \rho\|_{L^{2}}^{\min(\tilde{p}, 2)}
+\Theta(u)  \Big)^{\tilde{p}} \lesssim
   \hat  \varepsilon\| \nabla \rho\|_{L^{2}}
   +o(1)\| \nabla \rho\|_{L^{2}} +\Theta(u),\endaligned
  $$
where the last inequality holds if $\tilde{p}\geq 1$ and $\tilde{p}>1$, namely $\mu<4$.  Therefore, if $3\leq N\leq 5$ and $\frac{N+2}{2}<\mu<\min(N,4)$, then we have that
\begin{equation}\label{eq:4.12}
\aligned \   \mathcal{H} _{19}&\lesssim \hat{\varepsilon}\| \nabla \rho\|_{L^{2}}^{2} +o(1)\| \nabla \rho\|_{L^{2}}^{2}
    +\Theta(u)\| \nabla \rho\|_{L^{2}}.
\endaligned
\end{equation}

\vskip0.3in

$\bullet$ Next we consider the  term $\mathcal{H} _{17}$, which is much thorny. Firstly,
  \begin{equation}\label{eq:4.13}
  \aligned
 \   \mathcal{H}_{17}=&\  \int_{\R^N}\Big(I_{\mu}\ast\big(|u|^{\tilde{p}}u - \sum_{i=1}^\nu\alpha_i|\alpha_i|^{\tilde{p}}
  \tilde{U}_i^{\tilde{p}+1}
  \big)\Big)\Big( \sum_{i=1}^\nu\alpha_i|\alpha_i|^{\tilde{p}-1}
  \tilde{U}_i^{\tilde{p}}   \Big)\rho\,dx  \\
  =&\ \sum_{i=1}^\nu  \alpha_i^{\tilde{p}}\int_{\R^N}\Big(I_{\mu}\ast\big(|u|^{\tilde{p}}u - \sum_{i=1}^\nu\alpha_i|\alpha_i|^{\tilde{p}}
  \tilde{U}_i^{\tilde{p}+1}-(\tilde{p}+1)\sigma^{\tilde{p}}\rho
  \big)\Big)
  \tilde{U}_i^{\tilde{p}}   \rho\,dx \\
  \ &\ +\sum_{i=1}^\nu  \alpha_i^{\tilde{p}} (\tilde{p}+1)\int_{\R^N} \big(I_{\mu} \ast \sigma^{\tilde{p}}\rho \big) \tilde{U}_{i}^{\tilde{p}}\rho \,dx
  \eqqcolon \sum_{i=1}^\nu  \alpha_i^{\tilde{p}} \mathcal{H} _{17,1}^{(i)}+ \sum_{i=1}^\nu  \alpha_i^{\tilde{p}} (\tilde{p}+1)\mathcal{H}_{17,2}^{(i)}.
  \endaligned
  \end{equation}
On the one hand, using  \eqref{eq:4.5},
  $$\aligned\mathcal{H}_{17,1}^{(i)} \lesssim &  \  \int_{\R^N}
  \Big(I_{\mu}\ast \big( \{\sigma^{\tilde{p}-1}|\rho|^2\}_{\tilde{p}+1>2} + |\rho|^{\tilde{p}+1} +
  \sum_{k\not= l}\tilde{U}_k^{\tilde{p}}\tilde{U}_l\big) \Big)\tilde{U}_i^{\tilde{p}}|\rho| \,dx\\
  \ = &\    \Big\{\int_{\R^N}(I_{\mu}\ast \sigma^{\tilde{p}-1}|\rho|^2)
  \tilde{U}_i^{\tilde{p}}|\rho| \,dx\Big\}_{\tilde{p}+1>2} + \int_{\R^N}(I_{\mu}\ast |\rho|^{\tilde{p}+1} )\tilde{U}_i^{\tilde{p}}|\rho| \,dx \\
  & \quad  + \sum_{k\not= l}
   \int_{\R^N}\big(I_{\mu}\ast\tilde{U}_k^{\tilde{p}}\tilde{U}_l\big)
   \tilde{U}_i^{\tilde{p}}|\rho| \,dx.\endaligned$$
  According to the HLS inequality \eqref{eq:1.9} with $ s=t=\frac{2N}{2N-\mu}=\frac{p+1}{\tilde{p}+1}$,
 H\"older and Sobolev inequalities, we get
 $$
  \aligned \Big\{\int_{\R^N}(I_{\mu}\ast \sigma^{\tilde{p}-1}|\rho|^2 )\tilde{U}_i^{\tilde{p}}|\rho| \,dx\Big\}_{\tilde{p}+1>2}
  &\lesssim \Big\{\| \sigma^{\tilde{p}-1}|\rho|^2\|_{L^{\frac{p+1}{\tilde{p}+1}}}
 \|\tilde{U}_i^{\tilde{p}}\rho\|_{L^{\frac{p+1}{\tilde{p}+1}}}\Big\}_{\mu<4}\\
 &\lesssim  \big\{\|\sigma\|_{L^{p+1}} ^{\tilde{p}-1}\|\rho\|_{L^{p+1}}^{2}
 \|\tilde{U}_i\|_{L^{p+1}}^{\tilde{p}}\|\rho\|_{L^{p+1}}\big\}_{\mu<4} \lesssim \big\{\|\nabla\rho\|_{L^2}^{3}\big\}_{\mu<4},\endaligned
 $$
 $$\aligned \int_{\R^N}(I_{\mu}\ast |\rho|^{\tilde{p}+1} )\tilde{U}_i^{\tilde{p}}|\rho| \,dx
 \lesssim \||\rho|^{\tilde{p}+1}\|_{L^{\frac{p+1}{\tilde{p}+1}}}
 \|\tilde{U}_i^{\tilde{p}}\rho \|_{L^{\frac{p+1}{\tilde{p}+1}}}\lesssim  \|\rho\|_{L^{p+1}}^{\tilde{p}+1}
 \|\tilde{U}_i\|_{L^{p+1}}^{\tilde{p}}\|\rho\|_{L^{p+1}}
  \lesssim \|\nabla\rho\|_{L^2}^{\tilde{p}+2}.\endaligned$$
 Similar to  the proof of $\mathcal{H}_{16}$, for $k\neq l$, by using \eqref{eq:4.10} and Lemma \ref{prop:interaction_and_coef}, we obtain that
 $$\aligned
 \int_{\R^N}\Big(I_{\mu}\ast \tilde{U}_k^{\tilde{p}}\tilde{U}_l\Big)
 \tilde{U}_i^{\tilde{p}}|\rho|\,dx
 &\lesssim
 \|\tilde{U}_k^{
 \tilde{p}}\tilde{U}_l\|_{L^{\frac{p+1}{\tilde{p}+1}}}
 \|\tilde{U}_i\|_{L^{p+1}}^{\tilde{p}} \|\rho\|_{L^{p+1}}\\
 &\lesssim\Big(\int_{\R^N} \tilde{U}_{k}^{p} \tilde{U}_{l}\,dx\Big) \|\nabla\rho\|_{L^2} \\
  &\lesssim   \left(\hat\varepsilon\| \nabla \rho\|_{L^{2}}+\| \nabla\rho\|_{L^{2}}^{\min(\tilde{p},2)}
 +\Theta(u)\right)\|\nabla\rho\|_{L^{2}}.\endaligned$$
 Thus,  if $\tilde{p}>1(\iff \mu <4)$, then
 \begin{equation}\label{eq:4.14}
     \aligned
     \mathcal{H}_{17,1}^{(i)}\lesssim \hat\varepsilon\|\nabla\rho\|_{L^2}^{2}+ o(1)\|\nabla\rho\|_{L^2}^{2}+\Theta(u)\|\nabla\rho\|_{L^2}.
     \endaligned
 \end{equation}
On the other hand, for f{}ixed $1\leq i\leq \nu$,  recalling  $$\sigma=\sum\limits_{k=1}^\nu \alpha_k\tilde{U}_k=\alpha_i\tilde{U}_i+\sum\limits_{j\neq i}\alpha_j\tilde{U}_j,$$ by using the elementary inequality
 $$|(a+b)^{\tilde{p}}-a^{\tilde{p}}|\leq \varepsilon_{1} |a|^{\tilde{p}}+C(\varepsilon_{1})|b|^{\tilde{p}}, \;\; \hbox{for} \tilde{p}>1  \hbox{ and } \varepsilon_{1}>0 $$ and the Young inequality $$\tilde{U}_j^{\tilde{p}} |\rho|\lesssim \tilde{U}_j^{\tilde{p}+1}+|\rho|^{\tilde{p}+1},$$
the HLS inequality, H\"older and Sobolev inequalities and by the same process as  the estimate of $\mathcal{H}_{19}$, we get that,  if $3\leq N\leq 5$ and $\frac{N+2}{2}<\mu<\min(N,4)$, then
  \begin{equation}\label{eq:4.15}
  \aligned\mathcal{H}_{17,2}^{(i)}&= \int_{\R^N} \big(I_{\mu} \ast \sigma^{\tilde{p}}\rho \big) \tilde{U}_{i}^{\tilde{p}}\rho \,dx=\int_{\R^N} \Big(I_{\mu} \ast \big(\sigma^{\tilde{p}}\rho -\tilde{U}_{i}^{\tilde{p}}\rho\big)\Big) \tilde{U}_{i}^{\tilde{p}}\rho \,dx+\int_{\R^N} \big(I_{\mu} \ast \tilde{U}_{i}^{\tilde{p}}\rho \big) \tilde{U}_{i}^{\tilde{p}}\rho \,dx \\
  &\lesssim \varepsilon_{1} \int_{\R^N} \big(I_{\mu} \ast  \tilde{U}_{i}^{\tilde{p}}|\rho|\big)\tilde{U}_{i}^{\tilde{p}}|\rho| + C(\varepsilon_{1}) \sum_{j\neq i} \int_{\R^N} \big(I_{\mu} \ast  \tilde{U}_{j}^{\tilde{p}}|\rho|\big)\tilde{U}_{i}^{\tilde{p}}|\rho|+\int_{\R^N} \big(I_{\mu} \ast \tilde{U}_{i}^{\tilde{p}}\rho \big) \tilde{U}_{i}^{\tilde{p}}\rho \,dx  \\
  &\lesssim \varepsilon_{1} \int_{\R^N} \big(I_{\mu} \ast  \tilde{U}_{i}^{\tilde{p}}|\rho|\big)\tilde{U}_{i}^{\tilde{p}}|\rho| + C(\varepsilon_{1}) \sum_{j\neq i} \int_{\R^N} \big(I_{\mu} \ast  \tilde{U}_{j}^{\tilde{p}+1}\big)\tilde{U}_{i}^{\tilde{p}}|\rho|\\
  &\ \ \ \ + C(\varepsilon_{1}) \sum_{j\neq i}\int_{\R^N}\big(I_{\mu} \ast |\rho|^{\tilde{p}+1}\big)\tilde{U}_{i}^{\tilde{p}}|\rho|+\int_{\R^N} \big(I_{\mu} \ast \tilde{U}_{i}^{\tilde{p}}\rho \big) \tilde{U}_{i}^{\tilde{p}}\rho \,dx \\
  &\lesssim \varepsilon_{1}\|\nabla \rho\|_{L^{2}}^{2} +C(\varepsilon_{1})\sum_{j\neq i}\int_{\R^N} \tilde{U}_j^{p-\tilde{p}}\tilde{U}_i^{\tilde{p}}|\rho|\,dx+C(\varepsilon_{1})\|\nabla\rho\|_{L^2}^{\tilde{p}+2}  + \int_{\R^N} \big(I_{\mu} \ast \tilde{U}_{i}^{\tilde{p}}\rho \big) \tilde{U}_{i}^{\tilde{p}}\rho \,dx\\
  & \lesssim \varepsilon_{1}\|\nabla \rho\|_{L^{2}}^{2} +C(\varepsilon_{1})
  \Big(\hat\varepsilon\| \nabla \rho\|_{L^{2}}+o(1)\|\nabla\rho\|_{L^{2}}
  +\Theta(u)\Big)\|\nabla\rho\|_{L^{2}}\\
  &\ \ \ \ + C(\varepsilon_{1})\|\nabla\rho\|_{L^2}^{\tilde{p}+2}  + \int_{\R^N} \big(I_{\mu} \ast \tilde{U}_{i}^{\tilde{p}}\rho \big) \tilde{U}_{i}^{\tilde{p}}\rho \,dx \\
  & \lesssim \varepsilon_{1}\|\nabla \rho\|_{L^{2}}^{2} +\hat{\varepsilon}C(\varepsilon_{1})
 \| \nabla \rho\|_{L^{2}}^{2}+o(1)\|\nabla\rho\|_{L^{2}}^{2}
 +\Theta(u)\|\nabla\rho\|_{L^{2}}+\int_{\R^N} \big(I_{\mu} \ast \tilde{U}_{i}^{\tilde{p}}\rho \big) \tilde{U}_{i}^{\tilde{p}}\rho \,dx.
 \endaligned
\end{equation}
Next we show  that
\begin{equation}\label{eq:4.16}
\int_{\R^N} \big(I_{\mu} \ast \tilde{U}_{i}^{\tilde{p}}\rho \big) \tilde{U}_{i}^{\tilde{p}}\rho \,dx\lesssim  o(\|\nabla \rho\|^{2}_{L^{2}}) \text{\quad as \quad} \|\nabla \rho\|_{L^{2}}\to 0, \quad \forall 1\leq i\leq \nu.
\end{equation}
To this prupose, we let $\tilde{\rho}=T_{z_{i},\lambda_{i}}^{-1}\rho=T_{-\lambda_{i} z_{i},\lambda_{i}^{-1}}\rho$, then it follows from \eqref{eq:2.9} that $$\|\tilde{\rho}\|_{L^{p+1}}=\|\rho\|_{L^{p+1}}\ \ \hbox{and}\ \ \|\nabla \tilde{\rho}\|_{L^{2}}=\|\nabla \rho\|_{L^{2}}.$$
Recall $$\tilde{U}_{i}=\tilde{U}[z_{i},\lambda_{i}]=T_{z_{i},\lambda_{i}}(\tilde{U}[0,1])=T_{z_{i},\lambda_{i}} \tilde{U},$$
where we denote $\tilde{U}=\tilde{U}[0,1]$, then by the conformal invariance \eqref{eq:2.10}, we see that
 $$
\int_{\R^N} \big(I_{\mu} \ast \tilde{U}_{i}^{\tilde{p}}\rho \big) \tilde{U}_{i}^{\tilde{p}}\rho \,dx= \int_{\R^N} \big(I_{\mu} \ast (T_{z_{i},\lambda_{i}}\tilde{U})^{\tilde{p}} T_{z_{i},\lambda_{i}}\tilde{\rho} \big) (T_{z_{i},\lambda_{i}}\tilde{U})^{\tilde{p}} T_{z_{i},\lambda_{i}}\tilde{\rho} \,dx  =\int_{\R^N} \big(I_{\mu} \ast \tilde{U}^{\tilde{p}}\tilde{\rho} \big) \tilde{U}^{\tilde{p}}\tilde{\rho} \,dx,
 $$
hence the orthogonality condition \eqref{eq:4.3} becomes
 $$
    0=\int_{\R^N} \big(I_{\mu} \ast \tilde{U}_{i}^{\tilde{p}}\rho \big)\tilde{U}_{i}^{\tilde{p}+1}\,dx
    =\int_{\R^N} \big(I_{\mu} \ast (T_{z_{i},\lambda_{i}}\tilde{U})^{\tilde{p}} T_{z_{i},\lambda_{i}}
    \tilde{\rho} \big) (T_{z_{i},\lambda_{i}}\tilde{U})^{\tilde{p}+1}  \,dx
    =\int_{\R^N} \big(I_{\mu} \ast \tilde{U}^{\tilde{p}}\tilde{\rho} \big)\tilde{U}^{\tilde{p}+1} \,dx,
 $$
i.e., $\tilde{\rho} $ is  orthogonal to  $\tilde{U}$ in $\mathcal{D}^{1,2}(\R^{N})$. For the simplicity of notation, we still use notation $\rho$ instead of $\tilde{\rho}$, namely for the proof of \eqref{eq:4.16}, it is suf{}f{}ice to prove that
\begin{equation}\label{eq:4.17}
    \int_{\R^N} \big(I_{\mu} \ast \tilde{U}^{\tilde{p}}\rho \big) \tilde{U}^{\tilde{p}}\rho \,dx\lesssim o(\|\nabla \rho\|^2_{L^{2}})\text{\quad as \quad}
    \|\nabla \rho\|_{L^{2}}\to 0,
\end{equation}
under the orthogonality condition
\begin{equation}\label{eq:4.18}
    \int_{\R^N} \big(I_{\mu} \ast \tilde{U}^{\tilde{p}}\rho \big)\tilde{U}^{\tilde{p}+1} \,dx=0.
\end{equation}
In fact, we  take $R=1/\|\nabla \rho\|_{L^{2}}^{\tau_{1}}$, where $\tau_{1}>0$ will be determined later, then $R\gg 1$ as $\|\nabla \rho\|_{L^{2}}\ll 1$ and
\begin{equation}\label{eq:4.19}
\aligned
\int_{\R^N} \big(I_{\mu} \ast \tilde{U}^{\tilde{p}}\rho \big) \tilde{U}^{\tilde{p}}\rho \,dx
 &=\ \int_{\R^N} \big(I_{\mu} \ast \tilde{U}^{\tilde{p}}\rho \big) \tilde{U}^{\tilde{p}}\rho
\mathds{1}_{B(0,R)}\,dx +\int_{\R^N} \big(I_{\mu} \ast \tilde{U}^{\tilde{p}}\rho \big)
\tilde{U}^{\tilde{p}}\rho \mathds{1}_{B(0,R)^{c}}\,dx \\
 &\eqqcolon  \mathcal{I}_{1}+\mathcal{I}_{2}.\endaligned
\end{equation}
For the second term $\mathcal{I}_{2}$, we have
 $$
    |\mathcal{I}_{2}|=
    \Big|\int_{\R^N} 
     \big(I_{\mu} \ast
      \tilde{U}^{\tilde{p}}
      \rho \big)
       \tilde{U}^{\tilde{p}}
       \rho \mathds{1}_{
        B(0,R)^{c}}\,dx
        \Big|\lesssim 
        \|\tilde{U}
        \|^{\tilde{p}}_{
          L^{p+1}}\|\nabla 
          \rho\|_{L^{2}} 
          \|\tilde{U}
          \mathds{1}_{
            B(0,R)^{c}}\|^{\tilde{p}}_{L^{p+1}}\|\nabla \rho\|_{L^{2}},
            $$
where
\begin{equation}\label{eq:4.20}
     \aligned\|\tilde{U}\mathds{1}_{B(0,R)^{c}}\|^{\tilde{p}}_{L^{p+1}}
     &=\Big(\int_{B(0,R)^{c}}\tilde{U}^{p+1} \,dx   \Big)^{
        \frac{\tilde{p}}{p+1}}\\
     &\approx \bigg(\int_{|x|>R}
      \bigg(\frac{1}{1+|x|^{2}}\bigg)^{\frac{N-2}{2}\cdot \frac{2N}{N-2 }}
      \,dx\bigg)^{\frac{N-\mu+2}{2N}}\\
      &\approx \Big(\int_{R}^{\infty} \frac{1}{r^{2N}} r^{N-1}
      \,dr\Big)^{\frac{N-\mu+2}{2N}} \\
      &\approx R^{-\frac{N-\mu+2}{2}}.\endaligned
     \end{equation}
Thus,  if $\tau_{1}>0$, there holds
\begin{equation}\label{eq:4.21}
    |\mathcal{I}_{2}|\lesssim \|\nabla \rho\|_{L^{2}}^{\tau_{1}\frac{N-\mu+2}{2}} \|\nabla \rho\|_{L^{2}}^{2}\approx \|\nabla \rho\|_{L^{2}}^{\tau_{1}\frac{N-\mu+2}{2}+2}\lesssim o(\|\nabla \rho\|_{L^{2}}^{2}) \text{\quad as \quad} \|\nabla \rho\|_{L^{2}}\to 0.
\end{equation}
For the f{}irst term $\mathcal{I}_{1}$, by the orthogonality condition \eqref{eq:4.18} and $\rho=\rho^{+}-\rho^{-}$, we have
 $$
    \int_{\R^N} \big(I_{\mu} \ast \tilde{U}^{\tilde{p}}\rho^{+} \big)\tilde{U}^{\tilde{p}+1} \,dx=\int_{\R^N} \big(I_{\mu} \ast \tilde{U}^{\tilde{p}}\rho^{-} \big)\tilde{U}^{\tilde{p}+1} \,dx,
 $$
then for $\Lambda>0$, which depends on $\|\nabla \rho\|_{L^{2}}$ and will be determined later, we have
\begin{equation}\label{eq:4.22}
\aligned
\mathcal{I}_{1}&=\int_{\R^N}  \big(I_{\mu} \ast \tilde{U}^{\tilde{p}}\rho \big) \tilde{U}^{\tilde{p}}\rho \mathds{1}_{B(0,R)}\,dx=\int_{\R^N}  \big(I_{\mu} \ast \tilde{U}^{\tilde{p}}(\rho^{+}-\rho^{-}) \big) \tilde{U}^{\tilde{p}}(\rho^{+}-\rho^{-}) \mathds{1}_{B(0,R)}\,dx\\
    &\leq \int_{\R^N}  \big(I_{\mu} \ast \tilde{U}^{\tilde{p}}\rho^{+} \big) \tilde{U}^{\tilde{p}}\rho^{+} \mathds{1}_{B(0,R)}\,dx + \int_{\R^N}  \big(I_{\mu} \ast \tilde{U}^{\tilde{p}}\rho^{-} \big) \tilde{U}^{\tilde{p}}\rho^{-} \mathds{1}_{B(0,R)}\,dx   \\
    &=\int_{\R^N}  \big(I_{\mu} \ast \tilde{U}^{\tilde{p}}\rho^{+} \big) \tilde{U}^{\tilde{p}}\rho^{+} \mathds{1}_{B(0,R)}\,dx -\Lambda \int_{\R^N} \big(I_{\mu} \ast \tilde{U}^{\tilde{p}}\rho^{+} \big)\tilde{U}^{\tilde{p}+1} \,dx \\
    & \ \ \ \ + \int_{\R^N}  \big(I_{\mu} \ast \tilde{U}^{\tilde{p}}\rho^{-} \big) \tilde{U}^{\tilde{p}}\rho^{-} \mathds{1}_{B(0,R)}\,dx +\Lambda \int_{\R^N} \big(I_{\mu} \ast \tilde{U}^{\tilde{p}}\rho^{-} \big)\tilde{U}^{\tilde{p}+1} \,dx\\
    &=\int_{\R^N}  \big(I_{\mu} \ast \tilde{U}^{\tilde{p}}\rho^{+} \big) \tilde{U}^{\tilde{p}}(\rho^{+}-\Lambda \tilde{U}) \mathds{1}_{B(0,R)}\,dx -\Lambda \int_{\R^N} \big(I_{\mu} \ast \tilde{U}^{\tilde{p}}\rho^{+} \big)\tilde{U}^{\tilde{p}+1} \mathds{1}_{B(0,R)^{c}}\,dx \\
    &\ \ \ \ + \int_{\R^N}  \big(I_{\mu} \ast \tilde{U}^{\tilde{p}}\rho^{-} \big) \tilde{U}^{\tilde{p}}(\rho^{-}+\Lambda \tilde{U}) \mathds{1}_{B(0,R)}\,dx +\Lambda \int_{\R^N} \big(I_{\mu} \ast \tilde{U}^{\tilde{p}}\rho^{-} \big)\tilde{U}^{\tilde{p}+1} \mathds{1}_{B(0,R)^{c}}\,dx\\
    &\leq\int_{\R^N}  \big(I_{\mu} \ast \tilde{U}^{\tilde{p}}\rho^{+} \big) \tilde{U}^{\tilde{p}}(\rho^{+}-\Lambda \tilde{U}) \mathds{1}_{B(0,R)}\,dx  \\
    &\ \ \ \ + \int_{\R^N}  \big(I_{\mu} \ast \tilde{U}^{\tilde{p}}\rho^{-} \big) \tilde{U}^{\tilde{p}}(\rho^{-}+\Lambda \tilde{U}) \mathds{1}_{B(0,R)}\,dx +\Lambda \int_{\R^N} \big(I_{\mu} \ast \tilde{U}^{\tilde{p}}\rho^{-} \big)\tilde{U}^{\tilde{p}+1} \mathds{1}_{B(0,R)^{c}}\,dx\\
    &\eqqcolon \mathcal{I}_{1,1}+\mathcal{I}_{1,2}+\mathcal{I}_{1,3}.
\endaligned
\end{equation}
\vskip0.06in
 \noindent$\bullet$ For the third term $\mathcal{I}_{1,3}$, we have
\begin{equation}\label{eq:4.23}
    \mathcal{I}_{1,3}\lesssim \Lambda \|\tilde{U}\|_{L^{p+1}}^{\tilde{p}}\|\nabla\rho\|_{L^{2}}\|\tilde{U}\mathds{1}_{B(0,R)^{c}} \|_{L^{p+1}}^{\tilde{p}+1}\lesssim \Lambda\|\nabla\rho\|_{L^{2}} R^{-(N-\frac{\mu}{2})}\approx \Lambda\|\nabla\rho\|_{L^{2}}^{1+\tau_{1}(N-\frac{\mu}{2})},
\end{equation}
where we used the calculations in \eqref{eq:4.20}.
\vskip0.06in
      \noindent$\bullet$ For the f{}irst term $\mathcal{I}_{1,1}$, since for $x\in B(0,R)$,
$$\tilde{U}(x)\gtrsim 1/(1+R^{2})^{\frac{N-2}{2}} \approx R^{-(N-2)} \eqqcolon \Psi,$$ then for any $\tau_{2}>0$, which will be determined later, we have $$\tilde{U}^{\frac{\tau_{2}}{1+\tau_{2}}}\gtrsim \Psi^{\frac{\tau_{2}}{1+\tau_{2}}},$$ so $$\tilde{U}\gtrsim \Psi^{\frac{\tau_{2}}{1+\tau_{2}}}\tilde{U}^{\frac{1}{1+\tau_{2}}},$$ then
\begin{equation}\label{eq:4.24}
 \aligned   \mathcal{I}_{1,1}&=\int_{\R^N}  \big(I_{\mu} \ast \tilde{U}^{\tilde{p}}\rho^{+} \big)
    \tilde{U}^{\tilde{p}}(\rho^{+}-\Lambda \tilde{U}) \mathds{1}_{B(0,R)}\,dx \\
    &\leq \int_{\R^N}  \big(I_{\mu} \ast \tilde{U}^{\tilde{p}}\rho^{+} \big)
    \tilde{U}^{\tilde{p}}(\rho^{+}-\Lambda\Psi^{\frac{\tau_{2}}{1+\tau_{2}}}
    \tilde{U}^{\frac{1}{1+\tau_{2}}} ) \mathds{1}_{B(0,R)}\,dx. \endaligned
\end{equation}
Let $\varepsilon_2\in(0,1)$ and def{}ine the set $M_{1}\coloneqq\{\rho^{+}-\Lambda\Psi^{\frac{\tau_{2}}{1+\tau_{2}}}\tilde{U}^{\frac{1}{1+\tau_{2}}}\leq \varepsilon_{2}\rho^{+}\}$, then $M_{1}\neq \emptyset $ and $$M_{1}^{c}=\{\rho^{+}-\Lambda\Psi^{\frac{\tau_{2}}{1+\tau_{2}}}\tilde{U}^{\frac{1}{1+\tau_{2}}}> \varepsilon_{2}\rho^{+}\}=\{\tilde{U}<\Big(\frac{1-\varepsilon_{2}}{\Lambda}\Big)^{1+\tau_{2}}\Psi^{-\tau_{2}}(\rho^{+})^{1+\tau_{2}}\}.$$
Further, by \eqref{eq:4.24}, we have
\begin{equation}\label{eq:4.25}
\aligned
    \mathcal{I}_{1,1}&\leq \int_{\R^N}  \big(I_{\mu} \ast \tilde{U}^{\tilde{p}}\rho^{+} \big) \tilde{U}^{\tilde{p}}(\rho^{+}-\Lambda\Psi^{\frac{\tau_{2}}{1+\tau_{2}}}\tilde{U}^{\frac{1}{1+\tau_{2}}} ) \mathds{1}_{B(0,R)\cap M_{1}}\,dx\\
    &\ \ \ \ +\int_{\R^N}  \big(I_{\mu} \ast \tilde{U}^{\tilde{p}}\rho^{+} \big) \tilde{U}^{\tilde{p}}(\rho^{+}-\Lambda\Psi^{\frac{\tau_{2}}{1+\tau_{2}}}\tilde{U}^{\frac{1}{1+\tau_{2}}} ) \mathds{1}_{B(0,R) \cap M_{1}^{c}}\,dx  \\
    &\leq \varepsilon_{2}\int_{\R^N}  \big(I_{\mu} \ast \tilde{U}^{\tilde{p}}\rho^{+} \big) \tilde{U}^{\tilde{p}}\rho^{+} \mathds{1}_{B(0,R)\cap M_{1}}\,dx\\
    &\ \ \ \ +\int_{\R^N}  \big(I_{\mu} \ast \tilde{U}^{\tilde{p}}\rho^{+} \big)    \Big(\frac{1-\varepsilon_{2}}{\Lambda}\Big)^{\tilde{p}(1+\tau_{2})}\Psi^{-\tilde{p}\tau_{2}}(\rho^{+})^{\tilde{p}(1+\tau_{2})} \rho^{+} \mathds{1}_{B(0,R) \cap M_{1}^{c}}\,dx\\
    &\leq \varepsilon_{2}\int_{\R^N}  \big(I_{\mu} \ast \tilde{U}^{\tilde{p}}|\rho| \big)
    \tilde{U}^{\tilde{p}}|\rho|\,dx+\Big(\frac{1-\varepsilon_{2}}{\Lambda}\Big)^{\tilde{p}(1+\tau_{2})}\Psi^{-\tilde{p}\tau_{2}}\int_{\R^N}  \big(I_{\mu} \ast \tilde{U}^{\tilde{p}}|\rho| \big)    |\rho|^{\tilde{p}(1+\tau_{2})+1} \,dx\\
    &\eqqcolon \mathcal{I}_{1,1,1}+\mathcal{I}_{1,1,2}.
    \endaligned
\end{equation}
Obviously,
\begin{equation}\label{eq:4.26}
    \mathcal{I}_{1,1,1} \lesssim \varepsilon_{2} \|\nabla \rho\|_{L^{2}}^{2}.
\end{equation}
Next we are going to bound $\mathcal{I}_{1,1,2}$ with some powers of $\|\nabla \rho\|_{L^{2}}$. We remark that the exponent of the term   outside the convolution is $\tilde{p}(1+\tau_{2})+1\neq \tilde{p}+1$,  {hence the same method of using HLS inequality to control the terms $\mathcal{H}_{1}\sim \mathcal{H}_{16}$ directly is not feasible here} (see Lemma \ref{double-integral-estimate} and  Remark \ref{rmk3-double-integral-estimate}). In order to control $\mathcal{I}_{1,1,2}$, we need to modify the exponent of the  term inside the convolution. In fact, since $\tilde{U}\lesssim 1$ over $\mathbb{R}^{N}$, so for any $\zeta>0$ and $\tilde{U}\lesssim \tilde{U}^{1-\zeta}$,
by combining  the HLS inequality \eqref{eq:1.9} with $1/s+1/t=(2N-\mu)/N=2(\tilde{p}+1)/(p+1)$, we have that
\begin{equation}\label{eq:4.27}
\aligned
&\int_{\R^N}  \big(I_{\mu} \ast \tilde{U}^{\tilde{p}}|\rho| \big)    |\rho|^{\tilde{p}(1+\tau_{2})+1} \,dx \lesssim  \int_{\R^N}  \big(I_{\mu} \ast \tilde{U}^{\tilde{p}(1-\zeta)}|\rho| \big)    |\rho|^{\tilde{p}(1+\tau_{2})+1} \,dx\\
&\lesssim  \|\tilde{U}^{\tilde{p}(1-\zeta)}\rho\|_{L^{s}}\||\rho|^{\tilde{p}(1+\tau_{2})+1}\|_{L^{t}} =\|\tilde{U}^{\tilde{p}(1-\zeta)}\rho\|_{L^{s}} \|\rho\|^{\tilde{p}(1+\tau_{2})+1}_{L^{(\tilde{p}(1+\tau_{2})+1)t}}.
    \endaligned
\end{equation}
Let $(\tilde{p}(1+\tau_{2})+1)t=p+1$, then $s=\frac{p+1}{\tilde{p}(1-\tau_{2})+1}$.  By the H\"older inequality \big($((p+1)/s)^{\prime}=(\tilde{p}(1-\tau_{2})+1)^{\prime}=(\tilde{p}(1-\tau_{2})+1)/\tilde{p}(1-\tau_{2}) $ \big) and the Sobolev inequality, we obtain that
\begin{equation}\label{eq:4.28}
    \aligned
    \|\tilde{U}^{\tilde{p}(1-\zeta)}\rho\|_{L^{s}}&=  \Big(\int_{\R^N} \tilde{U}^{\tilde{p}(1-\zeta)s}|\rho|^{s}  \,dx\Big)^{\frac{1}{s}}\\
    &\lesssim \bigg( \Big(\int_{\R^N}\tilde{U}^{\tilde{p}(1-\zeta)s(\frac{p+1}{s})^{\prime}} \,dx\Big)^{1-\frac{s}{p+1}} \Big(\int_{\R^N}|\rho|^{s\frac{p+1}{s}} \,dx\Big)^{\frac{s}{p+1}} \bigg)^{\frac{1}{s}}\\
    &\lesssim \Big(\int_{\R^N}\tilde{U}^{(1-\zeta) \frac{p+1}{1-\tau_{2}}} \,dx\Big)^{\frac{\tilde{p}(1-\tau_{2})}{p+1}} \|\nabla \rho\|_{L^{2}}.
    \endaligned
\end{equation}
Then let  $\zeta=\tau_{2}$, it follows from \eqref{eq:4.27}, \eqref{eq:4.28} and the Sobolev inequality that
\begin{equation}\label{eq:4.29}
    \int_{\R^N}  \big(I_{\mu} \ast \tilde{U}^{\tilde{p}}|\rho| \big)    |\rho|^{\tilde{p}(1+\tau_{2})+1} \,dx \lesssim
    \|\nabla \rho\|_{L^{2}} \|\rho\|^{\tilde{p}(1+\tau_{2})+1}_{L^{p+1}}\lesssim \|\nabla \rho\|_{L^{2}}^{\tilde{p}(1+\tau_{2})+2}.
\end{equation}
In addition, if we take $\Lambda\coloneqq\|\nabla \rho\|_{L^{2}}^{1-\tau_{3}}$, where $\tau_{3}\in(0,1/2)$ will be determined later, and recall $\Psi =R^{-(N-2)}=\|\nabla \rho\|_{L^{2}}^{\tau_{1}(N-2)}$,  then by \eqref{eq:4.29} we have  that
\begin{equation}\label{eq:4.30}
\aligned
   \mathcal{I}_{1,1,2}&= \Big(\frac{1-\varepsilon_{2}}{\Lambda}\Big)^{\tilde{p}(1+\tau_{2})}\Psi^{-\tilde{p}\tau_{2}}\int_{\R^N}  \big(I_{\mu} \ast \tilde{U}^{\tilde{p}}|\rho| \big)    |\rho|^{\tilde{p}(1+\tau_{2})+1} \,dx\\
   &\lesssim\bigg(\frac{1-\varepsilon_{2}}{ \|\nabla \rho\|_{L^{2}}^{1-\tau_{3}} }\bigg)^{\tilde{p}(1+\tau_{2})} \|\nabla \rho\|_{L^{2}}^{-\tilde{p}\tau_{2}\tau_{1}(N-2)}\|\nabla \rho\|_{L^{2}}^{\tilde{p}(1+\tau_{2})+2}\\
   &=(1-\varepsilon_{2})^{\tilde{p}(1+\tau_{2})}\|\nabla \rho\|_{L^{2}}^{\tilde{p}(1+\tau_{2})+2-\tilde{p}(1+\tau_{2})(1-\tau_{3})-\tilde{p}\tau_{2}\tau_{1}(N-2) }\\
   &\lesssim\|\nabla \rho\|_{L^{2}}^{2+\tilde{p}[(1+\tau_{2})\tau_{3}-\tau_{2}\tau_{1}(N-2)]}.
\endaligned
\end{equation}
Then
\begin{equation}\label{eq:4.31}
\aligned
   \mathcal{I}_{1,1,2}\lesssim o(\|\nabla \rho\|_{L^{2}}^{2})
   \text{\quad as \quad} \|\nabla \rho\|_{L^{2}} \to 0&\iff 2+\tilde{p}\big[(1+\tau_{2})\tau_{3}-\tau_{2}\tau_{1}(N-2)\big]>2 \\
   &\iff \tau_{3}>\tau_{1} (N-2)\frac{\tau_{2}}{1+\tau_{2}}.
\endaligned
\end{equation}
By taking  $\varepsilon_{2}\coloneqq\|\nabla \rho\|_{L^{2}}\ll 1$, it follows from \eqref{eq:4.25},
\eqref{eq:4.26} and \eqref{eq:4.30} that
\begin{equation}\label{eq:4.32}
    \mathcal{I}_{1,1}\lesssim  \varepsilon_{2} \|\nabla \rho\|_{L^{2}}^{2} + \|\nabla \rho\|_{L^{2}}^{2+\tilde{p}[(1+\tau_{2})\tau_{3}-\tau_{2}\tau_{1}(N-2)]}\lesssim o(\|\nabla \rho\|_{L^{2}}^{2}) \\
\end{equation}
as long as  $\tau_{1}, \tau_{2}, \tau_{3} $ satisfy   $\eqref{eq:4.31}$. By the choice of $\Lambda$, going back to \eqref{eq:4.23}, we have
\begin{equation}
    \mathcal{I}_{1,3}\lesssim \|\nabla \rho\|_{L^{2}}^{2+\tau_{1}(N-\frac{\mu}{2})
    -\tau_{3}},
\end{equation}
then
\begin{equation}\label{eq:4.33}
\aligned
    \mathcal{I}_{1,3}\lesssim o(\|\nabla \rho\|_{L^{2}}^{2})
    \text{\quad as \quad}\|\nabla \rho\|_{L^{2}}\to 0& \hbox{ whenever  } 2+\tau_{1}(N-\frac{\mu}{2})-\tau_{3}>2.
\endaligned
\end{equation}
\vskip0.16in
      \noindent$\bullet$ Next we deal with the second term $\mathcal{I}_{1,2}$.  Since $\tilde{U}\lesssim 1$, thus for any $\tau_{4}>0$, $\tilde{U}\lesssim \tilde{U}^{\frac{1}{1+\tau_{4}}}$, we may def{}ine the set $M_{2}\coloneqq\{0\leq\rho^{-}< \Lambda\tilde{U}^{\frac{1}{1+\tau_{4}}}\}$ and  then
\begin{equation}\label{eq:4.34}
    \aligned
M_{2}^{c}=\{\rho^{-}\geq \Lambda\tilde{U}^{\frac{1}{1+\tau_{4}}}\},
    \endaligned
\end{equation}
it follows that
\begin{equation}\label{eq:4.35}
\aligned
    \mathcal{I}_{1,2}
    &\lesssim \int_{\R^N}  \big(I_{\mu} \ast \tilde{U}^{\tilde{p}}\rho^{-} \big) \tilde{U}^{\tilde{p}}(\rho^{-}+\Lambda \tilde{U}) \mathds{1}_{B(0,R)\cap M_{2}}\,dx \\
    &\ \ +\int_{\R^N}  \big(I_{\mu} \ast \tilde{U}^{\tilde{p}}\rho^{-} \big) \tilde{U}^{\tilde{p}}(\rho^{-}+\Lambda \tilde{U}^{\frac{1}{1+\tau_{4}}}) \mathds{1}_{B(0,R)\cap M_{2}^{c}}\,dx\\
    &\approx \int_{\R^N}  \big(I_{\mu} \ast \tilde{U}^{\tilde{p}}\rho^{-} \big) \tilde{U}^{\tilde{p}}\rho^{-} \mathds{1}_{B(0,R)\cap M_{2}}\,dx+\Lambda\int_{\R^N}  \big(I_{\mu} \ast \tilde{U}^{\tilde{p}}\rho^{-} \big) \tilde{U}^{\tilde{p}+1} \mathds{1}_{B(0,R)\cap M_{2}}\,dx\\
    &\ \ +\int_{\R^N}  \big(I_{\mu} \ast \tilde{U}^{\tilde{p}}\rho^{-} \big) \tilde{U}^{\tilde{p}}(\rho^{-}+\Lambda \tilde{U}^{\frac{1}{1+\tau_{4}}}) \mathds{1}_{B(0,R)\cap M_{2}^{c}}\,dx\\
    &\eqqcolon \mathcal{I}_{1,2,1}+\mathcal{I}_{1,2,2}+\mathcal{I}_{1,2,3}.
\endaligned
\end{equation}
Further,
\begin{equation}\label{eq:4.36}
\aligned
\mathcal{I}_{1,2,1}&=\int_{\R^N}  \big(I_{\mu} \ast \tilde{U}^{\tilde{p}}\rho^{-} \big) \tilde{U}^{\tilde{p}}\rho^{-} \mathds{1}_{B(0,R)\cap M_{2}}\,dx=\int_{\R^N}  \big(I_{\mu} \ast \tilde{U}^{\tilde{p}}\rho^{-} \big) \tilde{U}^{\tilde{p}}\rho^{-} \mathds{1}_{B(0,R)\cap M_{2}\cap\{\rho^{-}\neq 0\}}\,dx\\
&\lesssim  \int_{\R^N}  \big(I_{\mu} \ast \tilde{U}^{\tilde{p}}|\rho| \big) \tilde{U}^{\tilde{p}}|\rho| \mathds{1}_{B(0,R)\cap M_{2}\cap\{\rho^{-}\neq 0\}}\,dx \\
&\lesssim \|\tilde{U}\|_{L^{p+1}}^{\tilde{p}}\|\nabla \rho\|_{L^{2}}\|\tilde{U}
\mathds{1}_{B(0,R)\cap M_{2}\cap\{\rho^{-}\neq 0\}}\|_{L^{p+1}}^{\tilde{p}}\|\nabla \rho\|_{L^{2}}.
\endaligned
\end{equation}
Since $\tau_{3}<1$ implies that $\Lambda=\|\nabla\rho\|_{L^{2}}^{1-\tau_{3}}\to 0$ as $\|\nabla\rho\|_{L^{2}}\to 0$, then
$$M_{2}\cap\{\rho^{-}\neq0\}=\Big\{0<\rho^{-}< \Lambda
\tilde{U}^{\frac{1}{1+\tau_{4}}}\Big\} \to \emptyset
\text{\quad as \quad}\|\nabla\rho\|_{L^{2}}\to 0.$$ By the absolute continuity of the integral, we get that
\begin{equation}\label{eq:4.37}
\aligned
\|\tilde{U}\mathds{1}_{B(0,R)\cap M_{2}\cap\{\rho^{-}\neq 0\}}\|_{L^{p+1}}=o(1)\text{\quad as \quad} \|\nabla\rho\|_{L^{2}}\to 0.
\endaligned
\end{equation}
It follows from \eqref{eq:4.36} and \eqref{eq:4.37}  that
\begin{equation}\label{eq:4.38}
\aligned
\mathcal{I}_{1,2,1}&\lesssim o(1)\|\nabla\rho\|_{L^{2}}^{2}\approx o(\|\nabla\rho\|_{L^{2}}^{2}).
\endaligned
\end{equation}
Similarly, by the co-area formula $$|M_{2}|=\int_{M_2} \mathds{1}_{M_{2}}\,dx=\int_{-\infty}^{+\infty}\bigg(\int_{\{\rho^{-}=\sigma\}} \frac{1}{|\nabla \rho^{-}|}\,dx \bigg)\,d\sigma=\int_{0}^{\Lambda \tilde{U}^{\frac{1}{1+\tau_{4}}}}\bigg(\int_{\{\rho^{-}=\sigma\}} \frac{1}{|\nabla \rho^{-}|}\,dx \bigg)\,d\sigma=o(\Lambda),$$
then it follows from $\tau_{3}<1/2$ that
\begin{equation}\label{eq:4.39}
\aligned
\mathcal{I}_{1,2,2}&=\Lambda\int_{\R^N}  \big(I_{\mu} \ast \tilde{U}^{\tilde{p}}\rho^{-} \big) \tilde{U}^{\tilde{p}+1} \mathds{1}_{B(0,R)\cap M_{2}}\,dx\\
&\lesssim\Lambda\|\tilde{U}\|_{L^{p+1}}^{\tilde{p}}\|\nabla \rho\|_{L^{2}}\|\tilde{U}
\mathds{1}_{B(0,R)\cap M_{2}}\|_{L^{p+1}}^{\tilde{p}+1}\\
&\lesssim \Lambda \|\nabla\rho\|_{L^{2}}o(\Lambda)\approx o(\|\nabla\rho\|_{L^{2}}^{3-2\tau_{3}})\lesssim o(\|\nabla\rho\|_{L^{2}}^{2}).
\endaligned
\end{equation}
F{}inally, by $\eqref{eq:4.34}$ and the similar process as  \eqref{eq:4.27}, \eqref{eq:4.28} and
\eqref{eq:4.29}, we have that
\begin{equation}\label{eq:4.40}
\aligned
\mathcal{I}_{1,2,3}&=\int_{\R^N}  \big(I_{\mu} \ast \tilde{U}^{\tilde{p}}\rho^{-} \big) \tilde{U}^{\tilde{p}}(\rho^{-}+\Lambda \tilde{U}^{\frac{1}{1+\tau_{4}}}) \mathds{1}_{B(0,R)\cap M_{2}^{c}}\,dx\\
&\lesssim \int_{\R^N}  \big(I_{\mu} \ast \tilde{U}^{\tilde{p}}\rho^{-} \big) \Big(\frac{1}{\Lambda}\Big)^{\tilde{p}(1+\tau_{4})}  (\rho^{-}) ^{\tilde{p}(1+\tau_{4})}2\rho^{-} \mathds{1}_{B(0,R)\cap M_{2}^{c}}\,dx\\
&\lesssim  \Big(\frac{1}{\Lambda}\Big)^{\tilde{p}(1+\tau_{4})}\int_{\R^N}  \big(I_{\mu} \ast \tilde{U}^{\tilde{p}(1-\tau_{4})}|\rho| \big)   |\rho| ^{\tilde{p}(1+\tau_{4})+1} \,dx\\
&\lesssim \bigg(\frac{1}{ \|\nabla\rho\|_{L^{2}}^{1-\tau_{3}}}\bigg)^{\tilde{p}(1+\tau_{4})} \|\nabla\rho\|_{L^{2}}^{\tilde{p}(1+\tau_{4})+2}\approx \|\nabla\rho\|_{L^{2}}^{\tilde{p}(1+\tau_{4})\tau_{3}+2}\lesssim o(\|\nabla\rho\|_{L^{2}}^{2}),
\endaligned
\end{equation}
whenever $\tau_{4},\tau_{3}>0$. Combining \eqref{eq:4.35}, \eqref{eq:4.38}, \eqref{eq:4.39} and \eqref{eq:4.40}, we have that
\begin{equation}\label{eq:4.41}
    \aligned
    \mathcal{I}_{1,2}\lesssim\mathcal{I}_{1,2,1}+\mathcal{I}_{1,2,2}+\mathcal{I}_{1,2,3}\lesssim o(\|\nabla\rho\|_{L^{2}}^{2}), \quad  \tau_{3}\in(0,1/2), \tau_{4}>0.
    \endaligned
\end{equation}
F{}inally we show that by choosing appropriate parameters $\tau_{1}\sim\tau_{4}$, all the above estimates related to $ \mathcal{I}_{1,1}\sim\mathcal{I}_{1,3}$ will hold.  In fact,
f{}ix $0<\tau_{1}<<1$, such that $\tau_{1}(N-\frac{\mu}{2})<1/2 $, then for $\tau_{2}\ll 1\implies \tau_{1}(N-\frac{\mu}{2})>\tau_{1}(N-2)\frac{\tau_{2}}{1+\tau_{2}}$,   the set
$$\mathcal{O}\coloneqq\Big\{(\tau_{1},\tau_{2},\tau_{3}):\tau_{1}(N-\frac{\mu}{2})>\tau_{3}>\tau_{1}(N-2)\frac{\tau_{2}}{1+\tau_{2}}\Big\}$$
is nonempty. We f{}ix such a $\tau_{2}$ and take $\tau_{3}\in\mathcal{O}$, then $\tau_{3}\in(0,1/2)$. Finally  for  any
$\tau_{4}>0$,  it follows from \eqref{eq:4.22},
\eqref{eq:4.32}, \eqref{eq:4.33} and \eqref{eq:4.41} that
\begin{equation}\label{eq:4.42}
    \aligned
    \mathcal{I}_{1}\lesssim \mathcal{I}_{1,1}+\mathcal{I}_{1,2}+\mathcal{I}_{1,3}\lesssim o(\|\nabla\rho\|_{L^{2}}^{2}).
    \endaligned
\end{equation}
Therefore, by \eqref{eq:4.19}, \eqref{eq:4.42} and
\eqref{eq:4.21}, we have that
 $$
    \int_{\R^N} \big(I_{\mu} \ast \tilde{U}^{\tilde{p}}\rho \big) \tilde{U}^{\tilde{p}}\rho \,dx\lesssim \mathcal{I}_{1}+|\mathcal{I}_{2}|\lesssim o(\|\nabla\rho\|_{L^{2}}^{2}),
 $$
this  is \eqref{eq:4.17}-\eqref{eq:4.16}.

\vskip0.12in

Therefore,  by using \eqref{eq:4.13}, \eqref{eq:4.14}, \eqref{eq:4.15} and  \eqref{eq:4.16}, if $3\leq N\leq 5$ and $\frac{N+2}{2}<\mu<\min(N,4)$, then
\begin{equation}\label{eq:4.43}
    \mathcal{H}_{17}\lesssim \big(\varepsilon_{1}+\hat{\varepsilon}(1+C(\varepsilon_{1}))
+o(1)\big)\|\nabla\rho\|_{L^{2}}^{2}
 +\Theta(u)\|\nabla\rho\|_{L^{2}} .
\end{equation}
\vskip0.05in
  $\bullet$ Next we consider the term $\mathcal{H} _{18}$:
  \begin{equation}\label{eq:4.44}
  \aligned\  \mathcal{H} _{18}=&\ \int_{\R^N}\Big(I_{\mu}\ast\sum_{i=1}^\nu\alpha_i|\alpha_i|^{\tilde{p}}
  \tilde{U}_i^{\tilde{p}+1}\Big)\Big(|u|^{\tilde{p}-1}u - \sum_{i=1}^\nu\alpha_i|\alpha_i|^{\tilde{p}-1}
  \tilde{U}_i^{\tilde{p}}   \Big)\rho\,dx\\
  =&\ \sum_{i=1}^\nu \alpha_i^{\tilde{p}+1} \int_{\R^N}
    \Big(I_{\mu}\ast\tilde{U}_i^{\tilde{p}+1}\Big)
    \Big(|u|^{\tilde{p}-1}u - \sum_{i=1}^\nu\alpha_i|\alpha_i|^{\tilde{p}-1}
    \tilde{U}_i^{\tilde{p}}-\tilde{p}\sigma^{\tilde{p}-1}\rho\Big)\rho\,dx\\
    & \ \ +\sum_{i=1}^\nu \alpha_i^{\tilde{p}+1}\tilde{p} \int_{\R^N}
    \Big(I_{\mu}\ast\tilde{U}_i^{\tilde{p}+1}\Big)
  \sigma^{\tilde{p}-1}\rho^2\,dx \\
  =& \ \sum_{i=1}^\nu \alpha_i^{\tilde{p}+1} \mathcal{H}_{18,1}^{(i)} +\sum_{i=1}^\nu \alpha_i^{\tilde{p}+1}\tilde{p}\,\mathcal{H}_{18,2}^{(i)}.\\
  \endaligned
  \end{equation}
  On the one hand, using relation \eqref{eq:2.3} and the inequality \eqref{eq:4.6}, we obtain
  \begin{equation}\label{eq:4.45}
  \aligned
     \mathcal{H}_{18,1}^{(i)} \ \approx&\ \int_{\R^N}
    \tilde{U}_i^{p-\tilde{p}}
    \Big(|u|^{\tilde{p}-1}u - \sum_{i=1}^\nu\alpha_i|\alpha_i|^{\tilde{p}-1}
    \tilde{U}_i^{\tilde{p}}-\tilde{p}\sigma^{\tilde{p}-1}\rho\Big)\rho\,dx\\
  \lesssim &\   \int_{\R^N}  \tilde{U}_i^{p-\tilde{p}}
    \Big(\{\sigma^{\tilde{p}-2}|\rho|^2\}_{\tilde{p}>2} + |\rho|^{\tilde{p}} +
    \sum_{ k\not= l}
   \tilde{U}_k^{\tilde{p}-1}\tilde{U}_l \Big)
   |\rho| \,dx,
   \endaligned
  \end{equation}
where
\begin{equation}\label{eq:4.46}
\aligned
\Big\{\int_{\R^N}\tilde{U}_i^{p-\tilde{p}}\sigma^{\tilde{p}-2}|\rho|^3 \,dx\Big\}_{\tilde{p}>2}
&\leq   \Big\{\|\tilde{U}_i
\|_{L^{p+1}}^{p-\tilde{p}}
 \|\sigma\|_{L^{p+1}}^{\tilde{p}-2}
 \|\rho\|_{L^{p+1}}^3\Big\}_{\mu<6-N}\lesssim \big\{\|\nabla\rho\|_{L^{p+1}}^3\big\}_{\mu<6-N},\\
\int_{\R^N}  \tilde{U}_i^{p-\tilde{p}}|\rho|^{\tilde{p}+1}\,dx&\leq \|\tilde{U}_i
\|_{L^{p+1}}^{p-\tilde{p}}
 \|\rho\|_{L^{p+1}}^{\tilde{p}+1}\lesssim \|\nabla\rho\|_{L^{p+1}}^{\tilde{p}+1}, \  \mu<N+2.
\endaligned
\end{equation}
For $k\neq l$, using Lemma \ref{prop:interaction_approx}, the H\"older and Sobolev inequalities and Lemma \ref{prop:interaction_and_coef}, we get  that

\noindent $\bullet$ if $i=k$, then $$\aligned \int_{\R^N}  \tilde{U}_i^{p-\tilde{p}}\tilde{U}_k^{\tilde{p}-1}\tilde{U}_l|\rho| \,dx
 &= \int_{\R^N}  \tilde{U}_k^{p-1}\tilde{U}_l|\rho| \,dx\leq  \Big(\int_{\R^N}\tilde{U}_k^{\frac{
(p-1)(p+1)}{p}}
\tilde{U}_l^{\frac{p+1}{p}}\,dx \Big)^{\frac{p}{p+1}}
\|\rho\|_{L^{p+1}}\\
&\approx Q^{\frac{N-2}{2}}\| \nabla \rho\|_{L^{2}}\approx \Big(\int_{\R^N}\tilde{U}_k^p
\tilde{U}_l\,dx\Big)\|\nabla \rho\|_{L^{2}} \\
&\lesssim  \Big(\hat  \varepsilon\| \nabla \rho\|_{L^{2}} +\| \nabla \rho\|_{L^{2}}^2
+\Theta(u) \Big)\| \nabla \rho\|_{L^{2}};\endaligned$$

\noindent $\bullet$ if $i=l$, then
 $$\aligned \int_{\R^N}  \tilde{U}_i^{p-\tilde{p}}\tilde{U}_k^{\tilde{p}-1}\tilde{U}_l|\rho| \,dx
 &= \int_{\R^N}  \tilde{U}_k^{\tilde{p}-1}\tilde{U}_l^{p-\tilde{p}+1}|\rho| \,dx\lesssim   \|\tilde{U}_k^{\tilde{p}-1}\tilde{U}_l^{p-\tilde{p}+1}\|_{L^{\frac{p+1}{p}}}\|\nabla \rho\|_{L^{2}}.
\endaligned$$
We next use Lemma \ref{prop:interaction_approx} to discuss $$ \|\tilde{U}_k^{\tilde{p}-1}\tilde{U}_l^{p-\tilde{p}+1}\|_{L^{\frac{p+1}{p}}} = \Big(\int_{\R^N}\tilde{U}_k^{\frac{
(\tilde{p}-1)(p+1)}{p}}
\tilde{U}_l^{\frac{(p-\tilde{p}+1)(p+1)}{p}} \,dx\Big)^{\frac{p}{p+1}}.$$
Recall that we have imposed  the constrains  $ 3\leq N\leq5$ and $\frac{N+2}{2}< \mu<\min(N,4)$. Notice $\frac{N+2}{2}>\frac{6-N}{2}$ and $\mu>\frac{6-N}{2}$ which is equivalent to $\frac{(\tilde{p}-1)(p+1)}{p}<\frac{(p-\tilde{p}+1)(p+1)}{p}$. By Lemma \ref{prop:interaction_approx} and Lemma \ref{prop:interaction_and_coef}, we have that
 $$
   \aligned  &\  \|\tilde{U}_k^{\tilde{p}-1}\tilde{U}_l^{p-\tilde{p}+1}\|_{L^{\frac{p+1}{p}}}  \approx\ \Big(Q^{\frac{N-2}{2}\frac{(\tilde{p}-1)(p+1)}{p}}\Big)^{\frac{p}{p+1}}
 \approx  \Big(\int_{\R^N}\tilde{U}_k^p\tilde{U}_l\,dx\Big)^{\tilde{p}-1}\\
 &  \lesssim    \Big(\hat  \varepsilon\| \nabla \rho\|_{L^{2}}
   +\| \nabla \rho\|_{L^{2}}^{\min(\tilde{p}, 2)} +\Theta(u) \Big)^{\tilde{p}-1} \lesssim  \hat  \varepsilon\| \nabla \rho\|_{L^{2}}
   +o(1)\| \nabla \rho\|_{L^{2}} +\Theta(u),
   \endaligned
 $$
where the last inequality holds  whenever   $\tilde{p}-1\geq 1$ and $\tilde{p}>1$, namely $ \mu\leq6-N$. Note that $6-N\leq \min(N,4)$. However,  only when $N=3$, there holds $ \frac{N+2}{2}<6-N$, i.e.,  the set $(0,6-N]\cap(\frac{N+2}{2},\min(N,4))$ is nonempty. Therefore, if $N=3$ and $\frac{5}{2}<\mu<3$, we can get  the following expected estimate:
$$ \int_{\R^N}  \tilde{U}_i^{p-\tilde{p}}\tilde{U}_k^{\tilde{p}-1}\tilde{U}_l|\rho| \,dx \lesssim\hat  \varepsilon\| \nabla \rho\|_{L^{2}}^{2} +o(1)\| \nabla \rho\|_{L^{2}}^{2} +\Theta(u)\| \nabla \rho\|_{L^{2}};$$

\noindent $\bullet$ further, if $i\neq k$ and $i\neq l$, then
$$\aligned \int_{\R^N}  \tilde{U}_i^{p-\tilde{p}}\tilde{U}_k^{\tilde{p}-1}\tilde{U}_l|\rho| \,dx
 &\leq  \|  \tilde{U}_i^{p-\tilde{p}}\tilde{U}_k^{\tilde{p}-1} \|_{L^{\frac{p+1}{p-1}}}
\| \tilde{U}_l \|_{L^{p+1}}
\|\nabla\rho\|_{L^{2}}.\endaligned$$
Since we have imposed the constrains $N=3$ and $\frac{5}{2}<\mu<3$, which implies that   $\mu >2 \iff p-\tilde{p}>\tilde{p}-1 $. Therefore, by Lemma \ref{prop:interaction_approx} and Lemma \ref{prop:interaction_and_coef}, we have that
 $$
    \aligned   &  \|  \tilde{U}_i^{p-\tilde{p}}\tilde{U}_k^{\tilde{p}-1} \|_{L^{\frac{p+1}{p-1}}}  \approx \ \Big( Q^{\frac{N-2}{2}(\tilde{p}-1)\frac{p+1}{p-1} } \Big)^{\frac{p-1}{p+1}} \approx  \Big(\int_{\R^N}\tilde{U}_i^p\tilde{U}_k\,dx\Big)^{\tilde{p}-1}\\
  & \lesssim       \Big(\hat  \varepsilon\| \nabla \rho\|_{L^{2}}
   +\| \nabla \rho\|_{L^{2}}^{\min(\tilde{p}, 2)} +\Theta(u) \Big)^{\tilde{p}-1} \lesssim  \hat  \varepsilon\| \nabla \rho\|_{L^{2}}
   +o(1)\| \nabla \rho\|_{L^{2}} +\Theta(u),
    \endaligned
 $$
where the last inequality holds  since  $ \mu<3$ implies  that $\tilde{p}-1\geq 1$ and $\tilde{p}>1$. Therefore,  if $N=3$ and $\frac{5}{2}<\mu<3$, we can also  get
$$ \int_{\R^N}  \tilde{U}_i^{p-\tilde{p}}\tilde{U}_k^{\tilde{p}-1}\tilde{U}_l|\rho| \,dx \lesssim\hat  \varepsilon\| \nabla \rho\|_{L^{2}}^{2} +o(1)\| \nabla \rho\|_{L^{2}}^{2} +\Theta(u)\| \nabla \rho\|_{L^{2}}. $$
Summing up the above  three cases, if $N=3$ and $5/2<\mu<3$, we  get  that
\begin{equation}\label{eq:4.47}
\aligned
    \int_{\R^N}  \tilde{U}_i^{p-\tilde{p}}\tilde{U}_k^{\tilde{p}-1}\tilde{U}_l|\rho| \,dx \lesssim  \hat{\varepsilon}\| \nabla \rho\|_{L^{2}}^{2}+o(1)\| \nabla \rho\|_{L^{2}}^{2}+\Theta(u)\| \nabla \rho\|_{L^{2}}.
    \endaligned
\end{equation}
It follows by \eqref{eq:4.45}, \eqref{eq:4.46}  and \eqref{eq:4.47} that
\begin{equation}\label{eq:4.48}
\aligned
    \mathcal{H}_{18,1}^{(i)} \lesssim  \hat{\varepsilon}\| \nabla \rho\|_{L^{2}}^{2}+o(1)\| \nabla \rho\|_{L^{2}}^{2}+\Theta(u)\| \nabla \rho\|_{L^{2}}.
    \endaligned
\end{equation}
On the other hand, using relation \eqref{eq:2.3},  recalling $\sigma=\sum\limits_{i=1}^\nu \alpha_i\tilde{U}_i$ and applying Lemma \ref{prop:interaction_approx}, the
  H\"older and Sobolev inequalities, we obtain that
  \begin{equation}\label{eq:4.49}
  \aligned  \mathcal{H}_{18,2}^{(i)}   &= \int_{\R^N}\Big(I_{\mu}\ast\tilde{U}_i^{\tilde{p}+1}\Big)
  \sigma^{\tilde{p}-1}\rho^2 \,dx
   \approx \int_{\R^N}\tilde{U}_i^{p-\tilde{p}}\Big( \alpha_i \tilde{U}_i +\sum\limits_{j\neq i} \alpha_j \tilde{U}_j\Big)^{\tilde{p}-1}\rho^2 \,dx\\
& \lesssim \alpha_i^{\tilde{p}-1}
\int_{\R^N}\tilde{U}_i^{p-1}\rho^2 \,dx+\sum\limits_{j\neq i} \alpha_j^{\tilde{p}-1}
\int_{\R^N}\tilde{U}_i^{p-\tilde{p}}
\tilde{U}_j^{\tilde{p}-1}\rho^2 \,dx\\
 & \lesssim  \int_{\R^N}\Big(I_{\mu}\ast\tilde{U}_i^{\tilde{p}+1}\Big) \tilde{U}_i^{\tilde{p}-1}\rho^2 \,dx
  +\Big(\int_{\R^N}\tilde{U}_i^{
\frac{(p-\tilde{p})(p+1)}{p-1}}  \tilde{U}_j^{\frac{(\tilde{p}-1)(p+1)}{p-1}} \Big)^{\frac{p-1}{p+1}}\| \nabla\rho\|_{L^2}^2\\
 &\lesssim o(1) \|\nabla\rho\|_{L^2}^2,\endaligned
 \end{equation}
 where the last inequality holds  since  $\mu<4 (\iff \tilde{p}-1>0)$.  Therefore,
 $$
 \int_{\R^N}\tilde{U}_i^{
\frac{(p-\tilde{p})(p+1)}{p-1}}  \tilde{U}_j^{\frac{(\tilde{p}-1)(p+1)}{p-1}}\approx \begin{cases}
     Q^{\frac{N-2}{2}\min(\frac{(p-\tilde{p})(p+1)}{p-1},\frac{(\tilde{p}-1)(p+1)}{p-1})}, &\text{if } p-\tilde{p}\neq \tilde{p}-1\\
     Q^{\frac{N}{2}} \log(\frac{1}{Q}), &\text{if } p-\tilde{p}=\tilde{p}-1
     \end{cases}
     \ \approx \  o(1).
 $$
By using a similar process  as in the proof of  \eqref{eq:4.16},  we have that  $$\int_{\R^N}\Big(I_{\mu}\ast\tilde{U}_i^{\tilde{p}+1}\Big) \tilde{U}_i^{\tilde{p}-1}\rho^2 \,dx\lesssim  o(\|\nabla\rho\|_{L^2}^2).$$
 Therefore, combining \eqref{eq:4.44}, \eqref{eq:4.48}  and \eqref{eq:4.49}, we obtain that if $N=3$ and $5/2<\mu<3$, then
 \begin{equation}\label{eq:4.50}
 \mathcal{H} _{18}\lesssim \hat\varepsilon
 \| \nabla \rho\|_{L^{2}}^{2}
  +o(1)\| \nabla \rho\|_{L^{2}}^2 + \Theta(u)\| \nabla \rho\|_{L^{2}}.
 \end{equation}
 \vskip0.08in
  With these inequalities on hand, now we proof \eqref{eq:1.18} and \eqref{eq:1.19}.
  Firstly, it follows from  \eqref{eq:4.4}, \eqref{eq:4.7},  \eqref{eq:4.11}, \eqref{eq:4.12}, \eqref{eq:4.43} and $\eqref{eq:4.50}$ that, if $N=3$ and $5/2<\mu<3$,
  $$ \|\nabla\rho\|_{L^{2}}^{2} \lesssim \mathcal{H} _{1}+\cdots+\mathcal{H} _{19}+\Theta(u)\|\nabla\rho\|_{L^{2}}\lesssim \big(\varepsilon_{1}+\hat{\varepsilon}(4+C(\varepsilon_{1}))
+o(1)\big)\|\nabla\rho\|_{L^{2}}^{2}
 +\Theta(u)\|\nabla\rho\|_{L^{2}},$$
 i.e., there exists a constant $\bar C(N,\mu,\nu)>0$, such that
 $$ \|\nabla\rho\|_{L^{2}}^{2} \leq \bar C(N,\mu,\nu) \big(\varepsilon_{1}+\hat{\varepsilon}(4+C(\varepsilon_{1}))
+o(1)\big)\|\nabla\rho\|_{L^{2}}^{2}
 +\bar C(N,\mu,\nu)\Theta(u)\|\nabla\rho\|_{L^{2}}.$$
We now f{}ix $\varepsilon_{1}\ll1$ such that $\bar C(N,\mu,\nu)\varepsilon_{1}\leq 1/2$, then $C(\varepsilon_{1})$ is a f{}ixed constant and we get that
  $$ \|\nabla\rho\|_{L^{2}}^{2} \leq  2\bar C(N,\mu,\nu)\big(\hat{\varepsilon}(4+C(\varepsilon_{1}))
+o(1)\big)\|\nabla\rho\|_{L^{2}}^{2}
 +2\bar C(N,\mu,\nu)\Theta(u)\|\nabla\rho\|_{L^{2}}.$$
Then  we take a f{}ixed $\hat \varepsilon\ll 1$ such that $2\bar C(N,\mu,\nu)\hat{\varepsilon}(4+C(\varepsilon_{1}))\leq 1/2$. We obtain that
  $$ \|\nabla\rho\|_{L^{2}}^{2} \leq  4\bar C(N,\mu,\nu)o(1)\|\nabla\rho\|_{L^{2}}^{2}
 +4\bar C(N,\mu,\nu)\Theta(u)\|\nabla\rho\|_{L^{2}},$$
By Lemma \ref{prop:interaction_and_coef}, there exists $\delta_{0}=\delta_{0}(N,\mu,\nu,\hat{\varepsilon})$ such that the estimates \eqref{eq:3.11} and \eqref{eq:3.12} are applicable.  Then  we may take $\delta\leq \delta_{0}$ suf{}f{}iciently  small to guarantee that $4\bar C(N,\mu,\nu)o(1)\leq 1/2$ and  $\|\nabla\rho\|_{L^{2}}^{2}$ is suf{}f{}iciently small. Therefore, we get that
  \begin{equation}\label{eq:4.51}
     \|\nabla\rho\|_{L^{2}} \leq C(N,\mu,\nu)\Theta(u),
  \end{equation}
  where $C(N,\mu,\nu)\coloneqq 8 \bar C(N,\mu,\nu) >0$.
  Lastly, going back to \eqref{eq:3.11} and \eqref{eq:3.12}, if $N=3$ and $5/2<\mu<3$, by using \eqref{eq:4.51} and  the assumption $\Theta(u)\leq 1$ in \eqref{eq:4.1}, we have that
  \begin{equation}\label{eq:4.52}
      |\alpha_i-1|\lesssim \hat\varepsilon\|\nabla\rho\|_{L^2}  +\|\nabla \rho\|_{L^{2}}^{2 }
          + \Theta(u)\lesssim \Theta(u),
  \end{equation}
    and that
$$\int_{\R^N}
 \tilde{U}_i^{\tilde{p}}\tilde{U}_j\,dx \lesssim \hat\varepsilon\|\nabla\rho\|_{L^2}  +\|\nabla \rho\|_{L^{2}}^{2}+\Theta(u)\lesssim \Theta(u),
 $$
 that is \eqref{eq:1.19} of Theorem \ref{thm:main_close}.
 Further, we let $\sigma^{\prime}\coloneqq\sum\limits_{i=1}^{\nu}\tilde{U}_{i}$, then it follows from \eqref{eq:4.51} and \eqref{eq:4.52} that
 $$
 \aligned
 \Big\|\nabla u-\sum_{i=1}^{\nu}\nabla \tilde{U}_i\Big\|_{L^2}&=  \|\nabla u-\nabla\sigma ^{\prime}\|_{L^2}=\|\nabla u-\nabla\sigma \|_{L^2}+ \|\nabla\sigma-\nabla \sigma^{\prime} \|_{L^2}\\
 &\lesssim \|\nabla \rho \|_{L^2}+\sum_{i=1}^{\nu}|\alpha_i-1|\lesssim \Theta(u),
 \endaligned
 $$
 then we obtain \eqref{eq:1.18} of Theorem \ref{thm:main_close}. We f{}inally f{}inish the proof of Theorem \ref{thm:main_close}. $\Box$

 \end{document}